\documentclass[11pt]{article}
\usepackage[english]{babel}
\usepackage{newtxtext}
\usepackage{amsmath}
\usepackage{amssymb}
\usepackage{amsfonts}
\usepackage{amsthm}
\usepackage{ dsfont }
\usepackage{authblk}
\usepackage{enumitem}
\usepackage{graphicx}
\usepackage[numbers]{natbib}
\usepackage{tikz-cd}
\usetikzlibrary{positioning}
\usetikzlibrary{shapes,snakes}
\usetikzlibrary{arrows,automata,positioning}
\usepackage{multicol}
\usepackage{float}
\usepackage{ stmaryrd }
\usepackage{ wasysym }
\usepackage{caption}
\setcitestyle{square}
\usepackage[colorlinks=true, allcolors=black]{hyperref}
\usepackage[a4paper, total={6in, 8in}]{geometry}
\usepackage{tikz}
\usepackage{multirow}
\usepackage{mathtools}
\usetikzlibrary{arrows,automata,positioning}
\usepackage[T1]{fontenc}
\usepackage[utf8]{inputenc}
\usepackage{babel}
\usepackage{graphicx}
\usepackage{ebproof}
\usepackage{prooft}
\usepackage{ mathrsfs }
\usepackage{color}
\usepackage{thmtools}
\theoremstyle{plain}
\newtheorem{theorem}{Theorem}[section]

\newtheorem{corollary}[theorem]{Corollary}
\newtheorem{proposition}[theorem]{Proposition}
\newtheorem{lemma}[theorem]{Lemma}
\theoremstyle{definition}
\newtheorem{definition}[theorem]{Definition}

\theoremstyle{remark}
\newtheorem{remark}[theorem]{Remark}
\usepackage{fancyhdr}
\usepackage[explicit,compact]{titlesec}
\titleformat{\chapter}[block]
    {\bfseries\huge}{\filright\huge\thechapter.}{1ex}{\huge\filright #1}
\usepackage{txfonts}
\newcommand{\RN}[1]{%
  \textup{\uppercase\expandafter{\romannumeral#1}}%
}

\newcommand{\dmnd}[2]{\langle  \scalebox{0.85}{$#1$} \rangle_{#2}~}
\newcommand{\idmnd}[3]{\langle  \scalebox{0.85}{$#1$} \rangle_{#2}^{#3}~}
\newcommand{\dwn}[0]{{\downarrow~}}
\newcommand{\W}[0]{\mathbf{W}_\mathsf{BI}}

\newcommand{\bio}[0]{{\mathsf{BI_\omega}}}

\newcommand{\rulerow}[1]{%

  \par\noindent\hbox to \textwidth{\hfil #1\hfil}\par\vspace{1.5em}%

}

\title{A Reflection Calculus for Bar Induction and the Linearity of Consistency Strength Ordering}

\author[1,2]{Fedor Pakhomov}
\author[1]{Vlad Lazar}

\affil[1]{Department of Mathematics: Analysis, Logic and Discrete Mathematics\\
Ghent University, Krijgslaan 297 S8, 9000 Ghent, Belgium}
\affil[2]{Steklov Mathematical Institute of the Russian Academy of Sciences\\
Ulitsa Gubkina 8, Moscow 117966, Russia}

\renewcommand\thanks[1]{\footnotetext{#1}} 

\date{\small October 2026}

\begin{document}
\maketitle

\begin{abstract}
We introduce a novel extension of reflection calculus, $\mathsf{RC_{BI}}$, to analyze a family of extensions of $\mathsf{ACA}_0$ generated by uniform $\omega$-model and syntactic reflection principles, together with several other related operators. The union of these theories coincides with the theory of Bar Induction ($\mathsf{BI}$). We prove that $\mathsf{RC_{BI}}$ is arithmetically sound, and that it is sound and complete with respect to a natural class of Kripke models. Leveraging these modal semantics, we demonstrate that the class of theories denoted by $\nabla$-free closed formulas is a pre-wellordering with respect to $\Pi^1_1$-soundness and consistency strength. This allows us to establish an ordinal analysis for $\mathsf{BI}$ and its fragments, analogous to Beklemishev's analysis of $\mathsf{PA}$. The structural properties of $\mathsf{RC_{BI}}$ naturally give rise to a novel ordinal notation system for the Bachmann-Howard ordinal which differs substantially from existing representations.
\end{abstract}


\vspace{\baselineskip}
\section{Introduction}
Reflection principles play a fundamental role in proof theory, serving both as tools for ordinal analysis and as subjects of independent study. This paper focuses on extensions of $\mathsf{ACA}_0$ obtained by iterating two common types of reflection principles: the syntactic reflection principle (which asserts that everything provable is true) and the semantic reflection principle (which asserts that everything true in the universe holds in some countably coded $\omega$-model). As shown by Harvey Friedman \cite{fried}, taking the union of all $\Pi^1_n$-restricted reflection principles over $\mathsf{ACA}_0$ yields the theory of Bar Induction ($\mathsf{BI}$). This characterization was subsequently refined by Gerhard Jäger and Thomas Strahm \cite{jager1999bar}, who established a precise level-by-level correspondence.

Traditionally, ordinal analysis of such theories has relied on the study of cut-free proofs. However, Lev Beklemishev \cite{beklemishev2004provability} introduced a radically different approach: carrying out proof-theoretic analysis by studying families of fragments generated through finite combinations of reflection principles. For first-order Peano Arithmetic ($\mathsf{PA}$) \cite{beklemishev2012calibrating} \cite{beklemishev2022reflection}, the modern presentation of this approach utilizes the modal logic defined as the reflection calculus $(\mathsf{RC})$.

The standard approach to the utilization of modal logic within the context of proof theory involves interpreting modal formulas $\Box\varphi$ as formal provability statements $\text{Prv}_T(\varphi)$ over a base theory $T$. This interpretation is captured by Gödel-Löb ($\mathsf{GL}$) (defined as $\mathsf{K}$ + Löb's axiom: $\Box(\Box A \to A) \to \Box A$) which has been proven to be arithmetically complete with respect to \textsf{PA} \cite{solovay1976provability}. Japaridze's polymodal logic ($\mathsf{GLP}$) \cite{GLP} is an extension of $\mathsf{GL}$ which has an infinite sequence of modalities $\Box_i$ which correspond to increasingly strong provability predicates. This system has yielded numerous significant results in proof theory, such as Beklemishev's ordinal analysis of Peano Arithmetic \cite{beklemishev2004provability}. However, the full logic remains notoriously difficult to work with due to its lack of traditional Kripke semantics and high computational complexity. Consequently, much recent work has focused on well-behaved fragments of the system, such as the strictly positive fragment \textsf{RC}.

We define a formula as (strictly) positive if it is constructed from atomic propositions and the truth constant $\top$, using only the conjunction operator ($\wedge$) and diamond modalities ($\Diamond_i$). The systematic study of the positive fragment of $\mathsf{GLP}$ was initiated by Dashkov, who proved that it is Kripke complete and polynomial time decidable \cite{dashkov2012positive}. $\mathsf{RC}$ was later formally introduced by Beklemishev \cite{beklemishev2012calibrating} utilizing an axiomatization that is equivalent to Dashkov's, but oriented toward a sequent-style calculus.

One of the main advantages of $\mathsf{RC}$ regarding its applicability to the analysis of arithmetic comes from interpreting its formulas as c.e. extensions of $\mathsf{PA}$, with modalities $\Diamond_n$ being interpreted as the uniform reflection principle $\Pi_{n+1}\textsf{-RFN}$. $\mathsf{RC}$ provides a sound and complete characterization of the inclusions between these fragments of $\mathsf{PA}$ and naturally serves as an ordinal notation system for $\varepsilon_0$, the proof-theoretic ordinal of $\mathsf{PA}$.

In this paper we will do a somewhat analogous treatment for the much stronger theory of Bar Induction $\mathsf{BI}$. The theory $\mathsf{BI}$ is a subtheory of classical second-order arithmetic that extends $\mathsf{ACA}_0$ by the full scheme of transfinite induction over arbitrary well-orders.

We develop a variant of $\mathsf{RC}$, denoted $\mathsf{RC_{BI}}$, tailored for the analysis of $\mathsf{BI}$. The formulas of $\mathsf{RC_{BI}}$ are generated by the following grammar: 
\[ A := \top \mid p \mid A\wedge B \mid \dmnd{1}{n} A \text{ for $n\ge 1$} \mid \dmnd{2}{n} A \text{ for $n\ge 2$} \mid {\downarrow~} A \mid \nabla A. \]

In the arithmetical interpretation of $\mathsf{RC_{BI}}$, we assign to the formulas c.e.\ extensions of $\mathsf{ACA}_0$. For the purpose of our analysis, it suffices to work with the set $\W$ of $\nabla$-free closed formulas of $\mathsf{RC_{BI}}$. We assign to $\top$ the base theory $\mathsf{ACA}_0$, to $A\land B$ the deductive closure of the union of the interpretations of $A$ and $B$, and we define the arithmetical interpretations of the modal connectives as follows:
\begin{itemize}
    \item $\dmnd{1}{n} T$ is $\mathsf{ACA}_0$ extended by the uniform (syntactic) $\Pi^1_n$-reflection principle over $T$: 
    \[ \Pi^1_n\textsf{-RFN}(T) : \forall \varphi \in \Pi^1_n \, (\text{Prv}_T(\varphi) \rightarrow \text{Tr}_{\Pi^1_n}(\varphi)). \]
    
    \item $\dmnd{2}{n} T$ is $\mathsf{ACA}_0$ extended by the uniform (semantic) $\omega$-model $\Pi^1_n$-reflection principle over $T$:  
    \[ \Pi^1_n\text{-}\omega\textsf{RFN}(T) : \forall \varphi(X) \in \Pi^1_n \, \forall P \, (\forall M(P\dot \in M \land M \models T \rightarrow M\models \varphi(P)) \rightarrow \text{Tr}_{\Pi^1_n}(\varphi(P))). \]
    
    \item $\dwn T$ is interpreted as the $\omega$-model projection operator applied to $T$: 
    \[ {\dwn} T : \{\varphi \mid T\vdash \forall M(M\models \mathsf{ACA}_0 \rightarrow M\models \varphi)\}. \]
    
    \item $\nabla T$ is $\mathsf{ACA}_0$ extended by the set of true $\Pi^1_1$ theorems of $T$: 
    \[ \Pi^1_1(T) : \{\varphi\in\Pi^1_1 \mid T\vdash \varphi\}. \]
\end{itemize}
The structures $M$ above range over $\omega$-models, which are structures in $L_2$ whose first-order part is the standard model $\mathds{N}$ of $\mathsf{PA}$. 

The language of the calculus allows one to talk about a large variety of fragments of $\mathsf{BI}$. As proved by Harvey Friedman \cite{fried}:
\[\mathsf{BI} \equiv \bigcup\limits_{n\in\omega}\Pi^1_n\text{-}\omega\textsf{RFN}(\mathsf{ACA}_0).\]
Thus the closed formulas of our calculus correspond to a family of subtheories of $\mathsf{BI}$, whose union is $\mathsf{BI}$ itself. Also the family contains a lot of natural theories (see Section \ref{section ONS} for more details): \begin{itemize} \item $\top$ is $\mathsf{ACA}_0$, \item $\dmnd{1}{2}\top $ is $\mathsf{ACA}_0'$, \item $\dmnd{2}{2} \top$ is $\mathsf{ACA}_0^+$, \item $\downarrow \top$ is $\mathsf{ACA}$, \item $\dmnd{2}{3}\top$ is $\Sigma^1_1\textsf{-DC}_0$, \item $\dmnd{2}{2}\dmnd{2}{3}\top$ is $\mathsf{ATR}_0$, \item $\dmnd{2}{n+3}\top$ is $\Pi^1_{n+1}\textsf{-BI}_0$.\end{itemize}

One of the main contributions of this paper is the proof that the set $\W$ of $\nabla$-free closed formulas of $\mathsf{RC_{BI}}$ forms a linear, well-founded pre-order under the $\Pi^1_1$-reflection strength ordering. The strict part of this relation is defined as:
\[ A <^1_1 B \stackrel{\text{def}}{\iff} B \vdash \dmnd{1}{1} A, \]
and the associated non-strict relation as $A\leq^1_1 B$ iff $A<^1_1 B$ or $\nabla A\dashv\vdash\nabla B$ (see Definition \ref{def order 1}).

\begin{theorem}\label{thm_lin_int}
    The relation $\leq^1_1$ is a linear, well-founded pre-order on $\W$; in particular, the quotient $\W/{\equiv^1_1}$ is well-ordered by $<^1_1$.
\end{theorem}

At the level of the ordinal notation system, this ensures a proper total ordering on ordinal denotations. Conceptually, it implies that the fragments of $\mathsf{BI}$ denoted in our system form a well-ordered hierarchy with respect to both $\Pi^1_1$-reflection strength and consistency strength, extending previous results on iterated reflection in $L_2$ due to Walsh and Pakhomov \cite{pakhomov2021reflection} \cite{pakhomov2023reducing}.\\

On the technical level we also establish reduction properties (conservativity properties) between theories within the targeted family of fragments, which are a central requirement of Beklemishev's style of analysis. The reduction properties for our logic are as follows:

\begin{lemma}\label{lem 1 red 1}
    Given $T\subseteq \Pi^1_{n+1}$, 
    \[ T+\dmnd{1}{n+1} A \equiv_{\Pi^1_n} T + \frac{\varphi\in\Pi^1_n}{\dmnd{1}{n}(A+\varphi)}. \]
\end{lemma}

\begin{lemma}\label{lem 2 red 2}
    Given $T\subseteq \Pi^1_{n+1}$, 
    \[ T+\dmnd{2}{n+1} A \equiv_{\Pi^1_n} T + \frac{\varphi\in\Pi^1_n}{\dmnd{2}{n}(A+\varphi)}. \]
\end{lemma}

\begin{lemma}\label{lem 3 red 3}\label{lem:red_rule_dar_bas}
    \[ T+ {\dwn} \mathsf{ACA}_0 \equiv T + \left\{ \frac{\varphi}{\dmnd{1}{n} \varphi} \;\middle|\; n\in \mathbb{N} \right\}. \]
\end{lemma}

\begin{lemma}\label{lem 4 red 4}\label{lem:red_rule_dar_ref}
    \[ T+ {\dwn}\dmnd{1}{1} A \equiv T + \left\{ \frac{\varphi}{\dmnd{1}{n}({\dwn} A+\varphi)} \;\middle|\; n\in \mathbb{N} \right\}. \]
\end{lemma}

\begin{lemma}\label{lem 5 red 5}
    Given a finitely axiomatizable theory $U$ and a $\Pi^1_2$-axiomatizable theory $T$,
    \[ T+ \dmnd{2}{2} U \equiv_{\Pi^1_1} T + \left\{ \frac{\varphi\in \Pi^1_1}{\dmnd{1}{1}(U+{\dwn}(\mathsf{ACA}_0+\varphi))} \right\}. \]
\end{lemma}

By utilizing all the above, we obtain an ordinal analysis for $\mathsf{BI}$ in the style of Beklemishev. Specifically, the theory occupying position $\alpha$ in our hierarchy has the proof-theoretic ordinal $\varepsilon_\alpha$. This entails that $(\W,<^1_1)$ is an ordinal notation system for the proof-theoretic ordinal of $\mathsf{BI}$; it was known that the proof-theoretic ordinal of $\mathsf{BI}$ is the Bachmann-Howard ordinal (see a footnote at the end of Section \ref{section ONS}). It appears that our ordinal notation system for the Bachmann-Howard ordinal is substantially different from the existing ones based on collapsing functions, extended Veblen hierarchy, or higher order primitive recursion on ordinals. Moreover, as for Beklemishev's analysis of $\mathsf{PA}$, our result entails the ordinal analysis not just for $\mathsf{BI}$ itself, but for a rich family of its subtheories.
\vspace{\baselineskip}
\subsection{On the problem of linearity of consistency strength ordering}
On top of the results related to the program of ordinal analysis and the study of reflection principles, we believe that the results of this paper shed some light on the question of linearity of consistency strength ordering for the natural theories. The above, rather than being a formal mathematical question, is an empirical observation. While in general it is known that the consistency strength ordering has all the ``bad'' properties\footnote{There are, for example, infinite descending chains in this order as well as non-linearities in it. The non-linearity of the closely related interpretability order is known at least since the work of R.~Montague \cite{montague1962theories}, while for the consistency strength ordering it is an easy consequence of Solovay's Theorem on Arithmetical Completeness of the provability logic $\mathsf{GL}$ \cite{solovay1976provability}. The ill-foundedness of the ordering seems to be independently proved by Friedman, Smorynski, and Solovay \cite{smorynski2012self,lindstrom2017aspects}. We also refer to a recent publication of Hamkins with examples of this nature \cite{hamkins2025nonlinearity}.}, it is known that when attention is restricted to \emph{natural} theories, the ordering does appear to behave as a pre-well-ordering. This property of linearity of the ordering for natural theories was mentioned by Friedman \cite{friedman98letter} and discussed more in depth by Steel \cite{steel2014godel}. Some more philosophically oriented discussions of this phenomenon can be found in \cite{walsh2026hierarchy,grotenhuis2022natural,meadows2025consistency}.  

There are arguments explaining the nice structure of consistency strength ordering that are based on the observations that some closely related orderings do have such a nice structure. Namely, Walsh and the first author of the present paper \cite{pakhomov2021reflection} showed well-foundedness of $\Pi^1_1$-reflection strength ordering on $\Pi^1_1$-sound theories. Meadows observed \cite{meadows2025consistency} that the order on $\beta$-verified $\beta$-consistency strengths of $\beta$-consistent theories is a pre-well-order, where $\beta$-consistency of a theory $T$ is simply defined as the existence of a transitive model of $T$, and $T\models_\beta\varphi$ is defined as $\varphi$ being true in all transitive models of $T$ (see also \cite{jeon2025ranking}). However, those observations do not directly entail ``good'' structure of the consistency strength ordering: In general, the consistency strength ordering can have more comparisons than the finer $\Pi^1_1$-reflection ordering, and segments of the former can have infinite descending chains where the corresponding segments of the latter are well-founded. And the ordering of Meadows is, in general, not directly connected to consistency strength ordering (being neither finer nor coarser, as can be seen from easy-to-construct examples). 

The kind of argument we are presenting in the present paper is the following. Instead of looking at very broad classes of theories given in terms of certain soundness conditions and thus approximating the class of natural theories from above, we look at a much smaller class of theories that simply consist of theories admitting denotations in our reflection calculus.  However, with this smaller class of theories our machinery enables us to prove pre-well-orderedness of the actual consistency strength ordering. Thus the question becomes how exhaustive is the class of the theories that we consider with respect to the whole class of natural theories; or, more precisely, since quite obviously we only consider the theories within the interval $[\mathsf{ACA}_0,\mathsf{BI}]$, how exhaustive is our class for the natural theories from the same interval. The answer to the latter question obviously very much depends on how one gauges the intuitive, informal, and at least to some extent subjective notion of naturality of theories. In the present paper we chiefly focus on investigation within the boundaries of mathematical methodology and  will abstain from a deeper discussion of this evidently interesting question.

One of the motivating considerations behind the present paper is a straightforward conceptual argument: If a system of fragments of a given $\Pi^1_1$-sound theory corresponds to a sufficiently well-behaved reflection calculus, then this family of fragments is indeed well-ordered with respect to consistency strength. This principle is captured by the following theorem:

\begin{theorem}\label{abstract_linearity}
    Suppose $\mathcal{T}$ is a set of c.e.\ $\Pi^1_1$-sound extensions of a c.e.\ theory $T_0 \supseteq \mathsf{ACA}_0$ such that:
    \begin{enumerate}
        \item \label{abstract_linearity_1} For any $T,U \in \mathcal{T}$, 
        \[ \text{if } T \equiv_{\Pi^1_1} U, \text{ then } T_0 \vdash (T \equiv_{\Pi^1_1} U). \]
        \item \label{abstract_linearity_2} For each $T \in \mathcal{T}$,   
        \[ T \equiv_{\Pi^1_1} T_0 + \{\Pi^1_1\textsf{-RFN}(U) \mid U <_{\Pi^1_1} T, U \in \mathcal{T}\}, \] 
        where $U <_{\Pi^1_1} T$ denotes that $T \vdash \Pi^1_1\textsf{-RFN}(U)$.
    \end{enumerate}
    Then $\mathcal{T}$ is a linear, well-founded pre-order with respect to $<_{\Pi^1_1}$ and $<_{\Pi^1_1}$ coincides with $<_{\mathsf{Con}}$ on $\mathcal{T}$.
\end{theorem}

\begin{proof} 
The fact that the relation $<_{\Pi^1_1}$ is transitive is straightforward to prove. Its well-foundedness is established by \cite[Theorem 3.2]{pakhomov2021reflection}. Thus, it only remains to show that any two elements are comparable.

We proceed by contradiction. Suppose that $<_{\Pi^1_1}$ is not linear. Because the relation is well-founded, we can choose a $<_{\Pi^1_1}$-minimal theory $T_1 \in \mathcal{T}$ that is incomparable with some other theory in $\mathcal{T}$. That is, there exists a $T_2 \in \mathcal{T}$ such that $T_1 \not<_{\Pi^1_1} T_2$, $T_1 \not\equiv_{\Pi^1_1} T_2$, and $T_2 \not<_{\Pi^1_1} T_1$. We pick a $<_{\Pi^1_1}$-minimal $T_2$.

We first observe that the set of strictly weaker theories is identical for both $T_1$ and $T_2$. Specifically, for any $U \in \mathcal{T}$:
\[ U <_{\Pi^1_1} T_1 \iff U <_{\Pi^1_1} T_2 \]
To prove the left-to-right implication, suppose $U <_{\Pi^1_1} T_1$. By the minimality of $T_1$, $U$ must be comparable with $T_2$. If $U \equiv_{\Pi^1_1} T_2$, condition \ref{abstract_linearity_1} implies that $T_0$ proves this equivalence, hence $T_1$ proves the equivalence, thus the principles $\Pi^1_1\textsf{-RFN}(U)$ and $\Pi^1_1\textsf{-RFN}(T_2)$ are equivalent over $T_1$, and therefore $T_2 <_{\Pi^1_1} T_1$, which contradicts our initial assumption. If $T_2 <_{\Pi^1_1} U$, transitivity yields $T_2 <_{\Pi^1_1} T_1$, again a contradiction. Thus, the only remaining possibility is $U <_{\Pi^1_1} T_2$. The reverse implication ($U <_{\Pi^1_1} T_2 \implies U <_{\Pi^1_1} T_1$) is established by an analogous argument.

Consequently, by applying condition \ref{abstract_linearity_2} to both $T_1$ and $T_2$, we obtain:
\begin{align*}
T_1 &\equiv_{\Pi^1_1} T_0+\{\Pi^1_1\textsf{-RFN}(U)\mid U<_{\Pi^1_1} T_1, U\in \mathcal{T}\} \\
    &= T_0+\{\Pi^1_1\textsf{-RFN}(U)\mid U<_{\Pi^1_1} T_2, U\in \mathcal{T}\} \\
    &\equiv_{\Pi^1_1} T_2
\end{align*}
This yields $T_1 \equiv_{\Pi^1_1} T_2$, which directly contradicts our initial assumption that $T_1$ and $T_2$ are incomparable. Therefore, the $<_{\Pi^1_1}$ pre-order must be linear.

To see that $<_{\mathsf{Con}}$ coincides with $<_{\Pi^1_1}$ on $\mathcal{T}$ we simply have to observe that any extra comparison arising from $<_{\mathsf{Con}}$ would lead to one of the two following impossible situations. Either a pair of theories $T_2<_{\mathsf{Con}} T_1$, for which $T_1<_{\Pi^1_1} T_2$, which would entail that $T_2$ proves its own consistency contradicting soundness of $T_2$ in the light of G\"odel's second incompleteness theorem. Or a pair of theories $T_1<_{\mathsf{Con}}T_2$ such that $\mathsf{ACA}_0$ proves that $T_1\equiv_{\Pi^1_1} T_2$, which would also be impossible since it leads to $T_2$ proving its own consistency.
\end{proof}

For our particular case, $T_0$ is simply $\mathsf{ACA}_0$, and the set $\mathcal{T}$ from Theorem \ref{abstract_linearity} is taken to be the set of fragments of $\mathsf{BI}$ corresponding to formulas $A \in \W$. In this setting, condition \ref{abstract_linearity_2} is exactly the reduction property. Condition \ref{abstract_linearity_1}, on the other hand, would be a consequence of the fact that the calculus $\mathsf{RC_{BI}}$ proves all arithmetically true sequents $\nabla A \vdash \nabla B$ (for $A, B \in \W$), combined with the provability of the arithmetical soundness of $\mathsf{RC_{BI}}$ in $\mathsf{ACA}_0$.

Unfortunately, our proof of the linearity of $<^1_1$ (Theorem \ref{thm_lin_int}) does not proceed via a direct application of Theorem \ref{abstract_linearity}: the only way we currently know to prove the arithmetical completeness of $\mathsf{RC_{BI}}$ for sequents $\nabla A \vdash \nabla B$ is as a consequence of the linearity of the $<^1_1$-order on $\W$ itself.

Our proof of Theorem \ref{thm_lin_int} shares similarities with that of Theorem \ref{abstract_linearity}; however, we replace the use of condition \ref{abstract_linearity_1} with an involved analysis of finite Kripke models for the logic $\mathsf{RC_{BI}}$. Nonetheless, we see our construction as a proof of concept for the following general setup: if for a natural theory $T$ a sufficiently rich family of natural operators on the fragments of $T$ is discovered, then the set of fragments of $T$ generated by these operators will be well-ordered by consistency strength with the order type of the proof-theoretic ordinal of $T$. When following the approach from this work, successful application of this vision to other theories will depend on both proving the corresponding reduction properties and showing that the laws governing the interaction of these operators can be captured by a sufficiently well-behaved formal system (such as $\mathsf{RC_{BI}}$ in the case of this paper).

\vspace{\baselineskip}
\section{A modal framework for reflection-based theories up to Bar Induction}\label{section calc}
We begin by formally defining the reflection calculus for Bar Induction ($\mathsf{RC_{BI}}$), its conservative extension $\mathsf{RC_{BI_\omega}}$, as well as several key syntactic properties and a Kripke semantics for the latter. In the first subsection we present the grammar and the axiomatization of $\mathsf{RC_{BI}}$, along with a proof of its arithmetical soundness (Proposition \ref{prop ari sound}). Afterwards, in Subsection \ref{subsec:omega} we introduce the conservative extension $\mathsf{RC_{BI_\omega}}$ of $\mathsf{RC_{BI}}$, which admits a better-behaved semantic treatment from the perspective of modal logic. Subsection \ref{subsec:syntactic} is devoted to syntactic properties of derivations in $\mathsf{RC_{BI_\omega}}$, which are needed in the completeness proof of Subsection \ref{subsec:kripke}. Finally, in Subsection \ref{subsec:kripke}, we introduce a Kripke-style semantics for the logic and prove our main semantic result: $\mathsf{RC_{BI_\omega}}$ is sound and complete with respect to this class of models (Theorem \ref{thm completeness}).
\vspace{\baselineskip}
\subsection{The Reflection Calculus for Bar Induction}\label{subsec:rcbi}
The logic $\mathsf{RC_{BI}}$ is designed to satisfy the following key properties:
\begin{enumerate}
    \item It is capable of interpreting the left-to-right implications from Lemmas \ref{lem 1 red 1} through \ref{lem 5 red 5}.
    \item The syntactic entailment relation $A \vdash \dmnd{1}{1} B$ induces a well-ordering on formulas.
\end{enumerate}
We define the language $\mathcal{L}_{\mathsf{BI}}$ of the standard reflection calculus $\mathsf{RC_{BI}}$ as the modal language containing the constant $\top$ (truth), a set of propositional variables $\{p, q, \dots\}$, the binary conjunction connective $\wedge$, and the unary modal connectives: $\dmnd{1}{\alpha}$ (for $1 \leq \alpha < \omega$), $\dmnd{2}{\beta}$ (for $2 \leq \beta < \omega$), $\dwn$, and $\nabla$. The grammar for formulas in $\mathsf{RC_{BI}}$ is as follows:
\[ A := \top \mid p \mid A \wedge B \mid \dmnd{1}{\alpha} A \mid \dmnd{2}{\beta} A \mid {\dwn} A \mid \nabla A. \]

We write $A \vdash_{\mathsf{BI}} B$ to indicate that the sequent $A \vdash B$ is derivable in $\mathsf{RC_{BI}}$. Furthermore, we define $\W$ as the set of $\nabla$-free closed formulas of $\mathsf{RC_{BI}}$.

\begin{definition}
The axiomatization of $\mathsf{RC_{BI}}$ consists of a standard set of inference rules, two trivial axioms for identity and truth ($A\vdash A$ and $A\vdash \top$, which we often use without mention), and fifteen non-trivial axioms grouped into four categories: five monotonicity axioms, five transitivity axioms, two conjunction axioms, and three constructive axioms. The inference rules are as follows:
\par\vspace{1em}
\begin{multicols}{3}
    \centering
    \begin{prooftree}
        \hypo{A \vdash C}
        \infer 1[L$\wedge_1$]{A \wedge B \vdash C}
    \end{prooftree}
    \par\vspace{1.5em}
    \begin{prooftree}
        \hypo{A \vdash B}
        \infer 1[Mod$_1$]{\dmnd{i}{\alpha} A \vdash \dmnd{i}{\alpha} B}
    \end{prooftree}
    
    \columnbreak
    \begin{prooftree}
        \hypo{B \vdash C}
        \infer 1[L$\wedge_2$]{A \wedge B \vdash C}
    \end{prooftree}
    \par\vspace{1.5em}
    \begin{prooftree}
        \hypo{A \vdash B}
        \infer 1[Mod$_2$]{{\dwn} A \vdash {\dwn} B}
    \end{prooftree}
    
    \columnbreak
    \begin{prooftree}
        \hypo{A \vdash B}
        \hypo{A \vdash C}
        \infer 2[R$\wedge$]{A \vdash B \wedge C}
    \end{prooftree}
    \par\vspace{1.5em}
    \begin{prooftree}
        \hypo{A \vdash B}
        \infer 1[Mod$_3$]{\nabla A \vdash \nabla B}
    \end{prooftree}
\end{multicols}
\par\vspace{1em}
\rulerow{%
    \begin{prooftree}
        \hypo{A \vdash B}
        \hypo{B \vdash C}
        \infer 2[Cut]{A \vdash C}
    \end{prooftree}
}
\begin{itemize}
    \item \textbf{The trivial axioms:}
    \begin{multicols}{2}
    \begin{itemize}
        \item[\textbullet] $A \vdash A$ (identity); 
        \item[\textbullet] $A \vdash \top$ (truth);
    \end{itemize}
    \end{multicols}
    
    \item \textbf{Monotonicity axioms:} 
    \begin{multicols}{2}
    \begin{enumerate}
        \item $\dmnd{i}{\alpha} A \vdash \dmnd{i}{\beta} A$ for $\beta < \alpha$; 
        \item $\dmnd{2}{\alpha} A \vdash \dmnd{1}{\alpha} A$;
        \item $A \vdash \nabla A$;
        \item $\dmnd{1}{1} A \vdash \nabla A$;
        \item ${\dwn} A \vdash \nabla A$;
        \vfill\null
    \end{enumerate}
    \end{multicols}

    \item \textbf{Transitivity axioms:}
    \begin{multicols}{2}
    \begin{enumerate}
        \setcounter{enumi}{5}
        \item $\dmnd{i}{\alpha} \dmnd{j}{\alpha} A \vdash \dmnd{j}{\alpha} A$; 
        \item $\nabla \nabla A \vdash \nabla A$; 
        \item $\nabla \dmnd{1}{1} A \vdash \dmnd{1}{1} A$;
        \item $\dmnd{1}{1} \nabla A \vdash \dmnd{1}{1} A$;
        \item ${\dwn} \nabla A \vdash {\dwn} A$;
        \vfill\null
    \end{enumerate}
    \end{multicols}

    \item \textbf{Conjunction axioms:}
    \begin{multicols}{2}
    \begin{enumerate}
        \setcounter{enumi}{10}
        \item \label{Ax_conj1} $\dmnd{2}{\alpha} A \wedge \dmnd{1}{1} B \vdash \dmnd{2}{\alpha}(A \wedge {\dwn} B)$;
        \item \label{Ax_conj2} $\dmnd{i}{\alpha} A \wedge \dmnd{j}{\beta} B \vdash \dmnd{i}{\alpha}(A \wedge \dmnd{j}{\beta} B)$ for $\beta < \alpha$;
    \end{enumerate}
    \end{multicols}

    \item \textbf{Constructive axioms:}
    \begin{multicols}{2}
    \begin{enumerate}
        \setcounter{enumi}{12}
        \item ${\dwn} \top \vdash \dmnd{1}{\alpha} \top$;
        \item ${\dwn} \dmnd{1}{1} A \vdash \dmnd{1}{\alpha} {\dwn} A$;
        \item $A \wedge \dmnd{2}{2} B \vdash \nabla({\dwn} A \wedge B)$.
    \end{enumerate}
    \end{multicols}
\end{itemize}
\end{definition}

The intended arithmetical interpretation for the standard system $\mathsf{RC_{BI}}$ is a map $*$ satisfying the following conditions:
\begin{itemize}
    \item $\top^* = \mathsf{ACA}_0$; 
    \item Propositional variables $p$ correspond to arbitrary c.e.\ extensions of $\mathsf{ACA}_0$;
    \item $(A \wedge B)^* = A^* \cup B^*$; 
    \item $(\dmnd{1}{n} A)^* = \mathsf{ACA}_0 + \Pi^1_n\textsf{-RFN}(A^*)$;
    \item $(\dmnd{2}{n} A)^* = \mathsf{ACA}_0 + \Pi^1_n\text{-}\omega\textsf{RFN}(A^*)$;
    \item $({\dwn} A)^* = {\dwn} (A^*)$; 
    \item $(\nabla A)^* = \mathsf{ACA}_0 + \Pi^1_1(A^*)$.
\end{itemize}

To prove that the derivable sequents of $\mathsf{RC_{BI}}$ are arithmetically sound, it is beneficial to consider the equivalent formulation of the reflection principles as consistency statements. 

\begin{definition}\label{def con form}
We define the consistency assertions $\Sigma^1_n\textsf{-Con}(A)$ and $\Sigma^1_n\textsf{-}\omega\textsf{Con}(A)$ of a theory $A$ as:
\[ \Sigma^1_n\textsf{-Con}(A) : \forall\varphi \in \Sigma^1_n \, (\text{Tr}_{\Sigma^1_n}(\varphi) \rightarrow \text{Con}(A + \varphi)); \]
\[ \Sigma^1_n\textsf{-}\omega\textsf{Con}(A) : \forall\varphi \in \Sigma^1_n \, (\text{Tr}_{\Sigma^1_n}(\varphi) \rightarrow \exists M (M \models A + \varphi)). \]
\end{definition}

\begin{proposition}[{\cite[Section 2]{pakhomov2023reducing}}]\label{prop con form}
The following equivalences hold:
 \[ \Pi_n^1\textsf{-RFN}(A) \equiv \Sigma^1_n\textsf{-Con}(A); \]
 \[ \Pi^1_n\textsf{-}\omega\textsf{RFN}(A) \equiv \Sigma^1_n\textsf{-}\omega\textsf{Con}(A). \]
 \end{proposition}

\begin{proposition}\label{prop ari sound}
    Given formulas $A$ and $B$ of $\mathsf{RC_{BI}}$, if $A \vdash_{\mathsf{BI}} B$ then $A^* \vdash B^*$.
\end{proposition}

\begin{proof} 
    Most of the axioms generally follow by simple inclusions, or arguments similar to those provided for the arithmetical soundness of the original reflection calculus \cite{beklemishev2014positive}, or the $\mathsf{RC^\nabla}$ version found in \cite{beklemishev2018reflection}. As such, we are only going to present the proof of the more involved Axioms 11--15:

    \begin{itemize}
        \item \textbf{Axiom 11}: $\dmnd{2}{\alpha} A \wedge \dmnd{1}{1} B \vdash \dmnd{2}{\alpha} (A \wedge {\dwn} B)$.
        
        We reason in $(\dmnd{2}{\alpha} A \wedge \dmnd{1}{1} B)^*$ that $(\dmnd{2}{\alpha}(A \wedge {\dwn} B))^*$ holds. By Proposition \ref{prop con form}, we can use the consistency formulation from Definition \ref{def con form} as an equivalent arithmetical interpretation for the modalities. Then, from $(\dmnd{2}{\alpha} A)^*$ we have that for every true $\Sigma^1_\alpha$-sentence $\varphi$, there exists an $\omega$-model of $A^* + \varphi$. For every such model, the validity of $(\dmnd{1}{1} B)^*$ implies that all axioms of $({\dwn} B)^*$ will be satisfied. Considering the consistency formulation once again, we conclude that $(\dmnd{2}{\alpha} (A \wedge {\dwn} B))^*$ holds.

        \item \textbf{Axiom 12}: $\dmnd{i}{\alpha} A \wedge \dmnd{j}{\beta} B \vdash \dmnd{i}{\alpha} (A \wedge \dmnd{j}{\beta} B)$, for $\alpha > \beta$.
        
        We have two similar cases depending on the value of $i$. For $i=1$, let $C \in \Pi^1_\alpha$ be a sentence provable in $(A \wedge \dmnd{j}{\beta} B)^*$. This is equivalent to saying that $A^* \vdash (\dmnd{j}{\beta} B)^* \rightarrow C$. Since $((\dmnd{j}{\beta} B)^* \rightarrow C)$ can be trivially transformed to a sentence from $\Pi^1_\alpha$, we get that $(\dmnd{i}{\alpha} A \wedge \dmnd{j}{\beta} B)^*$ implies $C$. 
        
        For the second case $i=2$, we argue by using the consistency formulation. From $(\dmnd{2}{\alpha} A)^*$, we have that for every true $\Sigma^1_\alpha$-sentence $\varphi$, there exists an $\omega$-model of $A^* + \varphi$. Since $(\dmnd{j}{\beta} B)^*$ is a true $\Pi^1_\beta$-sentence, hence a true $\Sigma^1_\alpha$-sentence as $\beta < \alpha$, we immediately get the consistency formulation of $(\dmnd{2}{\alpha}(A \wedge \dmnd{j}{\beta} B))^*$.

        \item \textbf{Axiom 13}: $\dwn\top\vdash\dmnd{1}{\alpha}\top$.

        We reason in $\mathsf{ACA}_0$. Suppose we have $\varphi\in\Pi^1_\alpha$ such that $M\models \text{Prv}_{\mathsf{ACA}_0}(\varphi)$, where $M\models \mathsf{ACA}_0$ is an $\omega$-model. Since provability is $\Sigma^0_1$, it is absolute for $\omega$-models, so $M\models \varphi$ which implies $M\models \text{Tr}_{\Pi^1_\alpha}(\varphi)$. Thus, every $\omega$-model $M$ of $\mathsf{ACA}_0$ satisfies $\Pi^1_\alpha\textsf{-RFN}(\mathsf{ACA}_0)$ for arbitrary $\alpha$.
        
        \item \textbf{Axiom 14}: $\dwn\dmnd{1}{1} A\vdash \dmnd{1}{\alpha}\dwn A$.

        We again make use of Proposition \ref{prop con form}. Reason in $\mathsf{ACA}_0+\Pi^1_1\textsf{-RFN}(A^*)$. Let $M\models \mathsf{ACA}_0$ be an $\omega$-model, and suppose that there is a $\Sigma^1_\alpha$ formula $\varphi$ true in $M$ such that $M\models \text{Prv}_{(\dwn A)^*+\varphi}(\bot)$. This implies (by absoluteness of provability) that $(\dwn A)^*\vdash \neg\varphi$. From the definition of $\downarrow$ we get that $A^*$ proves that $\forall M'(M'\models \mathsf{ACA}_0\rightarrow M'\models \neg \varphi)$. This formula is $\Pi^1_1$, so by $\Pi^1_1\textsf{-RFN}(A^*)$ it is true. Hence, $M\models \neg\varphi$ which is a contradiction. Thus, $M\models \Sigma^1_\alpha\textsf{-Con}((\dwn A)^*)$.
        
        \item \textbf{Axiom 15}: $A \wedge \dmnd{2}{2} B \vdash \nabla({\dwn} A \wedge B)$.

        It suffices to show that every $\Pi^1_1$-sentence $C$ with $(\dwn A\wedge B)^*\vdash C$ is provable in $(A\wedge \dmnd{2}{2}B)^*$. By compactness, there is $D\in \dwn A^*$ with $B^*+ D\vdash C$. We now proceed to reason in $A^*+\Pi^1_2\textsf{-$\omega$RFN}(B^*)$. Since $D\in \dwn A^*$, $A^*$ proves that every $\omega$-model of $\mathsf{ACA}_0$ satisfies $D$, hence every $\omega$-model of $B^*$ is an $\omega$-model of $B^*+D$. Hence, $(\dmnd{2}{2}B)^*$ implies $\text{Tr}_{\Pi^1_2}(C)$. Since $C\in\Pi^1_1$ we conclude that $C$ holds.
    \end{itemize}
\end{proof}

\vspace{\baselineskip}
\subsection{The extended calculus ($\mathsf{RC_{BI_\omega}}$)}\label{subsec:omega}
$\mathsf{RC_{BI_\omega}}$ extends the standard system $\mathsf{RC_{BI}}$ by a new modality $\dmnd{1}{\omega}$. Intuitively, this operator represents the limit of the finitely labeled modalities $\dmnd{1}{\alpha}$ for $\alpha < \omega$. That is, if a formula is provable at all finite levels, $\dmnd{1}{\omega}$ allows us to treat it via a single modality. In the context of the Kripke semantics this translates to having a single unified transition.

The extended language $\mathcal{L}_{\mathsf{BI_\omega}}$ is the language $\mathcal{L}_{\mathsf{BI}}$ with the additional modal operator $\dmnd{1}{\omega}$. The grammar remains the same, with the exception of the new modality. Regarding the syntax, the only changes made are to Axiom 12, Axiom 13, and Axiom 14 to better accommodate the new modal operator.

\begin{definition}\label{def 2 oRC}
We write $A \vdash_{\mathsf{BI_\omega}} B$ to mean that $A \vdash B$ is provable in $\mathsf{RC_{BI_\omega}}$. The extended system $\mathsf{RC_{BI_\omega}}$ is defined by the following inference rules and axioms:
\par\vspace{1em}
\begin{multicols}{3}
    \centering
    \begin{prooftree}
        \hypo{A \vdash C}
        \infer 1[L$\wedge_1$]{A \wedge B \vdash C}
    \end{prooftree}
    \par\vspace{1.5em}
    \begin{prooftree}
        \hypo{A \vdash B}
        \infer 1[Mod$_1$]{\dmnd{i}{\alpha} A \vdash \dmnd{i}{\alpha} B}
    \end{prooftree}
    
    \columnbreak
    \begin{prooftree}
        \hypo{B \vdash C}
        \infer 1[L$\wedge_2$]{A \wedge B \vdash C}
    \end{prooftree}
    \par\vspace{1.5em}
    \begin{prooftree}
        \hypo{A \vdash B}
        \infer 1[Mod$_2$]{{\dwn} A \vdash {\dwn} B}
    \end{prooftree}
    
    \columnbreak
    \begin{prooftree}
        \hypo{A \vdash B}
        \hypo{A \vdash C}
        \infer 2[R$\wedge$]{A \vdash B \wedge C}
    \end{prooftree}
    \par\vspace{1.5em}
    \begin{prooftree}
        \hypo{A \vdash B}
        \infer 1[Mod$_3$]{\nabla A \vdash \nabla B}
    \end{prooftree}
\end{multicols}
\par\vspace{1em}
\rulerow{%
    \begin{prooftree}
        \hypo{A \vdash B}
        \hypo{B \vdash C}
        \infer 2[Cut]{A \vdash C}
    \end{prooftree}
}
\begin{itemize}
    \item \textbf{The trivial axioms:} 
    \begin{multicols}{2}
    \begin{itemize}
        \item[\textbullet] $A \vdash A$ (identity); 
        \item[\textbullet] $A \vdash \top$ (truth);
    \end{itemize}
    \end{multicols}

    \item \textbf{Monotonicity axioms:} 
    \begin{multicols}{2}
    \begin{enumerate}
        \item $\dmnd{i}{\alpha} A \vdash \dmnd{i}{\beta} A$ for $\beta < \alpha$; 
        \item $\dmnd{2}{\alpha} A \vdash \dmnd{1}{\alpha} A$;
        \item $A \vdash \nabla A$;
        \item $\dmnd{1}{1} A \vdash \nabla A$;
        \item ${\dwn} A \vdash \nabla A$;
        \vfill\null
    \end{enumerate}
    \end{multicols}

    \item \textbf{Transitivity axioms:}
    \begin{multicols}{2}
    \begin{enumerate}
        \setcounter{enumi}{5}
        \item $\dmnd{i}{\alpha} \dmnd{j}{\alpha} A \vdash \dmnd{j}{\alpha} A$; 
        \item $\nabla \nabla A \vdash \nabla A$; 
        \item $\nabla \dmnd{1}{1} A \vdash \dmnd{1}{1} A$;
        \item $\dmnd{1}{1} \nabla A \vdash \dmnd{1}{1} A$;
        \item ${\dwn} \nabla A \vdash {\dwn} A$;
        \vfill\null
    \end{enumerate}
    \end{multicols}

    \item \textbf{Conjunction axioms:}
    \begin{multicols}{2}
    \begin{enumerate}
        \setcounter{enumi}{10}
        \item $\dmnd{2}{\alpha} A \wedge \dmnd{1}{1} B \vdash \dmnd{2}{\alpha}(A \wedge {\dwn} B)$;
        \item $\dmnd{i}{\alpha} A \wedge \dmnd{j}{\beta} B \vdash \dmnd{i}{\alpha}(A \wedge \dmnd{j}{\beta} B)$ for $\beta < \alpha$ or $\alpha=\beta=\omega$;
    \end{enumerate}
    \end{multicols}

    \item \textbf{Constructive axioms:}
    \begin{multicols}{2}
    \begin{enumerate}
        \setcounter{enumi}{12}
        \item ${\dwn} \top \vdash \dmnd{1}{\omega} \top$;
        \item ${\dwn} \dmnd{1}{1} A \vdash \dmnd{1}{\omega} {\dwn} A$;
        \item $A \wedge \dmnd{2}{2} B \vdash \nabla({\dwn} A \wedge B)$.
    \end{enumerate}
    \end{multicols}
\end{itemize}
\end{definition}

The introduction of the $\dmnd{1}{\omega}$ modality facilitates a better-behaved semantic treatment without increasing the deductive power of the logic over the original language $\mathcal{L}_{\mathsf{BI}}$. As established by the following proposition, the extended system $\mathsf{RC_{BI_\omega}}$ is a conservative extension of the standard calculus.

\begin{proposition}\label{prop equiv}
    For any formulas $A, B \in \mathcal{L}_{\mathsf{BI}}$, the standard and extended calculi are equivalent:
    \[ A \vdash_{\mathsf{BI_\omega}} B \iff A \vdash_{\mathsf{BI}} B. \]
\end{proposition}

\begin{proof}
    For the direct implication, we call a formula $A'\in\mathcal L_{\mathsf{BI}}$ a finite instance of a formula $A\in\mathcal L_{\mathsf{BI_\omega}}$ if $A'$ is obtained from $A$ by replacing every occurrence of $\dmnd{1}{\omega}$ by some $\dmnd{1}{n}$ with $1\leq n<\omega$ ($n$ may be chosen independently for each occurrence). We prove the following stronger statement by induction on the derivation of $A\vdash_{\mathsf{BI_\omega}}B$: for any $A,B\in\mathcal L_{\mathsf{BI_\omega}}$ with $A\vdash_{\mathsf{BI_\omega}}B$ and every finite instance $B'$ of $B$, there is a finite instance $A'$ of $A$ with $A'\vdash_{\mathsf{BI}}B'$.\\
    
       \noindent If the derivation is an instance of an axiom in which every modality written out in the axiom scheme has a finite index, then every schematic letter occurring on the left occurs on the right (except for the truth axiom $A\vdash\top$, where any finite instance $A'$ of $A$ will do), so we take $A'$ to be the instance of the left-hand side in which the schematic letters are instantiated as in $B'$, and $A'\vdash_{\mathsf{BI}}B'$ is an instance of the same axiom of $\mathsf{RC_{BI}}$. Hence, for the base case we only need to check the instances of Axioms 1, 6 and 12 with $\alpha=\omega$, and Axioms 13 and 14:
    \begin{itemize}
        \item Axiom 1: $\dmnd{1}{\omega}C\vdash\dmnd{1}{\beta}C$ with $\beta<\omega$.
        Given $B'=\dmnd{1}{\beta}C'$, we take $A'=\dmnd{1}{\beta+1}C'$.
        \item Axiom 6: $\dmnd{1}{\omega}\dmnd{1}{\omega}C\vdash\dmnd{1}{\omega}C$.
        Given $B'=\dmnd{1}{n}C'$, we take $A'=\dmnd{1}{n}\dmnd{1}{n}C'$.
        \item Axiom 12: $\dmnd{1}{\omega}C\wedge\dmnd{j}{\beta}D\vdash\dmnd{1}{\omega}(C\wedge\dmnd{j}{\beta}D)$ with $\beta\leq\omega$.
        Given $B'=\dmnd{1}{n}(C'\wedge\dmnd{j}{\beta'}D')$ (where $\beta'=\beta$ if $\beta<\omega$ and the index chosen in $B'$ otherwise), we define $k=\max(n,\beta'+1)$ and take $A'=\dmnd{1}{k}C'\wedge\dmnd{j}{\beta'}D'$. Since $\beta'<k$, Axiom 12 gives $A'\vdash_{\mathsf{BI}}\dmnd{1}{k}(C'\wedge\dmnd{j}{\beta'}D')$, from which we can easily derive $A'\vdash_{\mathsf{BI}}B'$.
        \item Axiom 13: ${\dwn}\top\vdash\dmnd{1}{\omega}\top$.
        Given $B'=\dmnd{1}{n}\top$, we take $A'={\dwn}\top$.
        \item Axiom 14: ${\dwn}\dmnd{1}{1}C\vdash\dmnd{1}{\omega}{\dwn}C$.
        Given $B'=\dmnd{1}{n}{\dwn}C'$, we take $A'={\dwn}\dmnd{1}{1}C'$.
    \end{itemize}
    For the inductive step we do a case distinction on the inference rules:
    \begin{itemize}
        \item L$\wedge_1$ (and by symmetry L$\wedge_2$), from $C\vdash B$ with $A=C\wedge D$.
        
        The induction hypothesis gives a finite instance $C'$ of $C$ with $C'\vdash_{\mathsf{BI}}B'$, and we take $A'=C'\wedge D'$ for an arbitrary finite instance $D'$ of $D$, using L$\wedge_1$.
        \item R$\wedge$, from $A\vdash C$ and $A\vdash D$ with $B=C\wedge D$.
        
        Then $B'=C'\wedge D'$ for finite instances $C'$ of $C$ and $D'$ of $D$, and the induction hypothesis gives finite instances $A'_1$ and $A'_2$ of $A$ with $A'_1\vdash_{\mathsf{BI}}C'$ and $A'_2\vdash_{\mathsf{BI}}D'$. We take $A'$ to be the finite instance of $A$ which at every occurrence of $\dmnd{1}{\omega}$ takes the larger of the indices chosen in $A'_1$ and $A'_2$. Then $A'\vdash_{\mathsf{BI}}A'_1$ and $A'\vdash_{\mathsf{BI}}A'_2$, since lowering the index of a $\dmnd{1}{n}$ inside a formula is derivable by Axiom 1 together with the rules Mod$_1$, Mod$_2$, Mod$_3$, L$\wedge$ and R$\wedge$ (by a straightforward induction on $A$). Hence $A'\vdash_{\mathsf{BI}}C'$ and $A'\vdash_{\mathsf{BI}}D'$ by Cut, and $A'\vdash_{\mathsf{BI}}B'$ by R$\wedge$.
        \item Mod$_1$, from $C\vdash D$ with $A=\dmnd{i}{\alpha}C$ and $B=\dmnd{i}{\alpha}D$. 
        
        Then $B'=\dmnd{i}{\alpha'}D'$ for a finite instance $D'$ of $D$, where $\alpha'=\alpha$ if $\alpha<\omega$ and $\alpha'<\omega$ is the index chosen in $B'$ if $\alpha=\omega$. The induction hypothesis gives $C'$ with $C'\vdash_{\mathsf{BI}}D'$, and we take $A'=\dmnd{i}{\alpha'}C'$, using Mod$_1$. The rules Mod$_2$ and Mod$_3$ are treated in the same way.
        \item Cut, from $A\vdash C$ and $C\vdash B$.
        
        The induction hypothesis for $C\vdash B$ gives a finite instance $C'$ of $C$ with $C'\vdash_{\mathsf{BI}}B'$, and then the induction hypothesis for $A\vdash C$ gives $A'$ with $A'\vdash_{\mathsf{BI}}C'$; we conclude by Cut.
    \end{itemize}

    Now if $A,B\in\mathcal L_{\mathsf{BI}}$, then $A$ and $B$ are their own only finite instances, so the statement gives $A\vdash_{\mathsf{BI}}B$.\\

    The reverse implication is straightforward, as any finite instance of Axioms 13 and 14 from the standard system can be recovered using the $\omega$-versions and Axiom 1:
    \begin{enumerate}
        \item Axiom 13: ${\dwn} \top \vdash \dmnd{1}{n} \top$
        \begin{center}
            \begin{prooftree}
                \hypo{{\dwn} \top \vdash \dmnd{1}{\omega} \top}
                \hypo{\dmnd{1}{\omega} \top \vdash \dmnd{1}{n} \top}
                \infer 2[Cut]{{\dwn} \top \vdash \dmnd{1}{n} \top}
            \end{prooftree}
        \end{center}
        \item Axiom 14: ${\dwn} \dmnd{1}{1} A \vdash \dmnd{1}{n} {\dwn} A$
        \begin{center}
            \begin{prooftree}
                \hypo{{\dwn} \dmnd{1}{1} A \vdash \dmnd{1}{\omega} {\dwn} A}
                \hypo{\dmnd{1}{\omega} {\dwn} A \vdash \dmnd{1}{n} {\dwn} A}
                \infer 2[Cut]{{\dwn} \dmnd{1}{1} A \vdash \dmnd{1}{n} {\dwn} A}
            \end{prooftree}
        \end{center}
    \end{enumerate}
    Finally, we note that Axiom 12 in $\mathsf{RC_{BI_\omega}}$ already includes all instances of the axiom from the standard system.
\end{proof}

\vspace{\baselineskip}
\subsection{Syntactic properties}\label{subsec:syntactic}

In this subsection, we establish several structural properties of derivations within our calculus. Despite the inclusion of the Cut rule in the axiomatization, the restricted interaction between the different modalities and the limited set of binary connectives ensure that derivations retain a high degree of regularity. These syntactic results serve as essential preliminaries for the proof of completeness.

For convenience, we will refer to the left-conjunction rules L$\wedge_1$ and L$\wedge_2$ collectively as L$\wedge$. Furthermore, we treat the commutation of conjuncts as implicit, allowing us to simplify the presentation of derivations by omitting purely structural steps related to the associativity and commutativity of conjunction.

We denote by $A \vdash_{\mathsf{BI_\omega}}^{cf} B$ the existence of a cut-free derivation of the sequent $A \vdash B$ in $\mathsf{RC_{BI_\omega}}$.

\begin{lemma}\label{lem cf con}
    If $A \vdash_{\mathsf{BI_\omega}}^{cf} B \wedge C$, then $A \vdash_{\mathsf{BI_\omega}}^{cf} B$ and $A \vdash_{\mathsf{BI_\omega}}^{cf} C$.
\end{lemma}

\begin{proof}
    Suppose we have a cut-free derivation of $A \vdash B \wedge C$. We proceed by induction on the length of the derivation. 
    
    If the derivation consists of a single axiom, the only one that can feature a conjunction on the right-hand side is the trivial axiom $A \vdash A$ (where $A = B \wedge C$). In this case, we obtain cut-free proofs of $A \vdash B$ and $A \vdash C$ via a single application of L$\wedge_1$ or L$\wedge_2$ to the identity sequents $B \vdash B$ or $C \vdash C$, respectively.

    If the derivation is not an axiom, we inspect the last rule applied. Given that the derivation is cut-free and the right-hand side is a conjunction, the only possible rules are L$\wedge$ or R$\wedge$. 
    
    If R$\wedge$ was the last rule applied:
    \[
    \begin{prooftree}
        \hypo{A \vdash B}
        \hypo{A \vdash C}
        \infer 2[R$\wedge$]{A \vdash B \wedge C}
    \end{prooftree}
    \]
    the premises immediately provide the required cut-free derivations. 

    If L$\wedge$ was the last rule applied, then $A$ is of the form $A' \wedge D$ (or $D \wedge A'$), and the derivation ends with:
    \[
    \begin{prooftree}
        \hypo{A' \vdash B \wedge C}
        \infer 1[L$\wedge$]{A' \wedge D \vdash B \wedge C}
    \end{prooftree}
    \]
    By the induction hypothesis, there exist cut-free derivations of $A' \vdash B$ and $A' \vdash C$. Applying the L$\wedge$ rule to these premises yields $A' \wedge D \vdash B$ and $A' \wedge D \vdash C$, completing the proof.
\end{proof}

\begin{lemma}\label{lem seq}
    If $A\vdash_\mathsf{BI_\omega} B$, then there exists a sequence $A_0, A_1,\dots,A_n$ with $n\geq1$ such that $A_0=A$, $A_n=B$ and $A_i\vdash_\mathsf{BI_\omega}^{cf} A_{i+1}$ for every $i< n$.
\end{lemma}
\begin{proof}
    We proceed by induction on the length of the derivation of $A\vdash_\mathsf{BI_\omega} B$. The base case is immediate, since the derivation is an instance of an axiom and we can take $n=1$. Otherwise we distinguish cases according to the last inference rule of the derivation:
    \begin{itemize}
        \item L$\wedge_1$ (and by symmetry L$\wedge_2$): the derivation ends with
        \begin{center}
            \begin{prooftree}
                \hypo{A' \vdash B}
                \infer 1[L$\wedge_1$]{A' \wedge A'' \vdash B}
            \end{prooftree}
        \end{center}
        where $A=A'\wedge A''$. By the induction hypothesis there is a sequence $A_0, A_1,\dots,A_n$ with $n\geq1$, $A_0=A'$, $A_n=B$ and $A_i\vdash_\mathsf{BI_\omega}^{cf} A_{i+1}$ for every $i<n$. Applying L$\wedge_1$ to the cut-free derivation of $A'\vdash A_1$ gives $A'\wedge A''\vdash_\mathsf{BI_\omega}^{cf} A_1$, so we simply replace $A_0$ by $A'\wedge A''$.
        
        \item R$\wedge$: the derivation ends with
        \begin{center}
            \begin{prooftree}
                \hypo{A \vdash B'}
                \hypo{A \vdash B''}
                \infer 2[R$\wedge$]{A \vdash B' \wedge B''}
            \end{prooftree}
        \end{center}
        where $B=B'\wedge B''$. By the induction hypothesis there are sequences $A'_0, A'_1,\dots,A'_n$ and $A''_0, A''_1,\dots,A''_m$ with the required properties, such that $A'_0=A''_0=A$, $A'_n=B'$ and $A''_m=B''$. Then the sequence
        \[ A,\ A'_1\wedge A,\ \dots,\ A'_{n-1}\wedge A,\ B'\wedge A,\ B'\wedge A''_1,\ \dots,\ B'\wedge A''_{m-1},\ B'\wedge B''=B \]
        is as required: 
        \begin{itemize}
            \item $A\vdash^{cf}_\mathsf{BI_\omega} A'_1\wedge A$ by R$\wedge$ from $A\vdash^{cf}_\mathsf{BI_\omega} A'_1$ and the identity axiom;
            \item  $A'_i\wedge A\vdash^{cf}_\mathsf{BI_\omega} A'_{i+1}\wedge A$ by R$\wedge$ from $A'_i\wedge A\vdash^{cf}_\mathsf{BI_\omega} A'_{i+1}$ and $A'_i\wedge A\vdash^{cf}_\mathsf{BI_\omega} A$.
        \end{itemize} The steps $B'\wedge A''_j\vdash^{cf}_\mathsf{BI_\omega} B'\wedge A''_{j+1}$ are symmetric.
        
        \item Mod$_i$ (for $i=1,2,3$): let $\circ$ be the corresponding modality of $\mathsf{RC_{BI_\omega}}$. The derivation ends with
        \begin{center}
            \begin{prooftree}
                \hypo{A' \vdash B'}
                \infer 1[Mod$_i$]{\circ A' \vdash \circ B'}
            \end{prooftree}
        \end{center}
        where $A=\circ A'$ and $B=\circ B'$. By the induction hypothesis there is a sequence $A_0, A_1,\dots,A_n$ with $A_0=A'$, $A_n=B'$ and $A_i\vdash_\mathsf{BI_\omega}^{cf} A_{i+1}$ for every $i<n$. Applying Mod$_i$ to each of these cut-free derivations, we can take the sequence $A=\circ A_0,\circ A_1,\dots,\circ A_n=B$.
        
        \item Cut: the derivation ends with
        \begin{center}
            \begin{prooftree}
                \hypo{A \vdash C}
                \hypo{C \vdash B}
                \infer 2[Cut]{A \vdash B}
            \end{prooftree}
        \end{center}
        By the induction hypothesis there are sequences $A'_0, A'_1,\dots,A'_n$ and $A''_0, A''_1,\dots,A''_m$ with the required properties, such that $A'_0=A$, $A'_n=A''_0=C$ and $A''_m=B$. Then we can simply take their concatenation $A=A'_0,A'_1,\dots,A'_n,A''_1,\dots,A''_m=B$.
    \end{itemize}
\end{proof}

\begin{lemma}\label{lem seq 2}
    Suppose $A_0, A_1, \dots, A_n$ is a sequence such that $A_i\vdash_\mathsf{BI_\omega}^{cf} A_{i+1}$ for every $i< n$. Then we can construct a new sequence $A'_0, A'_1, \dots, A'_{n}$ with the same properties and with $A'_0=A_0$ and $A'_n=A_n$, such that for every $0<i < n$, the sequent $A'_i \vdash_{\mathsf{BI_\omega}}^{cf} A'_{i+1}$ has a cut-free derivation that does not end with an application of an L$\wedge$ rule.
\end{lemma}

\begin{proof}
    We say that $A'$ is a sub-conjunction of $A$ if $A'$ is obtained from $A$ by repeatedly replacing a conjunction by one of its two conjuncts (in particular, $A$ is a sub-conjunction of itself). We note the following two facts:
    \begin{itemize}
        \item If $A\vdash^{cf}_\bio B$ and $B'$ is a sub-conjunction of $B$, then $A\vdash^{cf}_\bio B'$, by repeated applications of Lemma \ref{lem cf con}.
        \item Every cut-free derivation of $A\vdash B$ consists of a cut-free derivation of $A'\vdash B$ that does not end with L$\wedge$, for some sub-conjunction $A'$ of $A$, followed by applications of L$\wedge$ if $A'\neq A$.
    \end{itemize}
    We modify the sequence $A_0,A_1,\dots,A_n$ by processing the steps $A_i\vdash A_{i+1}$ in the order $i=n-1,n-2,\dots,1$ (maintaining that the sequence satisfies $A_j\vdash^{cf}_\bio A_{j+1}$ for every $j<n$, and that the steps already processed have cut-free derivations that do not end with L$\wedge$).
    
    Consider the step $A_i\vdash^{cf}_\bio A_{i+1}$. By the second fact, there is a sub-conjunction $A'_i$ of $A_i$ such that $A'_i\vdash A_{i+1}$ has a cut-free derivation that does not end with L$\wedge$, and by the first fact $A_{i-1}\vdash^{cf}_\bio A'_i$. We replace $A_i$ by $A'_i$. This changes only the steps $i-1$ and $i$: step $i$ now has the required derivation, and step $i-1$ is still cut-free. Since $i$ ranges over $1,\dots,n-1$, the endpoints $A_0$ and $A_n$ are never changed, and after processing $i=1$ every step except possibly the first has a cut-free derivation that does not end with L$\wedge$.
\end{proof}

\begin{lemma}\label{lem downarrow}
    If ${\dwn} A \wedge B \vdash_{\mathsf{BI_\omega}} {\dwn} C$, then either $B \vdash_{\mathsf{BI_\omega}} {\dwn} C$ or $A \vdash_{\mathsf{BI_\omega}} \nabla C$.
\end{lemma}

\begin{proof}
    By Lemma \ref{lem seq}, there is a sequence $A_0, A_1, \dots, A_n$ such that $A_0 = {\dwn} A \wedge B$, $A_n = {\dwn} C$, and $A_i \vdash_{\mathsf{BI_\omega}}^{cf} A_{i+1}$ for all $i < n$. We will prove the statement via induction on the minimal possible length $n$ of such a sequence. We note that by Lemma \ref{lem seq 2}, we can always transform the sequence into one of equal length where none of the derivations (apart from the first one) have L$\wedge$ as the last rule applied.

    For the base case ($n=1$), we have that ${\dwn} A \wedge B \vdash_{\mathsf{BI_\omega}}^{cf} {\dwn} C$. Such a cut-free derivation exists only if the last rule applied was L$\wedge$. Then, the premise of this step is either $B \vdash {\dwn} C$, ${\dwn} A \vdash {\dwn} C$, or ${\dwn} A \wedge B' \vdash {\dwn} C$ (in which case we repeat the argument until we land in one of the first two cases). If we are in the first case, we are done. If we land in the second case, we note that the derivation for ${\dwn} A \vdash {\dwn} C$ will also be cut-free. Thus, the only possibilities are: the last rule applied was an instance of Mod$_2$, from which we get $A \vdash C$; the sequent is an instance of Axiom 10 (${\downarrow} \nabla C \vdash {\downarrow} C$), which means $A = \nabla C$; or the sequent is an instance of the identity axiom, which means $A = C$. In all cases, we obtain $A \vdash \nabla C$ (using Axiom 3 in the first and third cases).

    For the inductive step, we look at the last derivation $A_{n-1} \vdash_{\mathsf{BI_\omega}}^{cf} {\dwn} C$. Since the last derivation rule cannot be L$\wedge$, either the last rule applied is Mod$_2$, or the cut-free derivation is an instance of Axiom 10. In the first case, $A_{n-1}$ is of the form ${\dwn} C'$ with $C' \vdash_{\mathsf{BI_\omega}} C$. In the second case, $A_{n-1}$ is of the form ${\dwn} \nabla C$. In both cases, we can apply the inductive argument on ${\dwn} A \wedge B \vdash_{\mathsf{BI_\omega}} {\dwn} C'$ (and ${\dwn} A \wedge B \vdash_{\mathsf{BI_\omega}} {\dwn} \nabla C$, respectively) to get that either $B \vdash_{\mathsf{BI_\omega}} A_{n-1}$, which gives $B \vdash_{\mathsf{BI_\omega}} {\dwn} C$, or $A \vdash_{\mathsf{BI_\omega}} \nabla C'$ (and $A \vdash_{\mathsf{BI_\omega}} \nabla \nabla C$, respectively). Since $C' \vdash_{\mathsf{BI_\omega}} C$ implies $\nabla C' \vdash_{\mathsf{BI_\omega}} \nabla C$, and $\nabla \nabla C \vdash_{\mathsf{BI_\omega}} \nabla C$ by Axiom 7, both paths yield $A \vdash_{\mathsf{BI_\omega}} \nabla C$.
\end{proof}

\begin{lemma}\label{lem downarrow to diamond}
    Suppose ${\dwn}A_1 \wedge \dots \wedge {\dwn}A_k \wedge B
    \vdash_{\mathsf{BI_\omega}} \dmnd{i}{\alpha}C$. Then there exist
    sequences $A^1_p, \dots, A^{n_p}_p$ for each $p \leq k$, and a sequence
    $D_1, \dots, D_m$, such that:
    \begin{enumerate}
        \item $A_p \vdash_{\mathsf{BI_\omega}} \dmnd{1}{1}A^q_p$, for all
              $p \leq k$ and $q \leq n_p$;
        \item ${\dwn}A_1 \wedge \dots \wedge {\dwn}A_k \wedge B
              \vdash_{\mathsf{BI_\omega}} \dmnd{1}{1}D_r$, for all $r \leq m$;
        \item $\displaystyle \bigwedge_{p \leq k}
              {\dwn}\Big(\bigwedge_{q \leq n_p}\dmnd{1}{1}A^q_p\Big) \wedge
              \bigwedge_{r \leq m}\dmnd{1}{1}D_r \wedge B
              \vdash_{\mathsf{BI_\omega}} \dmnd{i}{\alpha}C$.
    \end{enumerate}
    The empty conjunction is read as $\top$.
\end{lemma}
\begin{proof}
    For $k=0$, let $m=0$. Then, conditions 1 and 2 are trivial, and condition 3 is the same as the assumption. For $k \geq 1$ and $\alpha = 1$, it suffices to take all $n_p = 0$, $m = 1$, and $D_1 = C$. Then, condition 1 is again trivial, condition 2 is the assumption, and condition 3 holds by L$\wedge$. Thus, for the remainder of the proof we assume $k \geq 1$ and $\alpha > 1$.
 
    By Lemma \ref{lem seq} we know that there is a sequence $E_0, \dots, E_l$ with $E_0={\dwn}A_1 \wedge \dots \wedge {\dwn}A_k \wedge B$, $E_l = \dmnd{i}{\alpha}C$ such that $E_t \vdash^{cf}_{\mathsf{BI_\omega}} E_{t+1}$ for all $t < l$. We argue by induction on the minimal length $l$ of such a sequence.
    
    Fix a sequence of minimal length $l$ for the given sequent. By Lemma \ref{lem seq 2} we may assume that none of the sequents $E_t \vdash^{cf}_{\mathsf{BI_\omega}} E_{t+1}$, apart possibly from the first, ends with an application of L$\wedge$, and by minimality we may assume that no derivation is an instance of the identity axiom.
 
    We record one bookkeeping observation used throughout: if sequences $(A^q_p)$, $(D_r)$ and $({A'}^q_p)$, $({D'}_r)$ satisfy conditions 1--3 for conclusions $\dmnd{i}{\alpha}C$ and $\dmnd{j}{\beta}C'$ respectively, then their concatenation satisfies conditions 1 and 2, and satisfies the analogue of condition 3 for each of the two conclusions separately. Conditions 1 and 2 are easy to see, and for condition 3 we use that $\displaystyle {\dwn}\Big(\bigwedge_{q \in Q}\dmnd{1}{1}A^q_p\Big) \vdash {\dwn}\Big(\bigwedge_{q \in Q'}\dmnd{1}{1}A^q_p\Big)$ for  $Q' \subseteq Q$.
 
    We also note that no axiom with a $\Diamond$-conclusion has a ${\dwn}$-conjunct on its left-hand side: the only axioms with a conjunction on the left and a $\Diamond$ on the right are Axioms 11 and 12, and both of their conjuncts are $\Diamond$-formulas.

    We now proceed with the main inductive argument:
    \begin{itemize}
        \item Base case ($l = 1$): we have a cut-free derivation of $E_0 \vdash \dmnd{i}{\alpha}C$. Since $E_0$ contains a ${\dwn}$-conjunct and the right-hand side is not a conjunction, by the note above the last rule must be L$\wedge$. Iterating the argument through the premises as in Lemma \ref{lem seq 2}, we reach a cut-free derivation of either $B' \vdash \dmnd{i}{\alpha}C$ for a sub-conjunction $B'$ of $B$, or ${\dwn}A_p \vdash \dmnd{i}{\alpha}C$ for some $p \leq k$. In the first case we take all $n_p = m = 0$, and condition 3 holds by L$\wedge$. In the second case, the only cut-free derivations of a sequent with ${\dwn}$-leading left side and $\Diamond$-leading right side are instances of Axiom 13 or Axiom 14.

        For Axiom 13 we have $A_p = \top$ and $\dmnd{i}{\alpha}C = \dmnd{1}{\omega}\top$. By taking $n_p = m = 0$, so that the $p$-th conjunct of condition 3 is ${\dwn}\top$, condition 3 holds by L$\wedge$ and Axiom 13.
        
         For Axiom 14 we have $A_p = \dmnd{1}{1}A'$ and $\dmnd{i}{\alpha}C = \dmnd{1}{\omega}{\dwn}A'$. Let $n_p = 1$, $A^1_p = A'$ and all other sequences empty. Condition 1 holds by the identity axiom, and condition 3 reads $\dots \wedge {\dwn}\dmnd{1}{1}A' \wedge \dots \vdash \dmnd{1}{\omega}{\dwn}A'$, which holds by L$\wedge$ and Axiom 14.

        \item Inductive step ($l \geq 2$): consider the last sequent in the sequence: $E_{l-1} \vdash^{cf}_{\mathsf{BI_\omega}} \dmnd{i}{\alpha}C$. Its derivation does not end with L$\wedge$, and the right-hand side is not a conjunction, so it either ends with Mod$_1$ or is an axiom instance. 
        
        In the first case $E_{l-1} = \dmnd{i}{\alpha}C'$ with $C' \vdash^{cf}_{\mathsf{BI_\omega}} C$. Since the sequence $E_0, \dots, E_{l-1}$ has length $l-1$, the induction hypothesis yields sequences satisfying conditions 1--3 for $\dmnd{i}{\alpha}C'$. By using Cut with $\dmnd{i}{\alpha}C' \vdash \dmnd{i}{\alpha}C$ (from $C' \vdash C$ by Mod$_1$) we get condition 3 for $\dmnd{i}{\alpha}C$. 
        
        The second case is for the sequent to be an instance of Axiom 1, 2, 6, 11, 12, 13, or 14. When the sequence ends in a conjunction $X \wedge Y$, we apply the induction hypothesis as follows. Lemma \ref{lem cf con} applied to the final sequent yields cut-free $E_{l-2} \vdash X$ and $E_{l-2} \vdash Y$, hence sequences $E_0, \dots, E_{l-2}, X$ and $E_0, \dots, E_{l-2}, Y$ of length $l-1$. We apply the induction hypothesis to each separately and concatenate the outputs by the bookkeeping observation.
    \begin{itemize}
        \item Axioms 1, 2, 6: in each case $E_{l-1}$ is of the form $\dmnd{i'}{\alpha'}C'$ with $\alpha' \geq \alpha > 1$, and the axiom derives $\dmnd{i'}{\alpha'}C' \vdash \dmnd{i}{\alpha}C$. We apply the induction hypothesis to the sequence $E_0, \dots, E_{l-1}$ and Cut the axiom instance onto the derivation of condition 3.
        \item Axiom 11: $E_{l-1} = \dmnd{2}{\alpha}A' \wedge \dmnd{1}{1}B'$ and $\dmnd{i}{\alpha}C = \dmnd{2}{\alpha}(A' \wedge {\dwn}B')$. Splitting the sequence, the induction hypothesis applies to the branch ending in $\dmnd{2}{\alpha}A'$ (as $\alpha > 1$), yielding sequences $(A^q_p)$, $(D_r)$ with
        \[ \bigwedge_{p \leq k}{\dwn}\Big(\bigwedge_{q}\dmnd{1}{1}A^q_p\Big)
           \wedge \bigwedge_{r}\dmnd{1}{1}D_r \wedge B
           \vdash_{\mathsf{BI_\omega}} \dmnd{2}{\alpha}A'. \]
        The branch ending in $\dmnd{1}{1}B'$ witnesses $E_0 \vdash_{\mathsf{BI_\omega}} \dmnd{1}{1}B'$. We then extend the $D$-sequence by $D_{m+1} := B'$, for which condition 2 holds from the witness. Condition 3 for the extended sequences follows by R$\wedge$ with the conjunct $\dmnd{1}{1}D_{m+1}$ followed by Axiom 11.
        
        \item Axiom 12: $E_{l-1} = \dmnd{i}{\alpha}A' \wedge \dmnd{j}{\beta}B'$ with $\beta < \alpha$ or $\alpha = \beta = \omega$, and $\dmnd{i}{\alpha}C = \dmnd{i}{\alpha}(A' \wedge \dmnd{j}{\beta}B')$. We split the sequence and apply the induction hypothesis to the branch ending in $\dmnd{i}{\alpha}A'$. If $\beta > 1$, we also apply it to the branch ending in $\dmnd{j}{\beta}B'$ and concatenate. Condition 3 follows by R$\wedge$ and Axiom 12. If $\beta = 1$ (implying $j = 1$), we extend the $D$-sequence by $D_{m+1} := B'$ as in the previous case, and condition 3 then follows by R$\wedge$ and Axiom 12.
        
        \item Axiom 13: $E_{l-1} = {\dwn}\top$ and $\dmnd{i}{\alpha}C = \dmnd{1}{\omega}\top$. We take all sequences empty. Since $k \geq 1$, the left-hand side of condition 3 contains the conjunct ${\dwn}\top$, and ${\dwn}\top \vdash \dmnd{1}{\omega}\top$ is Axiom 13. Condition 3 follows by L$\wedge$.
        
        \item Axiom 14: $E_{l-1} = {\dwn}\dmnd{1}{1}A'$ and $\dmnd{i}{\alpha}C = \dmnd{1}{\omega}{\dwn}A'$. The sequence witnesses $E_0 \vdash_{\mathsf{BI_\omega}} {\dwn}\dmnd{1}{1}A'$. By applying Lemma \ref{lem downarrow} $k$ times (to the conjuncts ${\dwn}A_p$ one at a time), either $B \vdash_{\mathsf{BI_\omega}} {\dwn}\dmnd{1}{1}A'$, or $A_p \vdash_{\mathsf{BI_\omega}} \nabla\dmnd{1}{1}A'$ for some $p \leq k$. 
        
        In the first case we take all sequences empty. The left-hand side of condition 3 proves $B$ by L$\wedge$, hence ${\dwn}\dmnd{1}{1}A'$ by Cut, hence $\dmnd{1}{\omega}{\dwn}A'$ by Axiom 14.
        
        In the second case, Axiom 8 yields $A_p \vdash_{\mathsf{BI_\omega}} \dmnd{1}{1}A'$. We take $n_p = 1$, $A^1_p = A'$ and all other sequences empty. Condition 3 then reads $\dots \wedge {\dwn}\dmnd{1}{1}A' \wedge \dots \vdash \dmnd{1}{\omega}{\dwn}A'$, which follows by L$\wedge$ and Axiom 14.
    \end{itemize}
    \end{itemize}
\end{proof}

\vspace{\baselineskip}
\subsection{Kripke semantics}\label{subsec:kripke}
We introduce a fairly straightforward Kripke semantics for $\mathsf{RC_{BI_\omega}}$, with all modalities treated in a standard $\Diamond$-modality style. We prove the completeness of the logic with respect to this class of models via a canonical-model construction. Let $\mathsf{Rel(RC_{BI_\omega})} = \{R^i_\alpha, R^1_\omega, R^\dwn, R^\nabla \mid i \in \{1, 2\}, i \leq \alpha < \omega\}$ be the set of relation symbols corresponding to the modalities of $\mathsf{RC_{BI_\omega}}$.

A Kripke frame in the language $\mathcal{L}_{\mathsf{BI_\omega}}$ is defined as a tuple $(\mathcal{S}, \mathcal{R})$, where $\mathcal{S}$ is a non-empty set of states, and $\mathcal{R}$ is a family of binary relations on $\mathcal{S}$ indexed by the symbols in $\mathsf{Rel(RC_{BI_\omega})}$. A Kripke model $(\mathcal{S}, \mathcal{R}, \mathcal{V})$ is a Kripke frame $(\mathcal{S}, \mathcal{R})$ together with a valuation function $\mathcal{V} : \mathcal{S} \times \mathsf{Var} \to \{0, 1\}$, where $\mathsf{Var}$ is the set of propositional variables, assigning truth values to these variables at each state. We call a model finite if the set $\mathcal{S}$ is finite.

\begin{definition}\label{def 3 frame 1}
    An $\mathsf{RC_{BI_\omega}}$ model is a Kripke model $W=(\mathcal S, \mathcal R,\mathcal V)$ in the language $\mathcal{L}_\mathsf{BI_\omega}$ that satisfies the following conditions:
    \begin{multicols}{2}
    \begin{enumerate}[label=\Roman*.]
    \item \textbf{Monotonicity:}
    \begin{enumerate}[label=I.\arabic*]
        \item $aR^i_{\alpha} b$ implies $aR^i_\beta b$, for $\beta<\alpha$;
        \item $aR^2_\alpha b$ implies $aR^1_\alpha b$;
        \item $aR^\nabla a$, for all $a\in\mathcal{S}$;
        \item $aR^1_\alpha b$ implies $aR^\nabla b$;
        \item $aR^\dwn b$ implies $aR^\nabla b$.
    \end{enumerate}
    \item \textbf{Transitivity:} 
    \begin{enumerate}[label=II.\arabic*]
        \item $a R^i_\alpha b R^j_\alpha c$ implies $a R^j_\alpha c$;
        \item $aR^\nabla bR^\nabla c$ implies $a R^\nabla c$;
        \item $aR^\nabla b R^1_\alpha c$ implies $aR^1_1 c$;
        \item $a R^1_\alpha b R^\nabla c$ implies $a R^1_1 c$;
        \item $a R^\dwn b R^\nabla c$ implies $aR^\dwn c$.
    \end{enumerate}
    \item \textbf{Conjunction conditions:}
    \begin{enumerate}[label=III.\arabic*]
        \item $aR^i_\alpha b\text{ and } a R^j_\beta c$ implies $bR^j_\beta c$,\\ for $\beta<\alpha$ or $\alpha=\beta=\omega$;
        \item $a R^2_\alpha b\text{ and } aR^i_\beta c$ implies $bR^\dwn c$.
    \end{enumerate}
    \item \textbf{Constructive conditions:} 
    \begin{enumerate} [label=IV.\arabic*]
        \item $aR^\dwn b$ implies $\exists c$ such that $aR^1_\omega c$;
        \item $aR^\dwn bR^1_\alpha c$ implies $\exists d(aR^1_\omega dR^\dwn c)$;
        \item $aR^2_\alpha b$ implies $\exists c$ such that: $aR^\nabla c$, $cR^\dwn a$, $\forall d\forall R(bRd\Rightarrow cRd)$, and $\mathcal{V}(b,p)\leq \mathcal{V}(c,p)$ for all $p\in\mathsf{Var}$.
    \end{enumerate}
\end{enumerate}
\vfill\null
\end{multicols}
\end{definition}

\begin{definition}\label{def 4 sat model}
    Given an $\mathsf{RC_{BI_\omega}}$ model $W = (\mathcal{S}, \mathcal{R}, \mathcal{V})$ and a state $a \in \mathcal{S}$, we inductively define the semantic satisfaction relation $W, a \models A$ as follows:
    \begin{itemize}
        \item $W, a \models \top$; $W, a \models p$ iff $\mathcal{V}(a, p) = 1$ for $p \in \mathsf{Var}$;
        \item $W, a \models A \wedge B$ iff $W, a \models A$ and $W, a \models B$;
        \item $W, a \models \dmnd{1}{\alpha} A$ iff $\exists b \in \mathcal{S}$ such that $W, b \models A$ and $a R^1_\alpha b \in \mathcal{R}$;
        \item $W, a \models \dmnd{2}{\alpha} A$ iff $\exists b \in \mathcal{S}$  such that $W, b \models A$ and $a R^2_\alpha b \in \mathcal{R}$;
        \item $W, a \models {\dwn} A$ iff $\exists b \in \mathcal{S}$ such that $W, b \models A$ and $a R^\dwn b \in \mathcal{R}$;
        \item $W, a \models \nabla A$ iff $\exists b \in \mathcal{S}$ such that $W, b \models A$ and $a R^\nabla b \in \mathcal{R}$.
    \end{itemize}
    In proofs with a defined model $W = (\mathcal{S}, \mathcal{R}, \mathcal{V})$, when evident, we will often simply write $a R b$ to denote that $a R b \in \mathcal{R}$.
    A sequent $A \vdash B$ is true in $W$ if for any $a \in \mathcal{S}$, 
    \[ W, a \models A \implies W, a \models B. \]
\end{definition}

\begin{theorem}\label{thm completeness}
    $\mathsf{RC_{BI_\omega}}$ is sound and complete with respect to $\mathsf{RC_{BI_\omega}}$ models.
\end{theorem}
We will prove soundness and completeness separately in Proposition \ref{prop Soundness} and Proposition \ref{prop completeness}.

\begin{proposition}[Soundness]\label{prop Soundness}
    $\mathsf{RC_{BI_\omega}}$ is sound for the class of all $\mathsf{RC_{BI_\omega}}$ models. Explicitly, if $A \vdash_{\mathsf{BI_\omega}} B$, then $A \vdash B$ is true in all $\mathsf{RC_{BI_\omega}}$ models.
\end{proposition}

\begin{proof}
    The proof proceeds by direct inspection of the axioms and inference rules of $\mathsf{RC_{BI_\omega}}$. The inference rules preserve truth in a model by direct inspection of Definition \ref{def 4 sat model}. For the axioms, the correspondence with the frame conditions is as follows: Axioms 1--5 correspond to conditions I.1--I.5, Axioms 6--10 correspond to II.1--II.5, Axioms 11 and 12 correspond to III.2 and III.1, and Axioms 13--15 correspond to IV.1--IV.3. Axioms 1--14 follow immediately as they correspond precisely to the frame conditions (or a slightly weakened version in the case of Axioms 8, 9, and 11). 

    The only non-immediate case left is Axiom 15 ($A\wedge\dmnd{2}{2}B\vdash\nabla({\downarrow}A\wedge B)$): let $a\models A$ and $aR^2_\gamma b$ ($\gamma\geq2$) with $b\models B$, and let $c$ be as provided by IV.3. From $cR^\downarrow a$ and $a\models A$ we get $c\models{\downarrow}A$; from $b\models B$ we get $c\models B$. Hence $aR^\nabla c$, and so $a\models\nabla({\downarrow}A\wedge B)$.
\end{proof}

We now turn to the proof of completeness. We begin by introducing the canonical model $W_{\mathsf{BI_\omega}}$ in Definition \ref{def adequate 3}, whose states are given by deductively closed sets of formulas. In Lemma \ref{lem 7 adequate 2}, we show that membership in the closed theory associated with a state of $W_{\mathsf{BI_\omega}}$ corresponds precisely to semantic satisfaction at that state. In Lemma \ref{lem adequate 1}, we establish that the canonical model is indeed an $\mathsf{RC_{BI_\omega}}$ model by verifying the frame conditions. Finally, we conclude the completeness proof in Proposition \ref{prop completeness}.

\begin{definition}\label{def adequate 3}
    We define the canonical model $W_{\mathsf{BI_\omega}} = (\mathcal{S}_{\mathsf{BI_\omega}}, \mathcal{R}_{\mathsf{BI_\omega}}, \mathcal{V}_{\mathsf{BI_\omega}})$ of $\mathsf{RC_{BI_\omega}}$ as follows:
    \begin{itemize}
        \item $\mathcal{S}_{\mathsf{BI_\omega}}$ is the set of all deductively closed sets of formulas in $\mathcal{L}_{\mathsf{BI_\omega}}$.
        \item For any variable $p \in \mathsf{Var}$ and $a \in \mathcal{S}_{\mathsf{BI_\omega}}$, $\mathcal{V}_{\mathsf{BI_\omega}}(a, p) = 1$ iff $p \in a$.
        \item Given two states $a, b \in \mathcal{S}_{\mathsf{BI_\omega}}$:
        \begin{itemize}
            \item $a R^i_\alpha b \in \mathcal{R}_{\mathsf{BI_\omega}}$ iff either $\alpha < \omega$, or ($i=1$ and $\alpha=\omega$), and the following hold for every formula $A$:
            \begin{enumerate}[label=A\arabic*:]
                \item $A \in b$ implies $\dmnd{i}{\alpha} A \in a$.
                \item $(\beta < \alpha \text{ or } \alpha = \beta = \omega)$ and $\dmnd{j}{\beta} A \in a$ implies $\dmnd{j}{\beta} A \in b$.
                \item (Only if $i=2$) $\dmnd{j}{\beta} A \in a$ implies ${\dwn} A \in b$.
            \end{enumerate}
            \item $a R^\dwn b \in \mathcal{R}_{\mathsf{BI_\omega}}$ iff the following holds for every formula $A$:
            \begin{enumerate}[label=B\arabic*:]
                \item $A \in b$ implies ${\dwn} A \in a$.
            \end{enumerate}
            \item $a R^\nabla b \in \mathcal{R}_{\mathsf{BI_\omega}}$ iff the following holds for every formula $A$:
            \begin{enumerate}[label=C\arabic*:]
                \item $A \in b$ implies $\nabla A \in a$.
            \end{enumerate}
        \end{itemize}
    \end{itemize}
\end{definition}
\begin{lemma}\label{lem 7 adequate 2}
    For any $A\in \mathcal{L}_\mathsf{BI_\omega}$, $$W_{\mathsf{BI_\omega}},a\models A\text{ iff } A\in a.$$
\end{lemma}
\begin{proof}
We prove this via induction on the construction of $A$. If $A$ is $\top$ or a conjunction, the proof is immediate for both directions. 
\vspace{\baselineskip}\\
($\Rightarrow$) Suppose that $W_{\mathsf{BI_\omega}}, a \models A$. If $A$ is of the form $\dmnd{i}{\alpha} B$, $\dwn B$ or $\nabla B$, then by Definition \ref{def 4 sat model} there exists a corresponding accessible state $b$ with $W_{\mathsf{BI_\omega}},b\models B$. By the induction hypothesis we have that $B\in b$. From this and the rules A1, B1 and C1 we get that $A\in a$.
\vspace{\baselineskip}\\
($\Leftarrow$) For the reverse implication, given $A\in a$ we have five different cases based on the outermost modality of $A$: $A=\dmnd{1}{\alpha} B$, $A=\dmnd{1}{\omega} B$, $A=\dwn B$, $A=\nabla B$, and $A=\dmnd{2}{\alpha} B$.

The first case we consider is $A=\dmnd{1}{\alpha} B$. We need to show that there is a state $b$ such that $aR^1_\alpha b$ and $W_{\mathsf{BI_\omega}}, b\models B$. Let $\Delta=\{\dmnd{i}{\gamma} C\;|\;\dmnd{i}{\gamma} C\in a $, $ \gamma<\alpha\}$ and $b$ be the deductive closure of $\Delta\cup \{B\}$. By using the induction hypothesis on $B\in b$ we have that $W_{\mathsf{BI_\omega}},b\models B$, so we are left to show that $aR^1_\alpha b$. We can prove this by simply checking that the conditions A1 and A2 hold for $aR^1_\alpha b$:
    \begin{enumerate}[label=A\arabic*:]
        \item If $D\in b$ we have to show that $\dmnd{1}{\alpha} D\in a$. Since $D\in b$ then there is a finite set of formulas $\{\dmnd{i_1}{\gamma_1}C_1,\dmnd{i_2}{\gamma_2}C_2,...,\dmnd{i_n}{\gamma_n}C_n\}\subseteq \Delta$ such that $\dmnd{i_1}{\gamma_1}C_1\wedge...\wedge \dmnd{i_n}{\gamma_n} C_n\wedge B\vdash_{\mathsf{BI_\omega}} D$. By the definition of $\Delta$ we have that $\forall k\leq n(\gamma_k<\alpha)$, so by repeated applications of Axiom 12 we can derive that $\dmnd{i_1}{\gamma_1}C_1\wedge...\wedge \dmnd{i_n}{\gamma_n} C_n\wedge \dmnd{1}{\alpha} B\vdash_{\mathsf{BI_\omega}}\dmnd{1}{\alpha}(\dmnd{i_1}{\gamma_1}C_1\wedge...\wedge \dmnd{i_n}{\gamma_n} C_n\wedge B)$. By using these as assumptions, we can do the following derivation:
        \end{enumerate}
        \begin{center}
        \footnotesize
            \begin{prooftree}
                \hypo{\dmnd{i_1}{\gamma_1}C_1\wedge...\wedge \dmnd{i_n}{\gamma_n} C_n\wedge \dmnd{1}{\alpha} B\vdash \dmnd{1}{\alpha}(\dmnd{i_1}{\gamma_1}C_1\wedge...\wedge \dmnd{i_n}{\gamma_n} C_n\wedge B)}
                \hypo{\dmnd{i_1}{\gamma_1}C_1\wedge...\wedge \dmnd{i_n}{\gamma_n} C_n\wedge B\vdash D}
                \infer 1[Mod$_1$]{\dmnd{1}{\alpha}(\dmnd{i_1}{\gamma_1}C_1\wedge...\wedge \dmnd{i_n}{\gamma_n} C_n\wedge B)\vdash \dmnd{1}{\alpha} D}
                \infer 2[Cut]{\dmnd{i_1}{\gamma_1}C_1\wedge...\wedge \dmnd{i_n}{\gamma_n} C_n\wedge \dmnd{1}{\alpha} B\vdash \dmnd{1}{\alpha} D}
            \end{prooftree}
        \end{center}
        Since $a$ is deductively closed, we conclude that $\dmnd{1}{\alpha} D\in a$.
        \begin{enumerate}[resume, label=A\arabic*:]
        \item If $\dmnd{j}{\gamma} D\in a$ with $\gamma<\alpha$ we have to show that $\dmnd{j}{\gamma} D\in b$. Since $\dmnd{j}{\gamma} D\in a$ and $\gamma<\alpha$, from the definition of $\Delta$ we have that $\dmnd{j}{\gamma} D\in \Delta$. Since $\Delta\subset b$ we immediately get that $\dmnd{j}{\gamma} D\in b$.
    \end{enumerate}
    
    Suppose now that $A=\dmnd{1}{\omega} B$. The argument will be identical to the case above, with the only slight difference being in the definition of the set $\Delta$. We prove that there is a state $b$ such that $aR^1_\omega b$ and $W_\mathsf{BI_\omega}, b\models B$. Let $\Delta=\{\dmnd{i}{\gamma} C\;|\;\dmnd{i}{\gamma} C\in a \}$ and $b$ be the deductive closure of $\Delta\cup \{B\}$. By using the induction hypothesis on $B\in b$ we have that $W_{\mathsf{BI_\omega}},b\models B$, so we are left to show that $aR^1_\omega b$. We can prove this by simply checking that the conditions A1 and A2 hold for $aR^1_\omega b$:
    \begin{enumerate}[label=A\arabic*:]
        \item If $D\in b$ we have to show that $\dmnd{1}{\alpha} D\in a$. Since $D\in b$ then there is a finite set of formulas $\{\dmnd{i_1}{\gamma_1}C_1,\dmnd{i_2}{\gamma_2}C_2,...,\dmnd{i_n}{\gamma_n}C_n\}\subseteq \Delta$ such that $\dmnd{i_1}{\gamma_1}C_1\wedge...\wedge \dmnd{i_n}{\gamma_n} C_n\wedge B\vdash_{\mathsf{BI_\omega}} D$. By repeated applications of Axiom 12 we can derive that $\dmnd{i_1}{\gamma_1}C_1\wedge...\wedge \dmnd{i_n}{\gamma_n} C_n\wedge \dmnd{1}{\omega} B\vdash_{\mathsf{BI_\omega}}\dmnd{1}{\omega}(\dmnd{i_1}{\gamma_1}C_1\wedge...\wedge \dmnd{i_n}{\gamma_n} C_n\wedge B)$. By using these as assumptions, we can again do the same derivation as above:
        \end{enumerate}
        \begin{center}
            \begin{prooftree}
            \footnotesize
                \hypo{\dmnd{i_1}{\gamma_1}C_1\wedge...\wedge \dmnd{i_n}{\gamma_n} C_n\wedge \dmnd{1}{\omega} B\vdash \dmnd{1}{\omega}(\dmnd{i_1}{\gamma_1}C_1\wedge...\wedge \dmnd{i_n}{\gamma_n} C_n\wedge B)}
                \hypo{\dmnd{i_1}{\gamma_1}C_1\wedge...\wedge \dmnd{i_n}{\gamma_n} C_n\wedge B\vdash D}
                \infer 1[Mod$_1$]{\dmnd{1}{\omega}(\dmnd{i_1}{\gamma_1}C_1\wedge...\wedge \dmnd{i_n}{\gamma_n} C_n\wedge B)\vdash \dmnd{1}{\omega} D}
                \infer 2[Cut]{\dmnd{i_1}{\gamma_1}C_1\wedge...\wedge \dmnd{i_n}{\gamma_n} C_n\wedge \dmnd{1}{\omega} B\vdash \dmnd{1}{\omega} D}
            \end{prooftree}
        \end{center}
        Since $a$ is deductively closed, we conclude that $\dmnd{1}{\omega} D\in a$.
        \begin{enumerate}[resume, label=A\arabic*:]
        \item If $\dmnd{j}{\gamma} D\in a$ we have to show that $\dmnd{j}{\gamma} D\in b$. Since $\dmnd{j}{\gamma} D\in a$, from the definition of $\Delta$ we have that $\dmnd{j}{\gamma} D\in \Delta$. Since $\Delta\subset b$ we immediately get that $\dmnd{j}{\gamma} D\in b$.
    \end{enumerate}
    
    For the case $A=\dwn B$ we show that there exists a state $b$ such that $aR^\dwn b$ and $W_{\mathsf{BI_\omega}},b\models B$. For this we simply consider $b=\{C\;|\;B\vdash_\mathsf{BI_\omega}C\}$. By definition $b$ is deductively closed, so it is a state in $W_{\mathsf{BI_\omega}}$. Also from its definition we have that $B\in b$, and by the induction hypothesis we get $W_{\mathsf{BI_\omega}} ,b\models B$. To show that $a R^\dwn b$ we only need to check the condition B1, namely that if $C\in b$ then $\dwn C\in a$. Since for any formula $C\in b$ we have $B\vdash_\mathsf{BI_\omega}C$, by a simple application of the rule Mod$_2$, we get $\dwn B\vdash_\mathsf{BI_\omega}\dwn C$. Since $a$ is deductively closed and $\dwn B\in a$, we get that $\dwn C\in a$. 
    
    The argument for the case $A=\nabla B\in a$ is similar to the above one. We show that if we take $b=\{C\;|\;B\vdash_\mathsf{BI_\omega}C\}$, then $a R^\nabla b$ and $W_{\mathsf{BI_\omega}} , b\models B$. By definition $b$ is deductively closed, so it is a state in $W_{\mathsf{BI_\omega}}$. Also from its definition we have that $B\in b$, and by the induction hypothesis we get $W_{\mathsf{BI_\omega}} ,b\models B$. To show that $a R^\nabla b$ we only need to check the condition C1, namely that if $C\in b$ then $\nabla C\in a$. Since for any formula $C\in b$ we have $B\vdash_\mathsf{BI_\omega}C$, by a simple application of the rule Mod$_3$, we get $\nabla B\vdash_\mathsf{BI_\omega}\nabla C$. Since $a$ is deductively closed and $\nabla B\in a$, we get that $\nabla C\in a$.
    
    The last remaining case is $A=\dmnd{2}{\alpha} B$. We have to show that there is a state $b$ such that $aR^2_\alpha b$ and $W_{\mathsf{BI_\omega}}, b\models B$. Let $\Delta=\{\dmnd{i}{\gamma} C\;|\;\dmnd{i}{\gamma} C\in a $, $ \gamma<\alpha\}$, $\Delta^\dwn=\{\dwn C\;|\;\dmnd{i}{\gamma} C\in a $, for some $\gamma\}$ and $b$ be the deductive closure of $\Delta\cup\Delta^\dwn\cup\{B\}$. From the induction hypothesis it is clear that $W_{\mathsf{BI_\omega}},b \models B$. To show that $aR^2_\alpha b$ we are going to check that the conditions A1--A3 are satisfied:
    \begin{enumerate}[label=A\arabic*:]
        \item If $D\in b$ then $\dmnd{2}{\alpha} D\in a$. Since $D\in b$ then there are finite sequences $(\dmnd{i_k}{\gamma_k}C_k)_{k\leq n}$ from $\Delta$ and $(\dmnd{i'_k}{\gamma'_k}C'_k)_{k\leq m}$ from $a$ (hence $(\dwn C'_k)_{k\leq m}$ are in $\Delta^\dwn$) such that: 
    \end{enumerate}
        $$\dmnd{i_1}{\gamma_1}C_1\wedge...\wedge \dmnd{i_n}{\gamma_n} C_n\wedge \dwn C'_1\wedge \dwn C'_2 \wedge ... \wedge \dwn C'_m\wedge B\vdash_{\mathsf{BI_\omega}} D.$$ By using the Mod$_1$ rule on this we get:
        \begin{equation}\label{equation 2.18 1}\dmnd{2}{\alpha}(\dmnd{i_1}{\gamma_1}C_1\wedge...\wedge \dmnd{i_n}{\gamma_n} C_n\wedge \dwn C'_1\wedge \dwn C'_2 \wedge ... \wedge\dwn C'_m\wedge B)\vdash_{\mathsf{BI_\omega}} \dmnd{2}{\alpha} D.\end{equation}
        On the other hand, by repeated applications of instances of Axiom 12 we can derive that: 
        $$\dmnd{2}{\alpha} B\wedge \dmnd{i_1}{\gamma_1}C_1\wedge...\wedge \dmnd{i_n}{\gamma_n} C_n\wedge \dmnd{i'_1}{\gamma'_1}C'_1\wedge...\wedge \dmnd{i'_m}{\gamma'_m} C'_m\vdash_{\mathsf{BI_\omega}}\dmnd{2}{\alpha} (B\wedge \dmnd{i_1}{\gamma_1}C_1\wedge...\wedge \dmnd{i_n}{\gamma_n} C_n)\wedge \dmnd{i'_1}{\gamma'_1}C'_1\wedge...\wedge \dmnd{i'_m}{\gamma'_m} C'_m$$ and by using instances of Axiom 11 we can also derive: $$\dmnd{2}{\alpha} (B\wedge \dmnd{i_1}{\gamma_1}C_1\wedge...\wedge \dmnd{i_n}{\gamma_n} C_n)\wedge \dmnd{i'_1}{\gamma'_1}C'_1\wedge...\wedge \dmnd{i'_m}{\gamma'_m} C'_m\vdash_\mathsf{BI_\omega}\dmnd{2}{\alpha} (B\wedge \dmnd{i_1}{\gamma_1}C_1\wedge...\wedge \dmnd{i_n}{\gamma_n} C_n\wedge \dwn C'_1\wedge...\wedge \dwn C'_m).$$ Combining these two via the Cut rule we get: 
        \begin{equation}\label{equation 2.18 2}\dmnd{2}{\alpha} B\wedge \dmnd{i_1}{\gamma_1}C_1\wedge...\wedge \dmnd{i_n}{\gamma_n} C_n\wedge \dmnd{i'_1}{\gamma'_1}C'_1\wedge...\wedge \dmnd{i'_m}{\gamma'_m} C'_m\vdash_\mathsf{BI_\omega}\dmnd{2}{\alpha} (B\wedge \dmnd{i_1}{\gamma_1}C_1\wedge...\wedge \dmnd{i_n}{\gamma_n} C_n\wedge \dwn C'_1\wedge...\wedge \dwn C'_m).
        \end{equation}
        Using Cut on (\ref{equation 2.18 1}) and (\ref{equation 2.18 2}) gives us a derivation for:
        $$\dmnd{2}{\alpha} B\wedge \dmnd{i_1}{\gamma_1}C_1\wedge...\wedge \dmnd{i_n}{\gamma_n} C_n\wedge \dmnd{i'_1}{\gamma'_1}C'_1\wedge...\wedge \dmnd{i'_m}{\gamma'_m} C'_m\vdash_\mathsf{BI_\omega}\dmnd{2}{\alpha} D.$$ Since $a$ is deductively closed we can conclude that $\dmnd{2}{\alpha} D\in a$.
    \begin{enumerate}[resume, label=A\arabic*:]
        \item If $\dmnd{i}{\gamma} D\in a$ with $\gamma<\alpha$ then $\dmnd{i}{\gamma} D\in b$. From the definition of $\Delta$ we have that $\dmnd{i}{\gamma} D\in \Delta$. Since $\Delta \subset b$ we immediately get $\dmnd{i}{\gamma} D\in b$.

        \item If $\dmnd{i}{\gamma} D\in a$ then $\dwn D\in b$. By definition of $\Delta^\dwn$ we have that $\dwn D\in\Delta^\dwn$. Since $\Delta^\dwn\subset b$ we get $\dwn D\in b$.
    \end{enumerate}
\end{proof} 

\begin{lemma}\label{lem packing}
    Let $a, b \in \mathcal{S}_{\mathsf{BI_\omega}}$ with $aR^2_\alpha b \in \mathcal{R}_{\mathsf{BI_\omega}}$, and suppose $\dmnd{1}{1}T_1, \dots, \dmnd{1}{1}T_k \in a$. Then ${\dwn}\big(\dmnd{1}{1}T_1 \wedge \dots \wedge \dmnd{1}{1}T_k\big) \in b$.
\end{lemma}
\begin{proof}
    Since $b$ is deductively closed, $\top \in b$, so by condition A1 for $aR^2_\alpha b$ we have $\dmnd{2}{\alpha}\top \in a$. As $\alpha \geq 2 > 1$, $k$ applications of Axiom 12 yield
    \[ \dmnd{2}{\alpha}\top\wedge\dmnd{1}{1}T_1 \wedge \dots \wedge \dmnd{1}{1}T_k \vdash_{\mathsf{BI_\omega}} \dmnd{2}{\alpha}\big(\top \wedge \dmnd{1}{1}T_1 \wedge \dots \wedge \dmnd{1}{1}T_k\big), \] 
    so by deductive closure of $a$ we get $\dmnd{2}{\alpha}\big(\bigwedge_{q \leq k}\dmnd{1}{1}T_q\big) \in a$. Condition A3 for $aR^2_\alpha b$ then gives ${\dwn}\big(\bigwedge_{q \leq k}\dmnd{1}{1}T_q\big) \in b$.
\end{proof}

\begin{lemma}\label{lem adequate 1}
    $W_{\mathsf{BI_\omega}}$ is an $\mathsf{RC_{BI_\omega}}$ model.
\end{lemma}  
\begin{proof} We are going to prove that the model conditions I--IV are satisfied:
\begin{enumerate}[label=\Roman*.]
    \item Monotonicity: since states are deductively closed, these follow immediately from A1, B1 and C1.
    \item Transitivity: we check each condition.
        \begin{enumerate}[label=II.\arabic*]
            \item $aR^i_\alpha bR^j_\alpha c$ implies $aR^j_\alpha c$: we verify that A1--A3 hold for $aR^j_\alpha c$.
             \begin{enumerate}[label=A\arabic*:]
                \item If $A\in c$ then by condition A1 for $bR^j_\alpha c$ we have $\dmnd{j}{\alpha} A\in b$. By A1 for $aR^i_\alpha b$ we get $\dmnd{i}{\alpha} \dmnd{j}{\alpha} A\in a$. Since $a$ is deductively closed and $\dmnd{i}{\alpha} \dmnd{j}{\alpha} A\vdash_\mathsf{BI_\omega} \dmnd{j}{\alpha} A$ we get that $\dmnd{j}{\alpha} A\in a$.
                \item If $\beta <\alpha$ (or $\alpha=\beta=\omega$) and $\dmnd{k}{\beta} A\in a$, by A2 for $aR^i_\alpha b$ we have $\dmnd{k}{\beta} A\in b$, and further by A2 on $bR^j_\alpha c$ we get that $\dmnd{k}{\beta} A\in c$.
                \item Only for $j=2$. If $\dmnd{k}{\beta} A\in a$, then $\dmnd{1}{1} A\in a$ by Axioms 2 and 1. By A2 for $aR^i_\alpha b$ we have $\dmnd{1}{1} A\in b$, and by A3 for $bR^2_\alpha c$ we get $\dwn A\in c$.
            \end{enumerate}
            \item $aR^\nabla b R^\nabla c$  implies $a R^\nabla c$: we only need to check that C1 holds for $aR^\nabla c$. 
            
            If $A\in c$ then by C1 for $bR^\nabla c$ we have $\nabla A\in b$, and by C1 for $aR^\nabla b$ we then get $\nabla\nabla A\in a$. Since $a$ is deductively closed and $\nabla \nabla A\vdash_\mathsf{BI_\omega}\nabla A$ we get that $\nabla A\in a$.
            \item $aR^\nabla bR^1_\alpha c$ implies $a R^1_1 c$: we only need to check A1 for $aR^1_1 c$, as A2 is vacuous for rank $1$.

            If $A\in c$ then by A1 for $bR^1_\alpha c$ we have $\dmnd{1}{\alpha} A\in b$, and by C1 for $aR^\nabla b$ we get $\nabla \dmnd{1}{\alpha} A\in a$. Since $a$ is deductively closed and $\nabla \dmnd{1}{\alpha} A\vdash_\mathsf{BI_\omega}\dmnd{1}{1} A$, we conclude that $\dmnd{1}{1} A\in a$.
            \item $aR^1_\alpha b R^\nabla c$ implies $a R^1_1 c$: as before, we only need to check A1 for $aR^1_1 c$.

            If $A\in c$ then by C1 for $bR^\nabla c$ we have $\nabla A\in b$. From this, by A1 for $aR^1_\alpha b$ we get $\dmnd{1}{\alpha} \nabla A\in a$. Since $\dmnd{1}{\alpha} \nabla A\vdash_{\mathsf{BI_\omega}}\dmnd{1}{1} A$ and $a$ is deductively closed, we conclude that $\dmnd{1}{1} A\in a$.
            \item $aR^\dwn bR^\nabla c$ implies $a R^\dwn c$: we check B1 for $aR^\dwn c$.

            If $A\in c$ then by C1 for $bR^\nabla c$ we have $\nabla A\in b$. Then, by B1 for $aR^\dwn b$ we have $\dwn\nabla A\in a$. Since $\dwn\nabla A\vdash_\mathsf{BI_\omega}\dwn A$, by deductive closure we get $\dwn A\in a$.
        \end{enumerate}
    \item Conjunction conditions:
        \begin{enumerate}[label=III.\arabic*]
            \item $aR^i_\alpha b$ and $a R^j_\beta c$ implies $bR^j_\beta c$, for $\beta<\alpha$ or $\alpha=\beta=\omega$: we verify A1--A3 for $bR^j_\beta c$.
                \begin{enumerate}[label=A\arabic*:]
                    \item If $C\in c$ then by A1 for $aR^j_\beta c$ we have $\dmnd{j}{\beta} C\in a$. Since $\beta<\alpha$ or $\alpha=\beta=\omega$, by A2 for $aR^i_\alpha b$ we get $\dmnd{j}{\beta} C\in b$.
                    \item If $\dmnd{k}{\gamma} B\in b$, for $\gamma<\beta$, then by A1 for $aR^i_\alpha b$ we have $\dmnd{i}{\alpha} \dmnd{k}{\gamma} B\in a$. Since $\dmnd{i}{\alpha} \dmnd{k}{\gamma} B\vdash_\mathsf{BI_\omega}\dmnd{k}{\gamma} B$ we have $\dmnd{k}{\gamma} B\in a$. From this, by A2 for $aR^j_\beta c$ we get $\dmnd{k}{\gamma} B\in c$.
                    \item Only for $j=2$. If $\dmnd{k}{\gamma} B\in b$ then by A1 for $aR^i_\alpha b$ we get $\dmnd{i}{\alpha} \dmnd{k}{\gamma} B\in a$. Since $a$ is deductively closed and $\dmnd{i}{\alpha} \dmnd{k}{\gamma} B\vdash_\mathsf{BI_\omega}\dmnd{1}{1} B$, we have that $\dmnd{1}{1} B\in a$. By now using A3 on $aR^2_\beta c$ we get $\dwn B\in c$.
                \end{enumerate}
             \item $a R^2_\alpha b$ and $aR^i_\beta c$ implies $bR^\dwn c$: we only need to verify B1 for $bR^\dwn c$:

             If $A\in c$ then by A1 for $aR^i_\beta c$ we have $\dmnd{i}{\beta} A\in a$. By A3 on $aR^2_\alpha b$ we get that $\dwn A\in b$.
        \end{enumerate}
     \item Constructive conditions:
        \begin{enumerate} [label=IV.\arabic*] 
            \item $aR^\dwn b$ implies $\exists c(aR^1_\omega c)$. By B1 on $aR^\dwn b$ we have that $\dwn\top\in a$. Since $a$ is deductively closed and $\dwn\top\vdash_\mathsf{BI_\omega}\dmnd{1}{\omega} \top$, we also have that $\dmnd{1}{\omega}\top \in a$. Then by Lemma \ref{lem 7 adequate 2} there has to be a state $c$ such that $aR^1_\omega c$.
            \item $aR^\dwn bR^1_\alpha c$ implies $\exists d(aR^1_\omega dR^\dwn c)$. Let $d$ be the deductive closure of the set $\{\dwn C\;|\;C\in c\}\cup\{\dmnd{i}{\beta} A\;|\;\dmnd{i}{\beta} A\in a\}$. By definition $d R^\dwn c$. For $aR^1_\omega d$ we verify A1 and A2:
            \begin{enumerate}[label=A\arabic*:]
                \item If $D\in d$ then there are sequences $( C_k)_{k\leq n}$ from $c$, and $(\dmnd{i_k}{\gamma_k}A_k)_{k\leq m}$ from $a$ such that $\dwn C_1\wedge ...\wedge \dwn C_n\wedge \dmnd{i_1}{\gamma_1}A_1\wedge ...\wedge \dmnd{i_m}{\gamma_m} A_m\vdash_\mathsf{BI_\omega} D$. Since $c$ is deductively closed, we then have that $C'=C_1\wedge C_2\wedge\dots\wedge C_n$ is in $c$, hence $\dwn C'$ is in $d$. Moreover, $\dwn C'\wedge \dmnd{i_1}{\gamma_1}A_1\wedge ...\wedge \dmnd{i_m}{\gamma_m} A_m\vdash_\mathsf{BI_\omega} D$. By using the Mod$_1$ rule we get that 
                \begin{equation}\label{equation 2.19 IV.2 A1}
                \dmnd{1}{\omega}(\dwn C'\wedge \dmnd{i_1}{\gamma_1}A_1\wedge ...\wedge \dmnd{i_m}{\gamma_m} A_m)\vdash_\mathsf{BI_\omega} \dmnd{1}{\omega} D.
                \end{equation}
                Since $bR^1_\alpha c$ and $C'\in c$, condition A1 gives $\dmnd{1}{\alpha}C'\in b$, and since $aR^{\downarrow}b$, condition B1 gives ${\dwn}\dmnd{1}{\alpha}C'\in a$. Now ${\dwn}\dmnd{1}{\alpha}C'\vdash_{\mathsf{BI_\omega}}{\dwn}\dmnd{1}{1}C' \vdash_{\mathsf{BI_\omega}}\dmnd{1}{\omega}{\dwn}C'$ (Axiom 1 with Mod$_2$, then Axiom 14), so by deductive closure $\dmnd{1}{\omega}{\dwn}C'\in a$, and hence
                \[
                \dmnd{1}{\omega}{\dwn}C'\wedge\dmnd{i_1}{\gamma_1}A_1\wedge\dots\wedge\dmnd{i_m}{\gamma_m}A_m\in a.
                \]
                Each $\dmnd{i_k}{\gamma_k}A_k$ has $\gamma_k<\omega$, or $\gamma_k=\omega$ and $i_k=1$, so by $m$ applications of Axiom 12 we get
                \[
                \dmnd{1}{\omega}{\dwn}C'\wedge\dmnd{i_1}{\gamma_1}A_1\wedge\dots\wedge \dmnd{i_m}{\gamma_m}A_m\vdash_{\mathsf{BI_\omega}} \dmnd{1}{\omega}\big({\dwn}C'\wedge\dmnd{i_1}{\gamma_1}A_1\wedge\dots\wedge \dmnd{i_m}{\gamma_m}A_m\big).
                \]
                By using Cut with \eqref{equation 2.19 IV.2 A1} we get $\dmnd{1}{\omega}D\in a$, by the deductive closure of $a$.
                
                \item If $\dmnd{i}{\beta} A\in a$, then from the definition of the set $d$ we immediately get that $\dmnd{i}{\beta} A\in d$.
            \end{enumerate}
            \item $aR^2_\alpha b$ implies $\exists c$ such that $aR^\nabla c$,
            $cR^{\dwn}a$, $\forall d\,\forall R\,(bRd \Rightarrow cRd)$, and
            $\mathcal{V}(b,p) \leq \mathcal{V}(c,p)$ for all $p \in \mathsf{Var}$.
            
            Let $c$ be the deductive closure of $\{{\dwn}A \mid A \in a\} \cup b$. Since $b \subseteq c$, the condition $\mathcal{V}(b,p) \leq \mathcal{V}(c,p)$ is immediate. For $cR^{\dwn}a$, every $A \in a$ has ${\dwn}A \in c$ by construction, so B1 is satisfied.
            
            For $aR^\nabla c$ we have to verify C1: if $C \in c$ then $\nabla C \in a$. Given $C \in c$, there are $A_1, \dots, A_k \in a$ and $B_1, \dots, B_l \in b$ with ${\dwn}A_1 \wedge \dots \wedge {\dwn}A_k \wedge B_1 \wedge \dots \wedge B_l \vdash_{\mathsf{BI_\omega}} C$. By condition A1 for $aR^2_\alpha b$ we have $\dmnd{2}{\alpha}(B_1 \wedge \dots \wedge B_l) \in a$, hence by Axiom 1 also $\dmnd{2}{2}(B_1 \wedge \dots \wedge B_l) \in a$, so $A_1 \wedge \dots \wedge A_k \wedge \dmnd{2}{2}(B_1 \wedge \dots \wedge B_l)
            \in a$, from which Axiom 15 gives $\nabla({\dwn}A_1 \wedge \dots \wedge {\dwn}A_k \wedge B_1 \wedge \dots \wedge B_l) \in a$, and hence $\nabla C \in a$.
            
            We first prove the following property:
            \begin{equation}\label{eq b in c}
                \text{For every formula }\dmnd{j}{\beta}C,\text{ if }\dmnd{j} {\beta}C \in c\text{ then }\dmnd{j}{\beta}C \in b.
            \end{equation}
            \noindent To prove (\ref{eq b in c}), pick $A_1, \dots, A_k \in a$ and $B \in b$ such that
            \[ {\dwn}A_1 \wedge \dots \wedge {\dwn}A_k \wedge B
               \vdash_{\mathsf{BI_\omega}} \dmnd{j}{\beta}C. \]
            Lemma \ref{lem downarrow to diamond} yields sequences $(A^q_p)_{q \leq n_p}$
            (for $p \leq k$) and $(D_r)_{r \leq m}$ such that:
            \begin{enumerate}
                \item[(a)] $A_p \vdash_{\mathsf{BI_\omega}} \dmnd{1}{1}A^q_p$ for all $p, q$.\\
                
                By deductive closure of $a$, each $\dmnd{1}{1}A^q_p \in a$, and by Lemma \ref{lem packing} applied to $aR^2_\alpha b$,
                \[ {\dwn}\Big(\bigwedge_{q \leq n_p}\dmnd{1}{1}A^q_p\Big) \in b
                   \quad \text{for every } p \leq k. \]
                   
                \item[(b)] ${\dwn}A_1 \wedge \dots \wedge {\dwn}A_k \wedge B \vdash_{\mathsf{BI_\omega}} \dmnd{1}{1}D_r$ for all $r \leq m$.\\
                
                Each $\dmnd{1}{1}D_r \in c$, so by the already established condition C1 for $aR^\nabla c$ we get $\nabla\dmnd{1}{1}D_r \in a$, and by Axiom 8 and deductive closure, $\dmnd{1}{1}D_r \in a$. Condition A2 for $aR^2_\alpha b$ (as $1 < \alpha$) then gives $\dmnd{1}{1}D_r \in b$ for every $r \leq m$.\\
                
                \item[(c)] $\displaystyle \bigwedge_{p \leq k} {\dwn}\Big(\bigwedge_{q \leq n_p}\dmnd{1}{1}A^q_p\Big) \wedge \bigwedge_{r \leq m}\dmnd{1}{1}D_r \wedge B  \vdash_{\mathsf{BI_\omega}} \dmnd{j}{\beta}C$.
            \end{enumerate}
            By (a) and (b), every conjunct of the left-hand side of (c) belongs to $b$, so $\dmnd{j}{\beta}C \in b$ by deductive closure. Hence (\ref{eq b in c}) holds.\\
        
            It remains to verify $\forall d\,\forall R\,(bRd \Rightarrow cRd)$. The cases $R = R^{\dwn}$ and $R = R^\nabla$ are immediate since $b \subseteq c$. Suppose $bR^j_\gamma d$. We verify conditions A1--A3 for $cR^j_\gamma d$:
            \begin{itemize}
                \item[A1:] If $D \in d$, then $\dmnd{j}{\gamma}D\in b$ by A1 for $bR^j_\gamma d$, and $b \subseteq c$ gives $\dmnd{j}{\gamma}D \in c$.
                \item[A2:] Suppose $\beta < \gamma$ (or $\beta = \gamma = \omega$) and $\dmnd{i'}{\beta}C \in c$. By (\ref{eq b in c}), $\dmnd{i'}{\beta}C \in b$, and A2 for $bR^j_\gamma d$ gives $\dmnd{i'}{\beta}C \in d$.
                \item[A3:] ($j = 2$) Suppose $\dmnd{i'}{\beta}C \in c$. By (\ref{eq b in c}), $\dmnd{i'}{\beta}C \in b$, and A3 for $bR^2_\gamma d$ gives ${\dwn}C \in d$.
            \end{itemize}
        \end{enumerate}
\end{enumerate}
\end{proof}

\begin{proposition}[Completeness]\label{prop completeness}
    If $A \vdash B$ is true in all $\mathsf{RC_{BI_\omega}}$ models, then $A \vdash_{\mathsf{BI_\omega}} B$.
\end{proposition}

\begin{proof}
    We prove the contrapositive. Suppose that $A \nvdash_{\mathsf{BI_\omega}} B$ and let $a$ be the deductive closure of $A$ within $\mathcal{L}_{\mathsf{BI_\omega}}$. By definition we have that $A \in a$ and $B \notin a$. Also, $a$ is a valid state in $W_{\mathsf{BI_\omega}}$. From Lemma \ref{lem adequate 1}, we know that $W_{\mathsf{BI_\omega}}$ is an $\mathsf{RC_{BI_\omega}}$ model. By Lemma \ref{lem 7 adequate 2}, since $A \in a$ and $B \notin a$, we have $W_{\mathsf{BI_\omega}}, a \models A$ and $W_{\mathsf{BI_\omega}}, a \not\models B$. Therefore, the sequent $A \vdash B$ is not true in $W_{\mathsf{BI_\omega}}$.
\end{proof}


\vspace{\baselineskip}
\section{Finite model constructions}\label{section fin models}
In this section, we introduce several types of finite rooted structures derived from the variable-free formulas of $\mathsf{RC_{BI_\omega}}$ that will serve as the main technical tools in the proofs of Section~\ref{section linearity}. A key advantage of working with finite models is that they admit a natural notion of embeddability, which will play a central role in connecting semantic and syntactic aspects of the logic.

\begin{definition}\label{def 9 embedding}
    An embedding $f : V \hookrightarrow W$ of a model $V = (\mathcal{S}_V, \mathcal{R}_V, \mathcal{V}_V)$ into a model $W = (\mathcal{S}_W, \mathcal{R}_W, \mathcal{V}_W)$ is a function from $\mathcal{S}_V$ to $\mathcal{S}_W$ such that:
    \begin{itemize}
        \item $\forall a, b \in \mathcal{S}_V$ and any relation symbol $R$, $a R b \in \mathcal{R}_V$ implies $f(a) R f(b) \in \mathcal{R}_W$;
        \item For all $a\in\mathcal{S}_V$ and all propositional variables $p \in \mathsf{Var}$, if $V, a \models p$ then $W, f(a) \models p$.
    \end{itemize} 
    If such a function exists we write $V \hookrightarrow W$ and say that $V$ is embeddable into $W$.
\end{definition}

\begin{proposition}\label{prop embedding}
    Given models $V = (\mathcal{S}_V, \mathcal{R}_V, \mathcal{V}_V)$ and $W = (\mathcal{S}_W, \mathcal{R}_W, \mathcal{V}_W)$, if $f : V \hookrightarrow W$ is an embedding, then for every $a \in \mathcal{S}_V$ and every formula $A$,
    \[ V, a \models A \implies W, f(a) \models A. \]
\end{proposition}

\begin{proof} 
    We prove this via induction on the construction of $A$. The base case of $A$ being $\top$ is immediate, and the case of $A$ being a propositional variable is true by the definition of an embedding. 

    For the case $A = \dmnd{1}{\alpha} B$, if $V, a \models \dmnd{1}{\alpha} B$, then there exists $b \in \mathcal{S}_V$ such that $a R^1_\alpha b \in \mathcal{R}_V$ and $V, b \models B$. This entails that $f(a) R^1_\alpha f(b) \in \mathcal{R}_W$. By the inductive hypothesis we have $W, f(b) \models B$, hence $W, f(a) \models \dmnd{1}{\alpha} B$.
    
    The cases for $A = \dmnd{2}{\alpha} B$, $A = \nabla B$, and $A = {\dwn} B$ follow by the exact same argument.

    For $A = B \wedge C$, if $V, a \models B \wedge C$ then $V, a \models B$ and $V, a \models C$. By the inductive hypothesis we have that $W, f(a) \models B$ and $W, f(a) \models C$, from which we get $W, f(a) \models B \wedge C$.
\end{proof}

\vspace{\baselineskip}
\subsection{Rooted models and modal trees $T_A$}
In this subsection, we introduce the notion of rooted models and standard modal trees $T_A$ for formulas of $\mathsf{RC_{BI_\omega}}$. For a given formula $A$, the modal tree $T_A$ is a rooted model that explicitly represents the structure of $A$ in terms of its sub-formulas and modal operators. Each node in $T_A$ corresponds to a sub-formula, and the transitions are labeled by the respective modalities appearing in $A$. Despite not being full $\mathsf{RC_{BI_\omega}}$ models themselves, the modal trees characterise satisfiability: a formula holds at a given state in a Kripke model if and only if its modal tree can be embedded into the model such that the root is mapped to that state (Proposition \ref{prop embed}).

Essentially, $T_A$ can be viewed as a simple directed tree graph with a distinguished root point, constructed directly from the syntactic form of the formula. We proceed by formally defining rooted structures as follows:

\begin{definition}\label{def root}
    A rooted model $M$ within $\mathcal{L}_{\mathsf{BI_\omega}}$ is a tuple $(r_M, \mathcal{S}_M, \mathcal{R}_M, \mathcal{V}_M)$ where:
    \begin{itemize}
        \item $(\mathcal{S}_M, \mathcal{R}_M, \mathcal{V}_M)$ is a Kripke model in $\mathcal{L}_{\mathsf{BI_\omega}}$;
        \item $r_M \in \mathcal{S}_M$ is such that $\forall s \in \mathcal{S}_M$, there exists a directed path from $r_M$ to $s$. More formally, for any $s \neq r_M$, there exists a finite sequence of states $s_1, s_2, \dots, s_n \in \mathcal{S}_M$ and relation symbols $R_0, R_1, \dots, R_n$ such that $r_M R_0 s_1$, $s_1 R_1 s_2$, $\dots$, $s_n R_n s$ are all in $\mathcal{R}_M$.
    \end{itemize}
\end{definition}

\begin{definition}\label{def root embedding}
    Given rooted models $M_1$ and $M_2$ with roots $r_{M_1}$ and $r_{M_2}$ respectively, we write $M_1 \hookrightarrow_r M_2$ to denote that there exists an embedding $f : M_1 \hookrightarrow M_2$ such that $f(r_{M_1}) = r_{M_2}$.
\end{definition}

We denote the trivial model as $T_\top = (r_\top, \{r_\top\}, \emptyset, \mathcal{V}_\top)$, where $\mathcal{V}_\top(r_\top, p) = 0$ for all propositional variables $p \in \mathsf{Var}$. 


We define $T_A$ similarly to the canonical trees $T[A]$ from \cite{beklemishev2014positive}:
\begin{definition}\label{def 10 trees 1}
    Given an $\mathsf{RC_{BI_\omega}}$ formula $A$, we inductively define the modal tree $T_A = (r^T_A, \mathcal{S}^T_A, \mathcal{R}^T_A, \mathcal{V}^T_A)$ associated with it as follows:
    \begin{itemize}
        \item $T_\top$ is as previously defined; 
        \item $T_p = (r^T_p, \{r^T_p\}, \emptyset, \mathcal{V}^T_p)$, where $\mathcal{V}^T_p(r^T_p, q) = 1$ iff $q = p$ for $q \in \mathsf{Var}$;
        \item $T_{\dmnd{i}{\alpha}A}$ is $T_A$ extended by a new root point $r^T_{\dmnd{i}{\alpha}A}$ such that $r^T_{\dmnd{i}{\alpha}A} R^i_\alpha r^T_A$;
        \item $T_{\dwn A}$ is $T_A$ extended by a new root point $r^T_{\dwn A}$ such that $r^T_{\dwn A} R^\dwn r^T_A$;
        \item $T_{\nabla A}$ is $T_A$ extended by a new root point $r^T_{\nabla A}$ such that $r^T_{\nabla A} R^\nabla r^T_A$;
        \item $T_{A\wedge B}$ is the disjoint union of $T_A$ and $T_B$ with roots $r^T_A$ and $r^T_B$ identified (note: the roots take the union of the valuations).
    \end{itemize}
\end{definition}

\begin{proposition}\label{prop embed}
    If $W = (\mathcal{S}_W, \mathcal{R}_W, \mathcal{V}_W)$ is an $\mathsf{RC_{BI_\omega}}$ model, then for any $a \in \mathcal{S}_W$ and formula $A$, 
    \[ W, a \models A \text{ iff } \exists f : T_A \hookrightarrow W \text{ with } f(r^T_A) = a. \]
\end{proposition}

\begin{proof}
     ($\Leftarrow$) An easy induction on $A$ shows that $T_A, r^T_A\models A$: for $A$ of the forms $\dmnd{i}{\alpha} B$, ${\downarrow}B$, and $\nabla B$ the added root transition points to the root of $T_B$, and for conjunctions the glued root inherits the transitions and valuations of both roots. Proposition \ref{prop embedding} then transfers satisfaction along the embedding $f$, giving us that $W,a\models A$.

    ($\Rightarrow$) By induction on $A$. The cases $A=\top$ and $A=p$ are trivial. If $A=B\wedge C$, the embeddings of $T_B$ and $T_C$ obtained from the induction hypothesis both map their respective roots to $a$, hence their union is an embedding of $T_{B\wedge C}$. If $A=\dmnd{i}{\alpha} B$, then $W,a\models A$ provides $b$ with $aR^i_\alpha b$ and $W,b\models B$. Thus we can map the root of $T_A$ to $a$, and the induction hypothesis embeds $T_B$ at $b$. The cases $A={\downarrow}B$ and $A=\nabla B$ are analogous, using the witnessing transition directly.
\end{proof}

\vspace{\baselineskip}
\subsection{Tree-like models $M_A$}\label{subsec:tree-like-submodels}
We next define the finite tree-like model $M_A$ associated with a variable-free formula $A$ of $\mathsf{RC_{BI_\omega}}$. Unlike the modal trees $T_A$, the structure $M_A$ constitutes a valid $\mathsf{RC_{BI_\omega}}$ model (Proposition \ref{prop MA is a model}) such that $A$ is satisfied at its root (Proposition \ref{prop Mmodel}). Lemma \ref{lem 11 trees 3} establishes that satisfiability in the canonical model is characterized by embeddings from $M_A$. These properties yield a semantic characterization of derivability in $\mathsf{RC_{BI_\omega}}$: a sequent $A \vdash_{\mathsf{BI_\omega}} B$ is derivable if and only if $B$ is satisfied at the root of $M_A$ (Theorem \ref{thm trees}). For the rest of the paper, since we are working with the variable-free fragment, we omit valuation functions.

\begin{definition}\label{def rank}
We define the rank of a transition in a model as $1$ if it is an $R^\dwn$ or $R^\nabla$ transition, and $\alpha$ if it is an $R^i_\alpha$ transition. Correspondingly, we define the rank of the modality $\dmnd{i}{\alpha}$ as $\alpha$, and that of $\nabla$ and $\dwn$ as $1$.
\end{definition}
\begin{definition}\label{def alpha path}
    A path from $a$ to $b$ in a model is a sequence of points $a=a_0,a_1,\dots,a_n=b$ with $n\geq0$ such that $a_{k-1}Ra_k$ for every $1\leq k\leq n$. We call a path proper if $n\geq1$. For $\alpha\geq 2$, we define an $\alpha$-path to be a path whose transitions have rank $\geq\alpha$.
\end{definition}

\begin{definition}\label{def 12 trees 3}
    Given a variable-free $\mathsf{RC_{BI_\omega}}$ formula $A$, we define the rooted tree-like model $M_A = (r_A, \mathcal{S}_A, \mathcal{R}_A)$ as follows:
    \begin{enumerate}[label=(\arabic*)]
    \item \textbf{Pre-model} $M^0_A=(r_A, \mathcal{S}^0_A, \mathcal{R}^0_A)$ is defined via induction on $A$ in terms of the pre-models and models of the immediate sub-formulas of $A$:
    \begin{itemize}
        \item If $A = \top$, $M^0_A = T_\top$;
        \item If $A = \nabla B$, then $M^0_A$ is $M^0_B$ extended by $r_A$ such that $r_A R^\nabla r_B$;
        \item If $A = B \wedge C$, then $M^0_A$ is the disjoint union of $M^0_B$ and $M^0_C$ with $r_B$ and $r_C$ identified;
        \item If $A = \dmnd{1}{\alpha} B$, then $M^0_A$ is $M^0_B$ extended by $r_A$ such that $r_A R^1_\alpha r_B$;
         \item If $A = \dmnd{2}{\alpha} B$, then $M^0_A$ is $M^0_B$ extended by $r_A$ such that $r_A R^2_\alpha r_B$ and  $r_BR^1_\omega r_B\in \mathcal R^0_A$;
        \item If $A = {\dwn} B$, then $M^0_A$ is $M^0_B$ extended by $r_A$, a zero point $r^0_A$, and intermediate points $e^i_A$:
        \begin{itemize}
            \item $\mathcal{S}^0_A = \mathcal{S}^0_B \cup \{r_A,r_A^0\} \cup \{e^i_A\mid\exists b_i\in \mathcal S^0_B,\ r_B R^1_1 b_i\in\mathcal R_B\}$;
            \item $\mathcal R^0_A=\mathcal R^0_B\cup \{r_AR^\dwn r_B,\; r_AR^1_\omega r_A^0\}\cup \{r_AR^1_\omega e^i_A,\;e^i_AR^\dwn b_i\;|\;e^i_A\in \mathcal S^0_A\}$.
        \end{itemize}
    \end{itemize}
    Since the pre-models $M^0_A$ can be seen as extensions of the tree models $T_A$, we are going to call tree points the nodes that can be identified with the nodes of $T_A$. Note that these will then be precisely the points that are not an intermediate point or a zero point. As such, we can write the set of states as $\mathcal{S}^0_A=\mathcal{S}_A^{~tree}\cup\mathcal{S}_A^{~int}\cup \mathcal{S}_A^{~zero}$, where:
        \begin{itemize}
            \item[-] $\mathcal{S}_A^{~tree}$ denotes the set of ``tree points'';
            \item[-] $\mathcal{S}_A^{~int}$ denotes the set of ``intermediate points''; 
            \item[-] $\mathcal{S}_A^{~zero}$ denotes the set of ``zero points'';
        \end{itemize}
    The transitions between distinct points from $\mathcal{S}_A^{~tree}$ will be referred to as tree transitions.
    \item \textbf{Copy and anchor points.} Given a point $b$ in $M^0_A$, we define the 2-image of $b$ as the maximal 2-path $a_0,a_1,\dots,a_n=b$ (with $n\geq 0$) composed of pairwise distinct states in $\mathcal S^0_A$:
    \[ a_0\, R^{i_1}_{\alpha_1}\, a_1\, R^{i_2}_{\alpha_2}\, \cdots \, R^{i_n}_{\alpha_n}\, a_n \qquad (\alpha_k\geq2 \text{ for all } k\leq n). \]
    We call $a_0$ the anchor of $b$ (which we denote $b^\top$), and we call $b$ an anchor point if $b^\top=b$.

    We say that a point is an $R^2$ point if there is an $R^2_\alpha$ transition towards it in $\mathcal R^0_A$. For each $R^2$ point $b$ we add an extra point $b^\dwn$ which we call a ``copy point'' of $b$. We call the set of copy points $\mathcal S_A^{~copy}$, and define $\mathcal S_A=\mathcal S^0_A\cup\mathcal S_A^{~copy}$. We then expand the set of relations as follows:
    \[ \mathcal R^{c}_A = \mathcal R^{0}_A\cup\{b^\dwn R^1_\omega b,\;b^\top R^\nabla b^\dwn,\; b^\dwn R^\dwn b^\top \mid b\text{ is an } R^2\text{ point}\}. \]
    Denote the new structure as $M^c_A=(r_A,\mathcal S_A,\mathcal R^c_A)$ (informally the pre-model with copy points).

    \item \textbf{Closure.} $\mathcal{R}_A$ is the closure of $\mathcal{R}_A^c$ under the following rules (the \emph{internal closure procedure}):
        \begin{multicols}{2}
        \begin{itemize}
            \item Monotonicity:
             \begin{enumerate}
                \item If $aR^i_{\alpha} b$ add $aR^i_\beta b$ for all $i\leq \beta<\alpha$;
                \item If $aR^2_\alpha b$ add $aR^1_\alpha b$;
                \item Add $aR^\nabla a$ for all $a$;
                \item If $aR^1_\alpha b$ add $aR^\nabla b$;
                \item If $aR^\dwn b$ add $aR^\nabla b$;
            \end{enumerate}
            \item Transitivity:
            \begin{enumerate}
            \setcounter{enumii}{5}
                \item If $aR^i_\alpha b R^j_\alpha c$ add $a R^j_\alpha c$;
                \item If $aR^\nabla bR^\nabla c$ add $aR^\nabla c$;
                \item If $aR^\nabla bR^1_\alpha c$ add $aR^1_1 c$;
                \item If $aR^1_\alpha b R^\nabla c$ add $aR^1_1 c$;
                \item If $aR^\dwn b R^\nabla c$ add $aR^\dwn c$;
            \end{enumerate}
            \item Conjunction:
            \begin{enumerate}
            \setcounter{enumii}{10}
                \item If $aR^i_\alpha b$ and $aR^j_\beta c$ add $b R^j_\beta c$,\\ when either $\alpha>\beta$ or $\alpha=\beta=\omega$ holds;
                \item If $aR^2_\alpha b$ and $aR^i_\beta c$ add $b R^\dwn c$;
            \end{enumerate}
        \end{itemize}
        \end{multicols}
    \end{enumerate}
\end{definition}

We record the following facts about paths in $M^0_A$ that are used throughout the following proofs:
\begin{enumerate}[label=(P\arabic*)]
    \item Every $\alpha$-path is a $\beta$-path for $\beta\leq\alpha$, and the concatenation of two $\alpha$-paths is an $\alpha$-path.
    \item If there is a proper $\alpha$-path from $a$ to $b$ in $M^0_A$, then $aR^1_\alpha b\in\mathcal R_A$.
    \item In $M^0_A$ a point has at most one incoming transition of rank $\geq2$ apart from its own loop.
    \item Every point $b$ of $M^0_A$ has a unique 2-image, and every point admitting a 2-path to $b$ in $M^0_A$ lies on the 2-image of $b$.
    \item For an $R^2$ point $b$ the 2-image has $n\geq1$, consists of tree points, and ends with the $R^2$ tree transition into $b$.
\end{enumerate}
(P1) and (P2) are easy to prove, and (P4) and (P5) are direct consequences of (P3). 

For (P3): for a tree point this is its incoming tree transition, for a zero or intermediate point $e$ of $r_{A'}$ it is $r_{A'}R^1_\omega e$. Consequently, if there are proper $2$-paths from $b$ to $a$ and from $b'$ to $a$ with $b\neq b'$, then there is a $2$-path from $b$ to $b'$ or from $b'$ to $b$; and if there is a $2$-path from $b$ to $b'$, then every $2$-path from $b$ to $a$ passes through $b'$ and splits into a $2$-path from $b$ to $b'$ followed by a $2$-path from $b'$ to $a$. We also note that a proper $2$-path from a point to itself consists of loops, and that the anchor points are exactly $r_A$ and the tree points whose incoming tree transition has rank $1$.

\begin{proposition}\label{prop Mmodel}
For any variable-free $\mathsf{RC_{BI_\omega}}$ formula $A$, we have $M_A,r_A\models A$.
\end{proposition}
\begin{proof}
The tree points and tree transitions of $M^0_A$ form an isomorphic copy of $T_A$ rooted at $r_A$, and the further steps of the construction only add points and transitions. Hence $T_A\hookrightarrow_r M_A$, and the claim follows from the easy direction of Proposition \ref{prop embed}.
\end{proof}
\begin{definition}\label{def connector}
    Let $A$ be variable-free and $\alpha\geq2$. Given points $a$ and $b$ of $M^0_A$, a point $c$ is called an $\alpha$-connector of $a$ to $b$ if there is a $(1+\alpha)$-path from $c$ to $a$ and a proper $\alpha$-path from $c$ to $b$ in $M^0_A$. If $c$ is an $\alpha$-connector of $a$ to $b$, we say that $c$ connects $a$ to $b$ at $\alpha$. 
\end{definition}
\begin{lemma}\label{lem rank2 paths}
Let $A$ be variable-free, $\alpha\geq2$ and $aR^i_\alpha b\in\mathcal R_A$.
\begin{enumerate}[label=(\roman*)]
    \item If $a$ is not a copy point, then there exists an $\alpha$-connector $c$ of $a$ to $b$.
    \item If $a$ is a copy point $d^\dwn$, then either $b=d$, or there is an $\alpha$-connector $c$ of $d$ to $b$.
    \item If $i=2$, then in both (i) and (ii) the connector $c$ can be chosen such that $c\neq b$. Also, in (ii) the connector $c\neq b$ exists even when $b=d$.
\end{enumerate}
In particular, no transition of rank $\geq2$ of $\mathcal R_A$ is towards a copy point, $r_A$, or a tree point whose incoming tree transition has rank $1$.
\end{lemma}
\begin{proof}
First we note that a point connecting $a$ to $b$ at $\alpha$ connects them at every $\alpha'\leq\alpha$, and nothing connects any point to a copy point, since no transition of $\mathcal R^0_A$ enters a copy point. 

We argue by induction on the stage at which $aR^i_\alpha b$ enters the closure.

The transitions of rank $\geq2$ in $\mathcal R^c_A$ are those of $\mathcal R^0_A$, for which $a$ connects $a$ to $b$ (empty path to $a$, the transition itself to $b$), and the transitions $d^\dwn R^1_\omega d$, for which $b=d$. Rules (a) and (b) only lower the rank or change the upper index, so the claim is inherited from their premise. The only other rules producing transitions of rank $\geq2$ are (f) and (k).

Rule (f): $aR^i_\alpha dR^j_\alpha b$, where $d$ is not a copy point by the induction hypothesis. Suppose first that $a$ is not a copy point, and let $c$ connect $a$ to $d$ and $c'$ connect $d$ to $b$, both at $\alpha$. Both admit a $2$-path to $d$, so one admits a $2$-path to the other.
\begin{itemize}
    \item If there is a $2$-path from $c'$ to $c$, then $c'$ connects $a$ to $b$. By the choice of $c'$ there is a proper $\alpha$-path from $c'$ to $b$. The $(1+\alpha)$-path from $c'$ to $d$ splits at $c$, and its initial part followed by the $(1+\alpha)$-path from $c$ to $a$ is a $(1+\alpha)$-path from $c'$ to $a$.
    \item Otherwise there is a proper $2$-path from $c$ to $c'$, and $c$ connects $a$ to $b$. By the choice of $c$ there is a $(1+\alpha)$-path from $c$ to $a$. The $\alpha$-path from $c$ to $d$ splits at $c'$, and its initial part followed by the proper $\alpha$-path from $c'$ to $b$ is a proper $\alpha$-path from $c$ to $b$.
\end{itemize}
If $a=d'^{\dwn}$ is a copy point, then $d=d'$ or some point connects $d'$ to $d$ at $\alpha$. In the first case the induction hypothesis for $dR^j_\alpha b$ is what we need. In the second the argument above, with $d'$ in place of $a$, gives a point connecting $d'$ to $b$.

Rule (k): $dR^{i'}_\gamma a$ and $dR^j_\alpha b$ with $\gamma\geq1+\alpha$, producing $aR^j_\alpha b$. By the induction hypothesis for $dR^{i'}_\gamma a$, the point $a$ is not a copy point. Note that a $\gamma$-path is a $(1+\alpha)$-path, and a $(1+\gamma)$-path is an $\alpha$-path. Suppose first that $d$ is not a copy point, and let $c$ connect $d$ to $a$ at $\gamma$ and $c'$ connect $d$ to $b$ at $\alpha$. Both admit a $2$-path to $d$.
\begin{itemize}
    \item If there is a $2$-path from $c$ to $c'$, then $c$ connects $a$ to $b$ at $\alpha$: the proper $\gamma$-path from $c$ to $a$ is a $(1+\alpha)$-path. The $(1+\gamma)$-path from $c$ to $d$ splits at $c'$ and its initial part is an $\alpha$-path, so followed by the proper $\alpha$-path from $c'$ to $b$ it gives a proper $\alpha$-path from $c$ to $b$.
    \item Otherwise there is a proper $2$-path from $c'$ to $c$, and $c'$ connects $a$ to $b$ at $\alpha$. By the choice of $c'$ there is a proper $\alpha$-path from $c'$ to $b$. The $(1+\alpha)$-path from $c'$ to $d$ splits at $c$, and its initial part followed by the proper $\gamma$-path from $c$ to $a$ is a $(1+\alpha)$-path from $c'$ to $a$.
\end{itemize}
Now suppose $d=d'^{\dwn}$ is a copy point, so that $a=d'$ or some point connects $d'$ to $a$ at $\gamma$, and $b=d'$ or some point connects $d'$ to $b$ at $\alpha$.
\begin{itemize}
    \item If $a=b=d'$, then $d'$ connects itself to itself at $\alpha$: the loop $d'R^1_\omega d'$ belongs to $\mathcal R^0_A$ since $d'$ is an $R^2$ point, and is a proper $\alpha$-path.
    \item If $a=d'\neq b$, the point connecting $d'$ to $b$ is what we need.
     \item If $a\neq d'=b$, let $c$ connect $d'$ to $a$ at $\gamma$. Then $c$ connects $a$ to $b$ at $\alpha$: the proper $\gamma$-path from $c$ to $a$ is a $(1+\alpha)$-path, and the $(1+\gamma)$-path from $c$ to $d'$ followed by the loop at $d'$ is a proper $\alpha$-path.
    \item If neither $a$ nor $b$ is $d'$, argue as in the non-copy case with $d'$ in place of $d$.
\end{itemize}

For (iii) we check, along the same induction, that for a transition with upper index $2$ the connector obtained is different from $b$. Recall that a proper $2$-path from a point to itself consists of loops. Hence if $c'\neq b$ admits a $2$-path to $b$ and $c$ admits a $2$-path to $c'$, then $c\neq b$. 

The $R^2$ transitions of $\mathcal R^c_A$ are tree transitions $aR^2_\alpha b$, which are connected by $a\neq b$. Rule (a) inherits the claim from its premise and rule (b) only produces $R^1$ transitions.
 
Rule (f) with $j=2$: The induction hypothesis for $dR^2_\alpha b$ provides $c'\neq b$ connecting $d$ to $b$. The connector produced above is $c'$ or a point $c$ admitting a $2$-path to $c'$, hence different from $b$. If $a=d'^\dwn$ and $d=d'$, the point $c'$ itself is as required.

Rule (k) with $j=2$: If $d$ is not a copy point, the induction hypothesis for $dR^2_\alpha b$ provides $c'\neq b$, and the connector produced above is again $c'$ or a point admitting a $2$-path to $c'$. If $d=d'^\dwn$, the induction hypothesis provides $c'\neq b$ connecting $d'$ to $b$ at $\alpha$, whether or not $b=d'$. If $a=d'$ we are done. Otherwise let $c$ connect $d'$ to $a$ at $\gamma$ and argue as in the non-copy case with $d'$ in place of $d$: if $c$ admits a $2$-path to $c'$, then $c$ connects $a$ to $b$ at $\alpha$ and $c\neq b$. Otherwise $c'$ connects $a$ to $b$ at $\alpha$.

The last claim of the lemma holds because the path from the connecting point to $b$ is proper, and no transition of $\mathcal R^0_A$ of rank $\geq2$ enters a copy point, $r_A$, or a tree point whose incoming tree transition has rank $1$.
\end{proof}

\begin{lemma}\label{lem closure facts}
Let $A$ be a variable-free $\mathsf{RC_{BI_\omega}}$ formula.
\begin{enumerate}[label=(\roman*)]
    \item The closure rules never add an $R^2_\alpha$ transition towards a point that is not an $R^2$ point of $M^0_A$. Hence the $R^2$ points of $M_A$ are those of $M^0_A$, and each has a copy point and an anchor.
        \item For every point $d$ of $M^0_A$ which is not an anchor point: $d^\top R^1_2d$. If moreover $d$ is an $R^2$ point, then $d^\top R^2_2d$.
    \item If $aR^2_\alpha b\in\mathcal R_A$, then one of the following holds:
    \begin{enumerate}[label=(\arabic*)]
        \item $a=b^\top$;
        \item $a$ is not a copy point, $a\neq b^\top$ and $b^\top R^1_1a$;
        \item $a=d^\dwn$ for an $R^2$ point $d$ with $d^\top=b^\top$.
    \end{enumerate}
\end{enumerate}
\end{lemma}
\begin{proof}
(i) Rules (a) and (b) keep the target of the transition, rules (f) and (k) keep the target of their $R^2$ premise, and no other rule produces $R^2$ transitions.

(ii) The 2-image of $d$ is a proper $2$-path from $d^\top$ to $d$, so $d^\top R^1_2d$ by (P2). If $d$ is an $R^2$ point, then by (P5) the last transition of the 2-image is an $R^2$ transition. By (a) and (b) each transition of the 2-image yields an $R^1_2$ transition, the last one an $R^2_2$ transition, and (f) composes them to $d^\top R^2_2d$.

(iii) We use two observations.
\begin{itemize}
    \item By (P4), every point admitting a $2$-path to $b$ lies on the 2-image of $b$, and is therefore reached from $b^\top$ by a $2$-path.
    \item If there is a proper $2$-path from $b^\top$ to a point, then $b^\top$ is $R^1_1$-related to that point.
\end{itemize}
Now let $aR^2_\alpha b$. If $a$ is not a copy point, Lemma \ref{lem rank2 paths}(i) gives a point $c$ with a $(1+\alpha)$-path to $a$ and a proper $\alpha$-path to $b$; by the first observation $c$ lies on the 2-image of $b$, so there is a $2$-path from $b^\top$ to $a$. If it is empty we are in case (1), otherwise in case (2) by the second observation. If $a=d^\dwn$ is a copy point, Lemma \ref{lem rank2 paths}(ii) gives either $d=b$, whence $d^\top=b^\top$ and we are in case (3), or a point $c$ on the 2-image of $b$ with a $(1+\alpha)$-path to $d$, that is, one with all ranks $\geq3$ or all equal to $\omega$. Since $d$ is a tree point, this path consists of tree transitions apart from loops, so the tree path from $b^\top$ to $d$ has all ranks $\geq2$; as no transition of rank $\geq2$ enters $b^\top$, the 2-image of $d$ starts at $b^\top$, that is, $d^\top=b^\top$, and we are in case (3).
\end{proof}

\begin{lemma}\label{lem locality}
Let $A$ be a variable-free formula.
\begin{enumerate}[label=(\roman*)]
   \item For every anchor point $t$ of $M_A$ there is a formula $B$ with $r_B=t$ such that either $A=B$, or $\dmnd{1}{1}B$, ${\dwn}B$ or $\nabla B$ is a sub-formula of $A$. For every such $B$:
    \begin{itemize}
        \item every transition of $\mathcal R^0_A$ from a point outside $\mathcal S^0_B$ to a point of $\mathcal S^0_B$ has rank $1$, and no transition of $\mathcal R^0_A$ leads from a point of $\mathcal S^0_B$ to a point outside of it;
        \item if $a\in\mathcal S_B$ and $aRb\in\mathcal R_A$, then $b\in\mathcal S_B$ and $aRb\in\mathcal R_B$;
        \item for every $b\in\mathcal S_A$, we have $b\in\mathcal S_B$ if and only if $r_BR^\nabla b\in\mathcal R_A$;
        \item the restriction of $M_A$ to $\mathcal S_B$ is $M_B$, and for every $a\in\mathcal S_B$ and every formula $C$ we have $M_A,a\models C$ iff $M_B,a\models C$.
    \end{itemize}
    \item If $b$ is an $R^2$ point and $b^\dwn R^\dwn a\in\mathcal R_A$, then $b^\top R^\nabla a\in\mathcal R_A$.
    \item Let $c$ be an $R^2$ point. Then:
    \begin{itemize}
        \item $aR^1_1c^\dwn$ implies $aR^1_1c^\top$;
        \item $aR^\dwn c^\dwn$ implies $aR^\dwn c^\top$;
        \item $aR^\nabla c^\dwn$ implies $a=c^\dwn$ or $aR^\nabla c^\top$.
    \end{itemize}
\end{enumerate}
\end{lemma}
\begin{proof}
(i) If $t=r_A$ we take $B=A$. Otherwise, by the path facts following Definition \ref{def 12 trees 3}, $t$ is a tree point whose incoming tree transition has rank $1$, so it was created by a clause $\dmnd{1}{1}B$, ${\dwn}B$ or $\nabla B$ of Definition \ref{def 12 trees 3}(1), and $t=r_B$. 

Now fix such a $B$. If $B=A$ only the third claim needs proof, and the argument below applies. Assume $B\neq A$. We note that $\mathcal S_B\subseteq\mathcal S_A$: the points of $M^0_B$ are points of $M^0_A$ by construction, and the $R^2$ points of $M^0_B$ are $R^2$ points of $M^0_A$, so $M_B$ and $M_A$ share their copy points.

For the first claim, the transitions of $\mathcal R^0_A$ are the tree transitions, the loops at $R^2$ points, and the transitions $r_{A'}R^1_\omega r^0_{A'}$, $r_{A'}R^1_\omega e^i_{A'}$ and $e^i_{A'}R^\dwn b_i$ for the sub-formula occurrences $A'={\dwn}B'$ of $A$. Among these, the only ones from a point outside $\mathcal S^0_B$ to a point of $\mathcal S^0_B$ are the tree transition into $r_B$, which has rank $1$ by the choice of $B$, and the transitions $e^i_{A'}R^\dwn b_i$ with $e^i_{A'}\notin\mathcal S^0_B$ and $b_i\in\mathcal S^0_B$. None of them leads from a point of $\mathcal S^0_B$ to a point outside of it. In particular, the 2-image of a point of $\mathcal S^0_B$ is the same in $M^0_B$ and in $M^0_A$, so the $R^2$ points, anchors and copy points of $M_B$ are those of $M_A$ lying in $\mathcal S_B$, and the restriction of $\mathcal R^c_A$ to sources in $\mathcal S_B$ is exactly $\mathcal R^c_B$. Since the closure rules are the same for $M_A$ and $M_B$, this also gives $\mathcal R_B\subseteq\mathcal R_A$.

We prove the second claim by induction on the stage at which $aRb$ enters the closure. The base case is the observation above.
\begin{itemize}
    \item Rules (a), (b), (d), (e): the conclusion has the same source as the premise, so the induction hypothesis applies to the premise, and the same rule yields the conclusion in $\mathcal R_B$.
    \item Rule (c): immediate, since it has no premise.
    \item Rules (f)--(j): the premises are $aRd$ and $dR'b$. The induction hypothesis for $aRd$ gives $d\in\mathcal S_B$ and $aRd\in\mathcal R_B$, then the induction hypothesis for $dR'b$ gives $b\in\mathcal S_B$ and $dR'b\in\mathcal R_B$, and the same rule yields $aRb\in\mathcal R_B$.
    \item Rules (k) and (l): the premises are $dR^i_\alpha a$ with $\alpha\geq2$ and $dR'b$. It suffices to show that $d\in\mathcal S_B$, since then the induction hypothesis applies to both premises. By Lemma \ref{lem rank2 paths}, $a$ is not a copy point, and either $d=a^\dwn$, or there is an $\alpha$-connector $c$ of $d$ to $a$ (if $d$ is not a copy point), or an $\alpha$-connector $c$ of $d'$ to $a$ (if $d=d'^\dwn$). In the first case $d\in\mathcal S_B$, as $a\in\mathcal S^0_B$. In the other two cases, $c$ admits a proper $2$-path to $a\in\mathcal S^0_B$ in $M^0_A$. By the first claim, this path lies in $\mathcal S^0_B$, so $c\in\mathcal S^0_B$, and the $(1+\alpha)$-path from $c$ to $d$, respectively to $d'$, lies in $\mathcal S^0_B$ as well. Hence $d\in\mathcal S_B$.
\end{itemize}

For the third claim, if $r_BR^\nabla b\in\mathcal R_A$, then $b\in\mathcal S_B$ by the second claim, as $r_B\in\mathcal S_B$. Conversely, let $b\in\mathcal S_B$. By construction, $b$ is reachable from $r_B$ by a path in $M^c_B$, and rules (b), (c), (d), (e) and (g) turn this path into $r_BR^\nabla b\in\mathcal R_B$. Since $\mathcal R_B\subseteq\mathcal R_A$, we get $r_BR^\nabla b\in\mathcal R_A$.

For the fourth claim, the transitions of $\mathcal R_A$ with source in $\mathcal S_B$ belong to $\mathcal R_B$ by the second claim, and $\mathcal R_B\subseteq\mathcal R_A$. Hence the restriction of $M_A$ to $\mathcal S_B$ is $M_B$. Finally, since the transitions from points of $\mathcal S_B$ are the same in both models, $M_A,a\models C$ iff $M_B,a\models C$ for $a\in\mathcal S_B$ follows by induction on $C$.

(ii) By induction on the stage at which $b^\dwn R^\dwn a$ enters the closure. The transition $b^\dwn R^\dwn b^\top$ of $\mathcal R^c_A$ is covered by rule (c). Rule (j) derives $b^\dwn R^\dwn a$ from $b^\dwn R^\dwn d$ and $dR^\nabla a$. Then $b^\top R^\nabla d$ by the induction hypothesis, and $b^\top R^\nabla a$ by (g). Rule (l) only produces $R^\dwn$ transitions whose source is the target of an $R^2$ transition, which is not a copy point by Lemma \ref{lem rank2 paths}.

(iii) By Lemma \ref{lem rank2 paths}, every $R^i_\alpha$ transition of $\mathcal R_A$ into $c^\dwn$ is an $R^1_1$ transition, so rules (a) and (b) do not produce transitions into $c^\dwn$. We prove the three implications simultaneously by induction on the stage. The only transition of $\mathcal R^c_A$ into $c^\dwn$ is $c^\top R^\nabla c^\dwn$, and $c^\top R^\nabla c^\top$ holds by (c). The loop $c^\dwn R^\nabla c^\dwn$ added by (c) is covered by the third implication. For the remaining rules:
\begin{itemize}
    \item Rule (d): the premise is $aR^1_1c^\dwn$. Then $aR^1_1c^\top$ by the induction hypothesis, so $aR^\nabla c^\top$ by (d).
    \item Rule (e): the premise is $aR^\dwn c^\dwn$. Then $aR^\dwn c^\top$, so $aR^\nabla c^\top$ by (e).
    \item Rule (f): the premise is $aR^i_1dR^1_1c^\dwn$. Then $dR^1_1c^\top$, so $aR^1_1c^\top$ by (f).
    \item Rule (h): the premise is $aR^\nabla dR^1_1c^\dwn$. Then $dR^1_1c^\top$, so $aR^1_1c^\top$ by (h).
    \item Rules (g), (i), (j): the premises are $aRd$ and $dR^\nabla c^\dwn$, where $R$ is $R^\nabla$, $R^1_\alpha$ or $R^\dwn$ respectively. If $d=c^\dwn$, then the conclusion is $aR^\nabla c^\dwn$, $aR^1_1c^\dwn$ or $aR^\dwn c^\dwn$, and the claim is the induction hypothesis for the premise $aRd$ (for rule (i), the premise $aR^1_\alpha c^\dwn$ has $\alpha=1$ by Lemma \ref{lem rank2 paths}). Otherwise $dR^\nabla c^\top$ by the induction hypothesis, and the same rule applied to $aRd$ and $dR^\nabla c^\top$ gives the claim.
    \item Rule (k): the premises are $dR^i_\alpha a$ and $dR^1_1c^\dwn$ with $\alpha>1$. Then $dR^1_1c^\top$, so $aR^1_1c^\top$ by (k).
    \item Rule (l): the premises are $dR^2_\alpha a$ and $dR^1_1c^\dwn$. Then $dR^1_1c^\top$, so $aR^\dwn c^\top$ by (l).
\end{itemize}
\end{proof}

\begin{proposition}\label{prop MA is a model}
For any variable-free $\mathsf{RC_{BI_\omega}}$ formula $A$, $M_A$ is an $\mathsf{RC_{BI_\omega}}$ model.
\end{proposition}
\begin{proof}
Conditions I--III correspond to the closure rules of Definition \ref{def 12 trees 3}, so we only need to verify the constructive conditions IV.1--IV.3.

IV.1: $aR^\dwn b$ implies $aR^1_\omega c$ for some $c$. We check that every source of an $R^\dwn$ transition of $\mathcal R_A$ has an outgoing $R^1_\omega$ transition.
\begin{itemize}
    \item Transitions of step (1): these are $r_{A'}R^{\dwn}r_{B'}$ and $e^i_{A'}R^{\dwn}b_i$ for a sub-formula $A'={\dwn}B'$. By construction $r_{A'}R^1_\omega r^0_{A'}$, and applying (k) to $r_{A'}R^1_\omega e^i_{A'}$ with itself (case $\alpha=\beta=\omega$) gives $e^i_{A'}R^1_\omega e^i_{A'}$.
    \item Transitions of step (2): these are $b^\dwn R^\dwn b^\top$ for the $R^2$ points $b$, and $b^\dwn R^1_\omega b$.
    \item Rule (j) adds no new sources of $R^\dwn$ transitions.
    \item Rule (l) adds $R^\dwn$ transitions only from $R^2$ points, which have an $R^1_\omega$ loop.
\end{itemize}

IV.2: $aR^{\dwn}bR^1_\alpha c$ implies $aR^1_\omega dR^{\dwn}c$ for some $d$. We first observe that, by (d) and (j), $aR^{\dwn}bR^1_\alpha c$ gives $aR^{\dwn}c$, so if $aR^1_\omega a$ we may take $d=a$. This covers every $a$ that is an $R^2$ point (whose loop belongs to $\mathcal R^0_A$, by Lemma \ref{lem closure facts}(i)) or an intermediate point (whose loop was obtained in IV.1).

We argue by induction on the stage at which $aR^{\dwn}b$ enters $\mathcal R_A$.
\begin{itemize}
\item $a$ is an intermediate point, or $aR^{\dwn}b$ is added by rule (l): then $aR^1_\omega a$ and we are done by the observation.
\item $a=r_{A'}$ and $b=r_{B'}$ for a sub-formula $A'={\dwn}B'$. By Lemma \ref{lem locality}(i), $c\in\mathcal S_{B'}$ and $r_{B'}R^1_\alpha c\in\mathcal R_{B'}$, hence $r_{B'}R^1_1c\in\mathcal R_{B'}$ by (a). If $c\in\mathcal S^0_{B'}$, then by construction there is an intermediate point $e$ with $r_{A'}R^1_\omega e$ and $eR^{\dwn}c$. Otherwise $c=c_0^{\dwn}$ is a copy point of $M_{B'}$. By Lemma \ref{lem locality}(iii) applied to $M_{B'}$ we have $r_{B'}R^1_1c_0^\top\in\mathcal R_{B'}$, and $c_0^\top\in\mathcal S^0_{B'}$, so there is an intermediate point $e$ with $r_{A'}R^1_\omega e$ and $eR^{\dwn}c_0^\top$. Since $c_0^\top R^\nabla c_0^{\dwn}$, rule (j) gives $eR^{\dwn}c$.
\item $a=b_0^{\dwn}$ is a copy point. By Lemma \ref{lem locality}(ii), $b_0^\top R^\nabla b$, so $b_0^\top R^1_1c$ by (h). Since $b_0^\top R^2_2b_0$ by Lemma \ref{lem closure facts}(ii), rule (l) gives $b_0R^{\dwn}c$. As $b_0^{\dwn}R^1_\omega b_0$, we may take $d=b_0$.
\item $aR^{\dwn}b$ is added by rule (j) from $aR^{\dwn}b'$ and $b'R^\nabla b$, and $a$ is not a copy point. Then $b'R^\nabla bR^1_\alpha c$ gives $b'R^1_1c$ by (h), and the induction hypothesis applied to $aR^{\dwn}b'R^1_1c$ provides $d$.
\end{itemize}
These cases are exhaustive: the $R^{\dwn}$ transitions of $\mathcal R^c_A$ are $r_{A'}R^{\dwn}r_{B'}$, $eR^{\dwn}b_i$ and $b_0^{\dwn}R^{\dwn}b_0^\top$, and the only closure rules producing $R^{\dwn}$ transitions are (j) and (l).\\

IV.3: $aR^2_\alpha b$ implies that there is $c$ with $aR^\nabla c$, $cR^\dwn a$, and $\forall d\forall R\,(bRd\Rightarrow cRd)$. By Lemma \ref{lem closure facts}(i), $b$ has a copy point $b^\dwn$ and an anchor $b^\top$. The point $c$ will be $b$ or $b^\dwn$. We check that $bRd$ implies $b^\dwn Rd$ for all transitions $R$.
\begin{itemize}
    \item For $bR^i_\beta d$: from $b^\dwn R^1_\omega b$ we get $b^\dwn R^1_\beta b$ by (a) and then $b^\dwn R^i_\beta d$ by (f).
       \item For $bR^\nabla d$: we have $b^\top R^\nabla b$ by Lemma \ref{lem closure facts}(ii) and (d), so $b^\dwn R^\dwn b^\top R^\nabla b$ gives $b^\dwn R^\dwn b$ by (j), then $b^\dwn R^\dwn d$ by (j) and $b^\dwn R^\nabla d$ by (e).
    \item For $bR^{\dwn}d$: by (e) we have $bR^\nabla d$, and the previous case already produced $b^{\dwn}R^{\dwn}d$.
\end{itemize}
We now distinguish the same cases as in Lemma \ref{lem closure facts}(iii).
\begin{itemize}
    \item Case (1): $a=b^\top$. We take $c=b^\dwn$. The required $aR^\nabla b^\dwn$ and $b^\dwn R^\dwn a$ are the anchor transitions.
    \item Case (2): $b^\top R^1_1a$. We take $c=b$. Then $aR^\nabla b$ holds by (b), (d), and $bR^\dwn a$ follows by rule (l) from $b^\top R^2_2 b$ and $b^\top R^1_1 a$.
    \item Case (3): $a=d^\dwn$ with $d^\top=b^\top$. We again take $c=b^\dwn$. Then $d^\dwn R^\dwn b^\top R^\nabla b^\dwn$ gives $aR^\nabla b^\dwn$ by (e), (g), and $b^\dwn R^\dwn b^\top R^\nabla d^\dwn$ gives $b^\dwn R^\dwn a$ by (j).
\end{itemize}
\end{proof}

\begin{lemma}\label{lem omega loop}
    For every variable-free $\mathsf{RC_\bio}$ formula $A$, $\dmnd{2}{\alpha}A\vdash_\bio \dmnd{2}{\alpha}(A\wedge\dmnd{1}{\omega} A)$.
\end{lemma}
\begin{proof}
    Since $A$ is variable-free, it can be written as $A=B_A\wedge C_A \wedge D_A$ where:
    \begin{itemize}
        \item[-] $B_A$ is a conjunction of $\dwn$-formulas (i.e. formulas of the form $\dwn B'$);
        \item[-] $C_A$ is a conjunction of $\Diamond$-formulas (i.e. formulas of the form $\dmnd{i}{\alpha} C'$);
        \item[-] $D_A$ is a conjunction of $\nabla$-formulas (i.e. formulas of the form $\nabla D'$).
    \end{itemize}
    If any of $B_A$, $C_A$ or $D_A$ is empty, we read it as $\top$. Then, it is easy to see that the following derivation holds:
    \[\begin{aligned} \dmnd{2}{\alpha}A&\vdash_\bio \dmnd{2}{\alpha} A\wedge \dmnd{2}{\alpha}\dmnd{1}{1} A\\ &\vdash_\bio \dmnd{2}{\alpha}(A\wedge \dwn \dmnd{1}{1} A)\\ &\vdash_\bio \dmnd{2}{\alpha}(A\wedge \dmnd{1}{\omega}\dwn A)\\ &\vdash_\bio \dmnd{2}{\alpha}(A\wedge C_A\wedge \dmnd{1}{\omega}\dwn(B_A\wedge D_A))\\&\vdash_\bio \dmnd{2}{\alpha}(A\wedge\dmnd{1}{\omega}(B_A\wedge C_A\wedge D_A))\\&\vdash_\bio \dmnd{2}{\alpha}(A\wedge\dmnd{1}{\omega} A).\end{aligned}\]
\end{proof}

\begin{lemma}\label{lem loop witness}
Consider the canonical model $W_{\mathsf{BI_\omega}} = (\mathcal{S}_{\mathsf{BI_\omega}}, \mathcal{R}_{\mathsf{BI_\omega}}, \mathcal{V}_{\mathsf{BI_\omega}})$, and let $a\in\mathcal S_{\mathsf{BI_\omega}}$ consist only of variable-free formulas and $\dmnd{2}{\alpha}B\in a$. Then there is $b\in \mathcal{S}_{\mathsf{BI_\omega}}$ with $aR^2_\alpha b$, $bR^1_\omega b$ and $B\in b$.
\end{lemma}
\begin{proof}
Given a sequence $\vec C=C_1,\dots,C_k$ ($k\geq 0$), we define:
\[ G_{\vec C}=B\wedge{\dwn}C_1\wedge\dots\wedge{\dwn}C_k,\qquad
   H_{\vec C}=G_{\vec C}\wedge\dmnd{1}{\omega}G_{\vec C}. \]
Let $\Delta=\{\dmnd{i}{\beta}C\in a\mid\beta<\alpha\}$ and $b$ be the deductive closure of $\Delta\cup\{H_{\vec C}\mid \dmnd{1}{1}C_t\in a \text{ for all } t\leq k\}$. Taking $k=0$ gives $B\in b$.
Notice that the $H$ formulas are directed: $H_{\vec C+\vec C'}\vdash H_{\vec C}$ by L$\wedge$ and Mod$_1$.
Hence every $D\in b$ follows from finitely many
$\dmnd{i_1}{\beta_1}C'_1,\dots,\dmnd{i_l}{\beta_l}C'_l\in\Delta$ and a single $H=G\wedge\dmnd{1}{\omega}G$.
Moreover, by Axiom 12 with $\alpha=\beta=\omega$,
\[ H\vdash\dmnd{1}{\omega}G\wedge\dmnd{1}{\omega}G
    \vdash\dmnd{1}{\omega}(G\wedge\dmnd{1}{\omega}G)=\dmnd{1}{\omega}H. \tag{$*$} \]
We now only need to check that $bR^1_\omega b$ and $aR^2_\alpha b$:
\begin{itemize}
\item $bR^1_\omega b$: 

A1: let $D\in b$ be as above. By $(*)$ and $l$ applications of Axiom 12,
\[ \bigwedge_{s\le l}\dmnd{i_s}{\beta_s}C'_s\wedge H \vdash\bigwedge_{s\le l}\dmnd{i_s}{\beta_s}C'_s\wedge\dmnd{1}{\omega}H \vdash\dmnd{1}{\omega}\bigl(\bigwedge_{s\le l}\dmnd{i_s}{\beta_s}C'_s\wedge H\bigr) \vdash\dmnd{1}{\omega}D, \]
so $\dmnd{1}{\omega}D\in b$.

A2: Trivially holds. 

\item $aR^2_\alpha b$:

A1: from $\dmnd{2}{\alpha}B\in a$ and $\dmnd{1}{1}C_t\in a$, $k$ applications of
Axiom 11 give $\dmnd{2}{\alpha}G\in a$; as $G$ is variable-free, Lemma \ref{lem omega loop} gives
$\dmnd{2}{\alpha}H\in a$; $l$ applications of Axiom 12 (all ranks in $\Delta$ are $<\alpha$) give
$\dmnd{2}{\alpha}(\bigwedge_{s}\dmnd{i_s}{\beta_s}C'_s\wedge H)\in a$, whence $\dmnd{2}{\alpha}D\in a$ by Mod$_1$.

A2: this holds as $\Delta\subseteq b$.

A3: if $\dmnd{j}{\beta}C\in a$, then $\dmnd{1}{1}C\in a$ by Axioms 2 and 1, so $H_{C}\in b$ and hence ${\dwn}C\in b$.
\end{itemize}
\end{proof}

\begin{lemma}\label{lem 11 trees 3}
Let $A$ be a variable-free formula and let $a\in\mathcal S_{\mathsf{BI_\omega}}$ consist of variable-free formulas. Then
\[ W_{\mathsf{BI_\omega}},a\models A \;\;\text{iff}\;\; \text{there is an embedding } f:M_A\hookrightarrow W_{\mathsf{BI_\omega}} \text{ with } f(r_A)=a. \]
\end{lemma}
\begin{proof} 
The reverse implication is immediate by Propositions \ref{prop embedding} and \ref{prop Mmodel}. 

For the direct implication, note first that every point $b$ with $aRb\in\mathcal R_{\mathsf{BI_\omega}}$ consists of variable-free formulas: by conditions A1, B1 and C1 of Definition \ref{def adequate 3}, each formula of $b$ occurs, prefixed by a modality, in $a$. Hence Lemma \ref{lem loop witness} is available at every point reachable from $a$. Note further that it suffices to construct an embedding $f$ of $M^c_A$ into $W_{\mathsf{BI_\omega}}$ with $f(r_A)=a$: by induction on the stage, the image under $f$ of every transition added by a closure rule of Definition \ref{def 12 trees 3} is a transition of $W_{\mathsf{BI_\omega}}$, since each rule (a)--(l) is an instance of one of the conditions I--III satisfied by $W_{\mathsf{BI_\omega}}$ (Lemma \ref{lem adequate 1}). We define $f$ on the four kinds of points of $M^c_A$ in turn.
\begin{itemize}
    \item Tree points: We define $f$ on $T_A$ by induction on $A$ as in the proof of Proposition \ref{prop embed}, with one modification: for a conjunct $\dmnd{2}{\alpha}B$ of the formula at hand, satisfied at the current image point $f(t)$, we choose as witness for the root of $T_B$ the point $b$ provided by Lemma \ref{lem loop witness}, so that $f(t)R^2_\alpha b$, $bR^1_\omega b$ and $W_{\mathsf{BI_\omega}},b\models B$ by Lemma \ref{lem 7 adequate 2}. All tree transitions are thus matched, and so are the loops $bR^1_\omega b$ at the $R^2$ points.
    \item Zero points: For a sub-formula $A'={\dwn}B'$ we have $f(r_{A'})R^\dwn f(r_{B'})$, so condition IV.1 provides $c$ with $f(r_{A'})R^1_\omega c$; define $f(r^0_{A'})=c$.
    \item Copy points: Let $b$ be an $R^2$ point with 2-image $a_0,\dots,a_n=b$. The images $f(a_0),\dots,f(b)$ are connected by transitions of rank $\geq2$, the last one an $R^2$ transition, so $f(a_0)R^2_2f(b)$ by conditions I.1, I.2 and II.1. Condition IV.3 provides $c$ with $f(a_0)R^\nabla c$, $cR^\dwn f(a_0)$ and $\forall d\,\forall R\,(f(b)Rd\Rightarrow cRd)$; define $f(b^\dwn)=c$. The anchor transitions $b^\top R^\nabla b^\dwn$ and $b^\dwn R^\dwn b^\top$ are matched by the first two properties, and $b^\dwn R^1_\omega b$ is matched because $f(b)R^1_\omega f(b)$ and the third property of $c$.
    \item Intermediate points: We proceed by induction on the depth of the sub-formula $A'={\dwn}B'$, innermost first. By the induction hypothesis $f$ is already defined on $\mathcal S_{B'}$ and is an embedding of $M^c_{B'}$, since $\mathcal R^c_{B'}\subseteq\mathcal R^c_A$; by the observation above it is then an embedding of $M_{B'}$. Hence $r_{B'}R^1_1b_i\in\mathcal R_{B'}$ gives $f(r_{B'})R^1_1f(b_i)$, and condition IV.2 applied to $f(r_{A'})R^\dwn f(r_{B'})R^1_1f(b_i)$ provides $c$ with $f(r_{A'})R^1_\omega cR^\dwn f(b_i)$; define $f(e^i_{A'})=c$.
\end{itemize}
Every transition of $\mathcal R^c_A$ has now been matched, so $f$ is an embedding of $M^c_A$, hence of $M_A$.
\end{proof} 

\begin{theorem}\label{thm trees}
    For variable-free formulas $A$ and $B$,
    \[ A \vdash_{\mathsf{BI_\omega}} B \text{ iff } M_A, r_A \models B. \]
\end{theorem}
\begin{proof}
    If $A\vdash_{\mathsf{BI_\omega}}B$, then $M_A,r_A\models B$ by Proposition \ref{prop Soundness}, Proposition \ref{prop MA is a model} and Proposition \ref{prop Mmodel}.

    Conversely, suppose $M_A,r_A\models B$ and let $a\in\mathcal S_{\mathsf{BI_\omega}}$ be the deductive closure of $A$. Every formula of $a$ is variable-free: a formula derivable from $A$ is satisfied at $r_A$ in $M_A$ by Proposition \ref{prop Soundness}, Proposition \ref{prop MA is a model} and Proposition \ref{prop Mmodel}, while all variables are false in $M_A$. Since $A\in a$, Lemma \ref{lem 7 adequate 2} gives $W_{\mathsf{BI_\omega}},a\models A$, so by Lemma \ref{lem 11 trees 3} there is an embedding $f:M_A\hookrightarrow W_{\mathsf{BI_\omega}}$ with $f(r_A)=a$. By Proposition \ref{prop embedding} we get $W_{\mathsf{BI_\omega}},a\models B$, hence $B\in a$ by Lemma \ref{lem 7 adequate 2}, that is, $A\vdash_{\mathsf{BI_\omega}}B$.
\end{proof}

\begin{corollary}\label{cor trees}
    For variable-free formulas $A$ and $B$,
    \[ A \vdash_{\mathsf{BI_\omega}} B \text{ iff } T_B \hookrightarrow_r M_A. \] 
\end{corollary}
\begin{proof}
    Follows directly by applying Proposition \ref{prop embed} to Theorem \ref{thm trees}.
\end{proof}
\vspace{\baselineskip}
\subsection{Structural properties of $M_A$}\label{subsec:structural MA}
In this subsection we show that certain kinds of transitions of $M_A$ occur only between points of specific types. These results are used in the proof of Lemma \ref{lem 12 tree approx 1}. Throughout, $A$ is a variable-free formula (in Lemmas \ref{lem omega} and \ref{lem S} moreover $A\in\W$). Also, for a tree point $r$ of $M_A$, we let $Z(r)$ be the set of zero and intermediate points $e$ with $rR^1_\omega e\in\mathcal R^0_A$.

\begin{lemma}\label{lem nabla cluster}
Let $t$ be an anchor point of $M_A$ and let
\[ N_t=\{t\}\cup\{c^\dwn\mid c\text{ is an }R^2\text{ point with }c^\top=t\}. \]
If $aRb\in\mathcal R_A$ with $tR^\nabla a$ and $b\in N_t$, then $R$ is $R^\nabla$ or $R^\dwn$, and $a\in N_t$.

In particular, if $dR^1_1d\in\mathcal R_A$, then $d$ is neither an anchor point nor a copy point. Hence $d^\top\neq d$ and $d^\top R^1_2d$.
\end{lemma}

\begin{proof}
Let $B$ be as in Lemma \ref{lem locality}(i) with $r_B=t$. By that lemma, the points $a$ with $tR^\nabla a$ are those of $\mathcal S_B$, the transitions of $\mathcal R_A$ from these points are those of $\mathcal R_B$, and $N_t$ is the same set in $M_A$ and in $M_B$. Hence it suffices to prove the claim for $M_B$ and its root, so we may assume that $t=r_A$, in which case $tR^\nabla a$ holds for every point $a$.

We argue by induction on the stage at which $aRb$ enters the closure. Since $r_A$ has no incoming transition in $\mathcal R^0_A$, and $c^\top R^\nabla c^\dwn$ is the only transition of $\mathcal R^c_A$ into $c^\dwn$, the transitions of $\mathcal R^c_A$ into $N_t$ are $tR^\nabla c^\dwn$ and $c^\dwn R^\dwn t$ for the $R^2$ points $c$ with $c^\top=t$, and these have the required form. For the closure rules:
\begin{itemize}
    \item Rules (a), (b), (d), (f), (h), (k) and (l) have a premise which is a diamond transition into $b$ (that is, a transition $R^i_\alpha$), which is excluded by the induction hypothesis.
    \item Rule (c) adds $bR^\nabla b$ with $b\in N_t$.
    \item Rule (e): the premise $aR^\dwn b$ gives $a\in N_t$ by the induction hypothesis.
    \item Rules (g) and (j): from $aRdR^\nabla b$, the induction hypothesis gives $d\in N_t$, and then $a\in N_t$.
    \item Rule (i): from $aR^1_\alpha dR^\nabla b$, the induction hypothesis gives $d\in N_t$, and then the first premise is a diamond transition into $N_t$, which is excluded.
\end{itemize}
For the second claim, suppose that $dR^1_1d$. If $d$ is an anchor point, we take $t=d$: then $tR^\nabla d$ by (c) and $d\in N_t$, so $dR^1_1d$ contradicts the first claim. If $d=c^\dwn$ is a copy point, we take $t=c^\top$: then $tR^\nabla c^\dwn$ is an anchor transition and $c^\dwn\in N_t$, again a contradiction. Hence $d$ is a point of $M^0_A$ which is not an anchor point, and $d^\top R^1_2d$ by Lemma \ref{lem closure facts}(ii).
\end{proof}

\begin{lemma}\label{lem down}
Let $aR^\dwn c\in\mathcal R_A$.
\begin{enumerate}[label=(\roman*)]
    \item The point $a$ is an $R^2$ point, a copy point, an intermediate point, or the root $r_{A'}$ of a sub-formula occurrence $A'={\dwn}B'$ of $A$. More precisely:
    \begin{enumerate}[label=(\arabic*)]
        \item if $a$ is an $R^2$ point, then $a^\top R^1_1c$;
        \item if $a=b^\dwn$ is a copy point, then $b^\top R^\nabla c$;
        \item if $a=e^i_{A'}$ is an intermediate point of a sub-formula occurrence $A'={\dwn}B'$, then $b_iR^\nabla c$, and hence $r_{B'}R^1_1c$;
        \item otherwise, $a=r_{A'}$ for a sub-formula occurrence $A'={\dwn}B'$ of $A$ with $r_{B'}R^\nabla c$.
    \end{enumerate}
    \item If also $cR^1_1a\in\mathcal R_A$, then $a$ is an $R^2$ point.
\end{enumerate}
\end{lemma}

\begin{proof}
The sources of the $R^\dwn$ transitions of $\mathcal R^c_A$ are the tree points $r_{A'}$ for the sub-formula occurrences $A'={\dwn}B'$, the intermediate points, and the copy points. Rule (l) adds $R^\dwn$ transitions only from $R^2$ points, and rule (j) adds no new sources. Hence $a$ is a tree point, an intermediate point or a copy point. In particular $a$ is not a zero point.
\begin{itemize}
    \item Case 1: $a$ is an $R^2$ point. We argue by induction on the stage at which $aR^\dwn c$ enters the closure.
    \begin{itemize}
        \item If $aR^\dwn c\in\mathcal R^c_A$, then $a=r_{A'}$ and $c=r_{B'}$ for a sub-formula occurrence $A'={\dwn}B'$, since $a$ is a tree point. By Lemma \ref{lem closure facts}(ii) and rules (b) and (a) we have $a^\top R^1_1a$, and $aR^\nabla c$ by (e). So $a^\top R^1_1c$ by (i).
        \item If $aR^\dwn c$ is added by rule (j) from $aR^\dwn dR^\nabla c$, then $a^\top R^1_1d$ by the induction hypothesis, and $a^\top R^1_1c$ by (i).
        \item If $aR^\dwn c$ is added by rule (l) from $dR^2_\alpha a$ and $dR^i_\beta c$, then $dR^1_1c$ by (b) and (a), and by Lemma \ref{lem closure facts}(iii) one of the following holds:
        \begin{itemize}
            \item[-] $d=a^\top$. Then $a^\top R^1_1c$.
            \item[-] $a^\top R^1_1d$. Then $a^\top R^1_1c$ by (f).
            \item[-] $d=d'^\dwn$ with $d'^\top=a^\top$. Then $a^\top R^\nabla d$ is an anchor transition, and $a^\top R^1_1c$ by (h).
        \end{itemize}
    \end{itemize}
    \item Case 2: $a=b^\dwn$ is a copy point. This is Lemma \ref{lem locality}(ii).
    \item Case 3: $a=e^i_{A'}$ is an intermediate point. Then $a$ is not an $R^2$ point, so $aR^\dwn c$ is the transition $e^i_{A'}R^\dwn b_i$ or arises from it by applications of rule (j). Hence $b_iR^\nabla c$ by (c) and (g). Since $r_{B'}R^1_1b_i$, we get $r_{B'}R^1_1c$ by (i).
    \item Case 4: $a$ is a tree point which is not an $R^2$ point. As in the previous case, $aR^\dwn c$ is a transition $r_{A'}R^\dwn r_{B'}$ for a sub-formula occurrence $A'={\dwn}B'$ with $r_{A'}=a$, or arises from such a transition by applications of rule (j). Hence $r_{B'}R^\nabla c$.
\end{itemize}
For (ii), suppose that $aR^\dwn c$ and $cR^1_1a$, and that $a$ is not an $R^2$ point.
\begin{itemize}
    \item If $a=b^\dwn$ is a copy point, then $b^\top R^\nabla cR^1_1b^\dwn$ gives $b^\top R^1_1b^\dwn$ by (h). This is a diamond transition from $b^\top$ into the copy point $b^\dwn$, contradicting Lemma \ref{lem nabla cluster} with $t=b^\top$.
    \item Otherwise, let $A'={\dwn}B'$ be the sub-formula occurrence provided by Case 3 or Case 4. Then $r_{B'}R^\nabla cR^1_1a$ gives $r_{B'}R^1_1a$ by (h). But $a$ is $e^i_{A'}$ or $r_{A'}$, which do not belong to $\mathcal S_{B'}$, contradicting Lemma \ref{lem locality}(i).
\end{itemize}
\end{proof}

\begin{lemma}\label{lem omega}
Let $A\in\W$ and $aR^1_\omega d\in\mathcal R_A$. Let $r$ be the tree point with $a\in\{r\}\cup Z(r)$ if $a$ is not a copy point, and $r=b$ if $a=b^\dwn$. Then $d\in\{r\}\cup Z(r)$, and $d=r$ only if $r$ is an $R^2$ point.
\end{lemma}
\begin{proof}
Since $A$ contains no $\dmnd{1}{\omega}$, the transitions of rank $\omega$ in $\mathcal R^0_A$ are the loops $bR^1_\omega b$ at $R^2$ points $b$ and the transitions $r'R^1_\omega e$ with $e\in Z(r')$, and the points of $Z(r')$ have no outgoing transition of rank $\omega$ in $\mathcal R^0_A$. Hence every $\omega$-path of $M^0_A$ stays inside $\{r'\}\cup Z(r')$ for a single tree point $r'$. If it is proper, it starts at $r'$. If it is proper and ends at $r'$, it consists of loops, so $r'$ is an $R^2$ point.

We apply Lemma \ref{lem rank2 paths} to $aR^1_\omega d$. Let $a'=a$ if $a$ is not a copy point, and $a'=b$ if $a=b^\dwn$, so that $a'\in\{r\}\cup Z(r)$.
\begin{itemize}
    \item If $a=b^\dwn$ and $d=b$, then $b=a'$ is an $R^2$ point, hence a tree point, so $r=b=d$.
    \item Otherwise there is an $\omega$-connector $c$ of $a'$ to $d$, that is, an $\omega$-path from $c$ to $a'$ and a proper $\omega$-path from $c$ to $d$. By the above, $c$ is a tree point $r'$ with $a',d\in\{r'\}\cup Z(r')$. Then $r'=r$, and if $d=r$, then $r$ is an $R^2$ point.
\end{itemize}
\end{proof}

\begin{lemma}\label{lem S}
Let $A\in\W$ and let $b\in\mathcal S_A$ have an outgoing $R^1_\omega$ transition. Then there are a point $a\in\mathcal S_A$ and a non-empty set $T$ of anchor points such that:
\begin{enumerate}[label=(\roman*)]
    \item $bR^i_\beta c$ implies $aR^i_\beta c$, for all $i$, $\beta$ and $c$;
    \item $aR^\dwn t$ for every $t\in T$;
    \item if $bR^1_\omega dR^\dwn c$, then $tR^1_1c$ for some $t\in T$.
\end{enumerate}
\end{lemma}
\begin{proof}
Let $r$ be as in Lemma \ref{lem omega}, with $b$ in place of $a$. By that lemma, every $d$ with $bR^1_\omega d$ lies in $\{r\}\cup Z(r)$, and $d=r$ only if $r$ is an $R^2$ point. We also note that $rR^1_\omega e$ for every $e\in Z(r)$, and that $b=r^\dwn$ if $b$ is a copy point. Finally, if $bR^1_\omega dR^\dwn c$, then $d$ is not a zero point by Lemma \ref{lem down}(i), so $d=r$ or $d=e^i_{A'}$ for a sub-formula occurrence $A'={\dwn}B'$ with $r_{A'}=r$. We distinguish two cases.
\begin{itemize}
    \item Case 1: $r$ is an $R^2$ point. We take $a=r^\dwn$ and $T=\{r^\top\}$.
    \begin{itemize}
        \item (i): we have $r^\dwn R^1_\omega r$, and $b$ is $r^\dwn$, $r$ or a point of $Z(r)$, so $a=b$ or $aR^1_\omega b$ by (f). Hence $bR^i_\beta c$ gives $aR^1_\beta b$ by (a) and $aR^i_\beta c$ by (f).
        \item (ii): $r^\dwn R^\dwn r^\top$ is an anchor transition.
        \item (iii): let $bR^1_\omega dR^\dwn c$. If $d=r$, then $r^\top R^1_1c$ by Lemma \ref{lem down}(1). If $d=e^i_{A'}$ with $r_{A'}=r$, then $r_{B'}R^\nabla c$ by Lemma \ref{lem down}(3) and (d), so $rR^\dwn r_{B'}R^\nabla c$ gives $rR^\dwn c$ by (j), and $r^\top R^1_1c$ by Lemma \ref{lem down}(1).
    \end{itemize}
    \item Case 2: $r$ is not an $R^2$ point. Then $b$ is not a copy point, since $b=b_0^\dwn$ would give $r=b_0$, an $R^2$ point. Hence $b\in\{r\}\cup Z(r)$. Moreover $Z(r)\neq\emptyset$: the $R^1_\omega$ successor of $b$ lies in $\{r\}\cup Z(r)$ and is different from $r$. We take $a=r$ and
    \[ T=\{r_{B'}\mid A'={\dwn}B'\text{ is a sub-formula occurrence of }A\text{ with }r_{A'}=r\}. \]
    The points of $T$ are anchor points, and $T\neq\emptyset$ since $Z(r)\neq\emptyset$.
    \begin{itemize}
        \item (i): $b=r$ or $rR^1_\omega b$, and we conclude as in Case 1.
        \item (ii): $rR^\dwn r_{B'}\in\mathcal R^0_A$ for every $r_{B'}\in T$.
        \item (iii): let $bR^1_\omega dR^\dwn c$. Then $d\neq r$, so $d=e^i_{A'}$ for a sub-formula occurrence $A'={\dwn}B'$ with $r_{A'}=r$, and $r_{B'}R^1_1c$ by Lemma \ref{lem down}(3), where $r_{B'}\in T$.
    \end{itemize}
\end{itemize}
\end{proof}

\begin{lemma}\label{lem guards}
Let $Q\subseteq\mathcal S_A$ be non-empty.
\begin{enumerate}
    \item Let $2\leq\beta<\omega$ and suppose that $bR^1_\beta b'$ for all $b,b'\in Q$. Then there is a point $a$ such that $aR^1_{\beta+1}b$ for every $b\in Q$, and $aR^2_{\beta+1}b$ for every $b\in Q$ for which $b'R^2_\beta b$ holds for some $b'\in Q$.
    \item Suppose that $bR^1_1b'$ for all $b,b'\in Q$. Then all points of $Q$ have the same anchor $a$, and $aR^1_2b$ for every $b\in Q$, and $aR^2_2b$ for every $R^2$ point $b\in Q$.
\end{enumerate}
\end{lemma}
\begin{proof}
1. By Lemma \ref{lem rank2 paths}, $Q$ contains no copy points, since every point of $Q$ is the target of a transition of rank $\beta\geq2$. By (P3) and (P4), the points admitting a $(\beta+1)$-path to a point $b\in Q$ form a final segment of the 2-image of $b$. We define $h(b)$ to be the first point of this segment. Thus every point admitting a $(\beta+1)$-path to $b$ lies on the $(\beta+1)$-path from $h(b)$ to $b$.

We first show that $h(b)=h(b')$ for all $b,b'\in Q$. Let $c$ be a $\beta$-connector of $b$ to $b'$ and $c'$ a $\beta$-connector of $b'$ to $b$, which exist by Lemma \ref{lem rank2 paths}(i). Both $c$ and $c'$ admit $2$-paths to $b$, so one of them admits a $2$-path to the other.
\begin{itemize}
    \item If $c$ admits a $2$-path to $c'$, then the $(\beta+1)$-path from $c$ to $b$ splits at $c'$, and its initial part followed by the $(\beta+1)$-path from $c'$ to $b'$ is a $(\beta+1)$-path from $c$ to $b'$. So $c$ admits $(\beta+1)$-paths to both $b$ and $b'$.
    \item Otherwise $c'$ admits a $2$-path to $c$, and symmetrically $c'$ admits $(\beta+1)$-paths to both $b'$ and $b$.
\end{itemize}
So some point $d$ admits $(\beta+1)$-paths to both $b$ and $b'$. Then $d$ lies on the $(\beta+1)$-path from $h(b)$ to $b$, so $h(b)$ admits a $(\beta+1)$-path to $b'$ through $d$, and therefore $h(b)$ lies on the $(\beta+1)$-path from $h(b')$ to $b'$. By symmetry, $h(b')$ lies on the $(\beta+1)$-path from $h(b)$ to $b$. Hence $h(b)$ and $h(b')$ admit $2$-paths to each other, and since a proper $2$-path from a point to itself consists of loops, $h(b)=h(b')$.

We define $a$ to be this common value, and we let $b\in Q$.
\begin{itemize}
    \item If $a\neq b$, then the $(\beta+1)$-path from $a$ to $b$ is proper, so $aR^1_{\beta+1}b$ by (P2).
    \item If $a=b$, let $c$ be a $\beta$-connector of $b$ to itself. As $c$ admits a $(\beta+1)$-path to $b$ and $h(b)=b$, we have $c=b$, so there is a proper $\beta$-path from $b$ to $b$. It consists of loops, hence $bR^1_\omega b\in\mathcal R^0_A$ and $bR^1_{\beta+1}b$ by (a).
\end{itemize}
Finally, let $b'R^2_\beta b$ with $b,b'\in Q$. By Lemma \ref{lem closure facts}(i), $b$ is an $R^2$ point. By Lemma \ref{lem rank2 paths}(iii) there is a $\beta$-connector $c\neq b$ of $b'$ to $b$. Since $c$ admits a $(\beta+1)$-path to $b'$, the point $a=h(b')$ admits a $2$-path to $c$, and $c$ admits a proper $2$-path to $b$. If $a=b$, these two paths would form a proper $2$-path from $b$ to $b$ through $c\neq b$, which is impossible. So $a\neq b$. Hence the $(\beta+1)$-path from $a$ to $b$ is proper, and its last transition which is not a loop is the $R^2$ tree transition $dR^2_\gamma b$ into $b$, with $\gamma\geq\beta+1$. Then $dR^2_{\beta+1}b$ by (a), and $a=d$ or $aR^1_{\beta+1}d$ by (P2), so $aR^2_{\beta+1}b$ by (f).

2. Every $b\in Q$ satisfies $bR^1_1b$, so by Lemma \ref{lem nabla cluster} no point of $Q$ is a copy point or an anchor point. Let $b,b'\in Q$, and let $B$ be as in Lemma \ref{lem locality}(i) with $r_B=b^\top$. Since $b^\top$ admits a $2$-path to $b$ in $M^0_A$ and no transition of $\mathcal R^0_A$ leaves $\mathcal S^0_B$, we have $b\in\mathcal S^0_B$. So $b'\in\mathcal S_B$ by Lemma \ref{lem locality}(i), and $b'\in\mathcal S^0_B$ as $b'$ is not a copy point. The 2-image of $b'$ lies in $\mathcal S^0_B$, since all transitions of $\mathcal R^0_A$ entering $\mathcal S^0_B$ from outside have rank $1$. So $b'^\top$ is $r_B=b^\top$ or a tree point of $M^0_B$ below it. By symmetry, $b^\top$ is $b'^\top$ or a tree point below $b'^\top$, and therefore $b^\top=b'^\top$. We define $a$ to be the common anchor of the points of $Q$. Since no point of $Q$ is an anchor point, Lemma \ref{lem closure facts}(ii) gives $aR^1_2b$ for every $b\in Q$, and $aR^2_2b$ for every $R^2$ point $b\in Q$.
\end{proof}
\vspace{\baselineskip}
\section{Linearity of $\mathsf{RC_{BI}}$ formulas}\label{section linearity}

In this section, we establish that the $\nabla$-free closed formulas of $\mathsf{RC_{BI}}$ (i.e., the formulas from $\W$) form a linear pre-order with respect to the following relation:

\begin{definition}\label{def order 1}
    Define the $\Pi^1_1$-ordering of $\mathsf{RC_{BI}}$ formulas as:
    \[ A \leq^1_1 B \text{ iff } B \vdash_{\mathsf{BI_\omega}} \dmnd{1}{1} A \text{, or } \nabla A \dashv\vdash_{\mathsf{BI_\omega}} \nabla B. \]
    We write $A <^1_1 B$ iff $B \vdash_{\mathsf{BI_\omega}} \dmnd{1}{1} A$, and $A\equiv^1_1 B$ iff $\nabla A\dashv\vdash_{\mathsf{BI_\omega}}\nabla B$.
\end{definition} 

\begin{proposition}\label{prop trans and compose}
    The relation $<^1_1$ is transitive and composes with $\nabla$-equivalence.
\end{proposition}
\begin{proof}
 Suppose $C\vdash_{\mathsf{BI_\omega}}\dmnd{1}{1}B$ and $B\vdash_{\mathsf{BI_\omega}}\dmnd{1}{1}A$. Then we have $C\vdash\dmnd{1}{1}\dmnd{1}{1}A$, from which, by Cut and Axiom 6, we get $C\vdash\dmnd{1}{1}A$. 
 
 For the second part, notice that $\dmnd{1}{1}$ is invariant under $\nabla$-equivalence, since for any formula $T$ it holds that $\dmnd{1}{1}T\dashv\vdash_{\mathsf{BI_\omega}}\dmnd{1}{1}\nabla T$ (Axiom 3 with Mod$_1$ in one direction, Axiom 9 in the other). Hence $\nabla C\dashv\vdash\nabla C'$ and $C<^1_1 A$ imply $C'<^1_1A$.
\end{proof}

The main idea for the proof of linearity is to reduce comparisons of arbitrary formulas to the existence of embeddings of finite models by using the results on satisfiability and derivability developed in the previous section.

\begin{definition}\label{def fundamental}
    A formula is called fundamental if the only modality occurring in it is $\dmnd{1}{1}$. More explicitly, it is a formula generated by the following grammar: 
\[A := \top \mid p \mid A \wedge B \mid \dmnd{1}{1} A.\]
\end{definition}

\begin{lemma}\label{lem fundamental}
    Suppose $A$ is a fundamental formula in $\W$. Then it has the following properties:
    \begin{enumerate}
        \item There is a fundamental formula $B$, built using only $\top$ and $\dmnd{1}{1}$, such that $A\dashv\vdash_\mathsf{BI_\omega}B$.
        \item If $B$ is also a fundamental formula in $\W$, then either $A\geq^1_1 B$ or $B\geq^1_1 A$.
        \item If $B$ is a non-fundamental formula in $\W$, then $B>^1_1 A$.
    \end{enumerate}
\end{lemma}

\begin{proof}
Let $\idmnd{1}{1}{n} A$ represent the formula $\underbrace{\dmnd{1}{1}\dmnd{1}{1}\dots\dmnd{1}{1}}_n A$. Throughout the proof we will use that
\begin{equation}\label{eq 11diamonds}
    \idmnd{1}{1}{n}\top \vdash_{\mathsf{BI_\omega}} \idmnd{1}{1}{m}\top, \text{ for } n \geq m,
\end{equation}
obtained by applying Mod$_1$ $m$ times to
$\idmnd{1}{1}{n-m}\top \vdash \top$, and 
\begin{equation}\label{eq 11diamonds 2}
    \dmnd{1}{2} \top \vdash_{\mathsf{BI_\omega}} \idmnd{1}{1}{n}\top, \text{ for all } n,
\end{equation}
obtained by induction on $n$: the case $n=0$ is trivial, and if $\dmnd{1}{2}\top \vdash \idmnd{1}{1}{n}\top$ then we use that:
\[
\dmnd{1}{2}\top \;\vdash\; \dmnd{1}{2}\top \wedge \idmnd{1}{1}{n}\top
\;\vdash\; \dmnd{1}{2}\bigl(\top \wedge \idmnd{1}{1}{n}\top\bigr)
\;\vdash\; \dmnd{1}{1}\idmnd{1}{1}{n}\top = \idmnd{1}{1}{n+1}\top .
\]

1. We show by induction on the structure of $A$ that $A \dashv\vdash_{\mathsf{BI_\omega}} \idmnd{1}{1}{n}\top$ for some $n \in \mathbb{N}$.
\begin{itemize}
    \item[-] $A = \top$: take $n = 0$.
    \item[-] $A = \dmnd{1}{1} B$: by the inductive hypothesis $B \dashv\vdash \idmnd{1}{1}{n}\top$, so by Mod$_1$ we get $A \dashv\vdash \idmnd{1}{1}{n+1}\top$.
    \item[-] $A = B \wedge C$: by the inductive hypothesis $B \dashv\vdash \idmnd{1}{1}{n}\top$ and $C \dashv\vdash \idmnd{1}{1}{m}\top$; suppose w.l.o.g.\ $n \geq m$. Then $A \vdash \idmnd{1}{1}{n}\top$ by L$\wedge$, and conversely $\idmnd{1}{1}{n}\top \vdash B \wedge C$ by \eqref{eq 11diamonds} and R$\wedge$.
\end{itemize}

2. Let $A, B$ be fundamental formulas in $\W$. By part~1 there are $n, m$ such that $A \dashv\vdash_{\mathsf{BI_\omega}} \idmnd{1}{1}{n}\top$ and $B \dashv\vdash_{\mathsf{BI_\omega}} \idmnd{1}{1}{m}\top$. If $n = m$, then $A \dashv\vdash_{\mathsf{BI_\omega}} B$, whence $\nabla A \dashv\vdash_{\mathsf{BI_\omega}} \nabla B$ by
Mod$_3$, and so $A \equiv^1_1 B$. Otherwise suppose w.l.o.g. $n > m$. Then $n \geq m + 1$, so by \eqref{eq 11diamonds} $$A \vdash_{\mathsf{BI_\omega}} \idmnd{1}{1}{m+1}\top = \dmnd{1}{1}\idmnd{1}{1}{m}\top \vdash_{\mathsf{BI_\omega}} \dmnd{1}{1} B,$$ where the last step is Mod$_1$. Hence $B <^1_1 A$.

3. We are going to show that $B \in \W$ satisfies $B \vdash_{\mathsf{BI_\omega}} \idmnd{1}{1}{k}\top$ for every $k$. By part 1 this suffices to prove part 3. We argue by induction on the structure of $B$.
\begin{itemize}
    \item[-] $B = \dmnd{1}{\alpha} C$ with $\alpha \geq 2$: by Mod$_1$ applied to $C \vdash \top$ and Axiom~1 we get $B \vdash \dmnd{1}{2}\top$, and we conclude by (\ref{eq 11diamonds 2}).
    \item[-] $B = \dmnd{2}{\alpha} C$: by Axiom~2, $B \vdash \dmnd{1}{\alpha} C$, and we proceed as in the previous case.
    \item[-] $B = {\downarrow} C$: by Mod$_2$ applied to $C \vdash \top$, Axiom~13, and Axiom~1 we get $B \vdash {\downarrow}\top \vdash \dmnd{1}{\omega}\top \vdash \dmnd{1}{2}\top$, and we conclude by (\ref{eq 11diamonds 2}).
    \item[-] $B = \dmnd{1}{1} C$ with $C$ non-fundamental: by the inductive hypothesis $C \vdash \idmnd{1}{1}{k}\top$ for every $k$, so by Mod$_1$, $B \vdash \idmnd{1}{1}{k+1}\top \vdash \idmnd{1}{1}{k}\top$.
    \item[-] $B = C \wedge D$ with $C$ or $D$ non-fundamental: apply the inductive hypothesis to a non-fundamental conjunct and use L$\wedge$.
\end{itemize}

Now let $A,B\in\W$ such that $A$ is fundamental and $B$ non-fundamental. By part~1, $\idmnd{1}{1}{n}\top \vdash A$ for some $n$, and by the above argument, $B \vdash \idmnd{1}{1}{n+1}\top$. Combining these we get: $$B \vdash_{\mathsf{BI_\omega}} \dmnd{1}{1}\idmnd{1}{1}{n}\top\vdash_{\mathsf{BI_\omega}} \dmnd{1}{1} A,$$ i.e.\ $A <^1_1 B$.
\end{proof}

\begin{definition}\label{def form type}
Every formula of $\W$ can be written as $A = A_1 \wedge \dots \wedge A_k$ where each $A_j$ is $\top$ or has one of the modalities as its outermost connective; we call the $A_j$ the top conjuncts of $A$. We say that a non-fundamental formula $A$ is of:
\begin{itemize}
    \item[-] successor form, if every top conjunct is $\top$ or $\dmnd{1}{1}B$;
    \item[-] limit form, if some top conjunct of $A$ is of the form $\dwn B$ or $\dmnd{i}{\alpha}B$ with $\alpha \geq 2$.
\end{itemize}
\end{definition}

\begin{definition}\label{def approx}
Given a limit-form formula $A\in\W$, we define the approximations $(A)_n$ for $n\geq1$ by recursion on $A$ as follows.

\begin{itemize}
    \item[-] $A=\dmnd{i}{\alpha}A'$ with $\alpha\geq2$: define $(A)_1=A'$ and
    \[(A)_{n+1}=\begin{cases}
      \dmnd{i}{\alpha-1}\bigl(A'\wedge(A)_n\bigr) & \text{if }\alpha>i,\\[2pt]
      \dmnd{1}{1}\bigl(A'\wedge{\dwn}(A)_n\bigr) & \text{if }\alpha=i=2.
    \end{cases}\]

    \item[-] $A={\dwn}C$. If $C$ is of limit form, we define \[({\dwn}C)_n=\dmnd{1}{n}{\dwn}(C)_n.\] Otherwise $C=\dmnd{1}{1}E_1\wedge\dots\wedge\dmnd{1}{1}E_k$, with $k\geq0$, for which we define \[({\dwn}C)_n=\dmnd{1}{n}{\dwn}E_1\wedge\dots\wedge\dmnd{1}{n}{\dwn}E_k.\] If $k=0$, we set $(\dwn \top)_n=\dmnd{1}{n}\top$.

        \item[-] $A$ is a conjunction of at least two top conjuncts. We write $A=B_A\wedge C_A$, where $B_A={\dwn}B_1\wedge\dots\wedge{\dwn}B_{k_1}$ are the ${\dwn}$-conjuncts and $C_A=C_1\wedge\dots\wedge C_{k_2}$ are the diamond conjuncts, written $C_j=\dmnd{i_j}{\alpha_j}C'_j$ and sorted so that $\alpha_1\geq\dots\geq\alpha_{k_2}$.

    If $k_1>0$, then \[(A)_n=({\dwn}B_1)_n\wedge\dots\wedge({\dwn}B_{k_1})_n\wedge C_A.\]

    If $k_1=0$, let $m\geq1$ be maximal with $\alpha_1=\dots=\alpha_m$. Then $C_A=C_1\wedge\dots\wedge C_m\wedge C^0_A$ with $C^0_A=C_{m+1}\wedge\dots\wedge C_{k_2}$ (read as $\top$ if $m=k_2$), and we denote by $A^\circlearrowleft$ the (one-step) cyclic shift of $A$, which is the formula $C_2\wedge\dots\wedge C_m\wedge C_1\wedge C^0_A$ obtained by moving $C_1$ behind $C_m$. Note that $A^\circlearrowleft$ is again of this form, with the same $m$ and the same $C^0_A$, so the recursion below is well defined.
    
    We then define $(A)_1=C^0_A$ and
    \[ 
    (A)_{n+1}=\begin{cases}
        \dmnd{i_1}{\alpha_1-1}\bigl(C'_1\wedge(A^\circlearrowleft)_n\bigr) & \text{if }\alpha_1>i_1,\\[2pt]
        \dmnd{1}{1}\bigl(C'_1\wedge{\dwn}(A^\circlearrowleft)_n\bigr) & \text{if }\alpha_1=i_1=2.
    \end{cases}
    \]
\end{itemize}
\end{definition}

\begin{lemma}\label{lem approx linearity}
Let $A \in \W$ be of limit form. Then $A >^1_1 (A)_n$ for all $n\geq 1$.
\end{lemma}
\begin{proof}
We prove the stronger statement that $A\vdash\dmnd{1}{\alpha}(A)_n$ for all $n\geq 1$, for some $\alpha\geq 2$ which will depend on the structure of $A$. This inherently gives us $A\vdash\dmnd{1}{1}(A)_n$. We say that a formula $A$ is $\alpha$-absorbable if every top conjunct of $A$ is $\top$ or of the form $\dmnd{j}{\beta}B$ with $\beta<\alpha$ or $\alpha=\beta=\omega$. We note the following syntactic observations which we will use throughout the proof:
\begin{enumerate}
    \item \label{obs approx lin 1} If $A$ is $\alpha$-absorbable, then $A\wedge \dmnd{i}{\alpha} B\vdash \dmnd{i}{\alpha}(A\wedge B)$.
    \item \label{obs approx lin 2} Any conjunction of formulas of the form $\dmnd{1}{\omega} A_k$ is $\omega$-absorbable.
    \item \label{obs approx lin 3} The formula approximations $(\dwn A)_n$ are $\alpha$-absorbable for $\alpha>n$.
    \item \label{obs approx lin 4} $\dmnd{1}{\omega}A\vdash\dmnd{1}{\omega}\dmnd{1}{n}A$;
    \item \label{obs approx lin 5} $\dmnd{1}{\omega}A\wedge\dmnd{1}{\omega}B\vdash\dmnd{1}{\omega}(A\wedge B)$ for $\omega$-absorbable $A$.
\end{enumerate} 

\noindent We now proceed to argue by induction on the length of $A$.

\begin{itemize}
\item $A=\dmnd{2}{2}B$.

We first show $A\vdash_\bio\dmnd{2}{2}(B\wedge(A)_n)$ for all
$n\ge1$, by induction on $n$. For $n=1$ this is Mod$_1$ applied to
$B\vdash_\bio B\wedge B$. Suppose $A\vdash_\bio\dmnd{2}{2}(B\wedge(A)_n)$. Then:
\[
A\vdash_\bio\dmnd{2}{2}B\wedge\dmnd{1}{1}(B\wedge(A)_n)
\vdash_\bio\dmnd{2}{2}\big(B\wedge{\dwn}(B\wedge(A)_n)\big)
\vdash_\bio\dmnd{2}{2}\big(B\wedge{\dwn}(A)_n\big)
\vdash_\bio\dmnd{1}{1}\big(B\wedge{\dwn}(A)_n\big)=(A)_{n+1}.
\]
Since $(A)_{n+1}$ is $2$-absorbable we get $A\vdash_\bio\dmnd{2}{2}B\wedge(A)_{n+1}\vdash_\bio\dmnd{2}{2}(B\wedge(A)_{n+1})$. 

Now, since $A\vdash_\bio\dmnd{2}{2}(B\wedge(A)_n)\vdash_\bio\dmnd{1}{2}(A)_n$ we conclude that $A\vdash_\bio \dmnd{1}{2}(A)_n$.

\item $A=\dmnd{i}{\alpha}B$ with $i<\alpha$. 

Again we first show $A\vdash_\bio\dmnd{i}{\alpha}(B\wedge(A)_n)$ for all $n\ge1$ by induction on $n$. The base case is the same as in the previous case. If $A\vdash_\bio\dmnd{i}{\alpha}(B\wedge(A)_n)$, then by Axiom 1, $A\vdash_\bio\dmnd{i}{\alpha-1}(B\wedge(A)_n)=(A)_{n+1}$. Since $(A)_{n+1}$ is $\alpha$-absorbable, we get $A\vdash_\bio\dmnd{i}{\alpha}B\wedge(A)_{n+1}\vdash_\bio\dmnd{i}{\alpha}(B\wedge(A)_{n+1})$.

Finally $A\vdash_\bio\dmnd{i}{\alpha}(B\wedge(A)_n)\vdash_\bio\dmnd{1}{\alpha}(A)_n\vdash_\bio\dmnd{1}{2}(A)_n$.

\item $A={\dwn}B$. 

If $B$ is of limit form the outer induction hypothesis gives $B\vdash\dmnd{1}{1}(B)_n$, and by Mod$_2$, Axiom 14 and \ref{obs approx lin 4}, ${\dwn}B\vdash{\dwn}\dmnd{1}{1}(B)_n\vdash\dmnd{1}{\omega}{\dwn}(B)_n\vdash\dmnd{1}{\omega}\dmnd{1}{n}{\dwn}(B)_n=\dmnd{1}{\omega}({\dwn}B)_n$. Otherwise $B=\bigwedge_{j\le k}\dmnd{1}{1}E_j$ (up to $\top$), and by Mod$_2$, L$\wedge$, Axiom 14 and \ref{obs approx lin 4}, ${\dwn}B\vdash\bigwedge_j\dmnd{1}{\omega}\dmnd{1}{n}{\dwn}E_j$; each $\dmnd{1}{n}{\dwn}E_j$ is $\omega$-absorbable, so \ref{obs approx lin 5} gives $\dmnd{1}{\omega}(\bigwedge_j\dmnd{1}{n}{\dwn}E_j)=\dmnd{1}{\omega}({\dwn}B)_n$ (for $k=0$: Axiom 13 and \ref{obs approx lin 4}).

\item $A=B_A\wedge C_A$ (as in Definition \ref{def approx}).
\begin{itemize}
    \item If $k_1>0$: by the previous case, ${\dwn}B_j\vdash\dmnd{1}{\omega}({\dwn}B_j)_n$ for every $j\le k_1$ and $n\geq 1$, so $B_A\vdash\dmnd{1}{\omega}\big(\bigwedge_{j\le k_1}({\dwn}B_j)_n\big)$.

    Every conjunct of $C_A$ is a diamond formula, hence $\omega$-absorbable, which gives us that $A\vdash\dmnd{1}{\omega}\big(\bigwedge_{j}({\dwn}B_j)_n\wedge C_A\big)=\dmnd{1}{\omega}(A)_n$.
    \item If $k_1=0$: we show, by induction on $n$ and simultaneously for all formulas $A$ of this form (in particular for all iterated shifts of $A$, which have the same maximal rank $\alpha_1$), that $A\vdash_\bio(A)_n$ and that $(A)_n$ is $\alpha_1$-absorbable. All ranks in $A$ are finite, as $A\in\mathcal L_{\mathsf{BI}}$.

    For $n=1$, $(A)_1=C^0_A$ is a sub-conjunction of $A$, possibly $\top$, whose conjuncts are diamonds of rank smaller than $\alpha_1$.

    For the inductive step, the induction hypothesis for $A^\circlearrowleft$ gives $A^\circlearrowleft\vdash_\bio(A^\circlearrowleft)_n$ with $(A^\circlearrowleft)_n$ $\alpha_1$-absorbable, and $A\vdash_\bio A^\circlearrowleft$ by simple conjunction rule applications. Thus by observation \ref{obs approx lin 1},
    \[ A\vdash_\bio C_1\wedge(A^\circlearrowleft)_n\vdash_\bio\dmnd{i_1}{\alpha_1}\bigl(C'_1\wedge(A^\circlearrowleft)_n\bigr). \]
    \begin{itemize}
        \item[-] If $\alpha_1>i_1$, Axiom 1 gives $A\vdash_\bio\dmnd{i_1}{\alpha_1-1}(C'_1\wedge(A^\circlearrowleft)_n)=(A)_{n+1}$.
        \item[-] If $\alpha_1=i_1=2$, then $\dmnd{2}{2}(C'_1\wedge(A^\circlearrowleft)_n)\vdash_\bio\dmnd{1}{1}(A^\circlearrowleft)_n$,  so by Axiom 11 we get
        \[ A\vdash_\bio\dmnd{2}{2}C'_1\wedge\dmnd{1}{1}(A^\circlearrowleft)_n\vdash_\bio\dmnd{2}{2}\bigl(C'_1\wedge{\dwn}(A^\circlearrowleft)_n\bigr)\vdash_\bio\dmnd{1}{1}\bigl(C'_1\wedge{\dwn}(A^\circlearrowleft)_n\bigr)=(A)_{n+1}. \]
    \end{itemize}
    In both cases $(A)_{n+1}$ is a single diamond formula of rank at most $\alpha_1-1$, hence it is $\alpha_1$-absorbable.

    Finally by observation \ref{obs approx lin 1},
    \[ A\vdash_\bio C_1\wedge(A)_n\vdash_\bio\dmnd{i_1}{\alpha_1}\bigl(C'_1\wedge(A)_n\bigr)\vdash_\bio\dmnd{1}{\alpha_1}(A)_n\vdash_\bio\dmnd{1}{2}(A)_n. \]
\end{itemize}
\end{itemize}
In every case $A\vdash\dmnd{1}{\alpha}(A)_n\vdash\dmnd{1}{1}(A)_n$ (i.e. $A>^1_1(A)_n$) for all $n\geq 1$.
\end{proof}

\begin{lemma}\label{lem linearity 1}
    Given formulas $A, B$ in $\W$,
    \[ A \vdash_{\mathsf{BI_\omega}} \nabla B \text{ iff } T_B \hookrightarrow M_A. \] 
\end{lemma}

\begin{proof}
    By Lemma \ref{lem locality}(i) with $B=A$, we have $r_AR^\nabla a\in\mathcal R_A$ for every $a\in\mathcal S_A$. Hence $T_B\hookrightarrow M_A$ is equivalent to $T_{\nabla B}\hookrightarrow_r M_A$, which by Corollary \ref{cor trees} is equivalent to $A\vdash_{\mathsf{BI_\omega}}\nabla B$.
\end{proof}

\begin{lemma}\label{lem rank bound}
Let $A$ be a variable-free formula. If $aR^1_\beta b\in\mathcal R_A$ for infinitely many $\beta<\omega$, then $aR^1_\omega b\in\mathcal R_A$.
\end{lemma}
\begin{proof}
Let $m$ be the highest finite rank of a modality occurring in $A$. We show, by induction on the stage at which $aR^i_\beta b$ enters the closure, that $aR^i_\beta b\in\mathcal R_A$ with $m<\beta<\omega$ implies $aR^1_\omega b\in\mathcal R_A$. The transitions of $\mathcal R^c_A$ have rank in $\{1,\dots,m\}\cup\{\omega\}$, so we only need to consider the closure step. Since rules (c), (d), (e), (g), (h), (i), (j) and (l) produce transitions of rank $1$, we only need to consider the following cases:
\begin{itemize}
    \item Rule (a): $aR^i_\alpha b$ with $\alpha>\beta$. Either $\alpha=\omega$ and we are done, or $m<\alpha<\omega$ and $aR^1_\omega b$ by the induction hypothesis.
    \item Rule (b): $aR^2_\beta b$. Then $aR^1_\omega b$ by the induction hypothesis.
    \item Rule (f): $aR^i_\beta dR^j_\beta b$. The induction hypothesis gives $aR^1_\omega d$ and $dR^1_\omega b$, so $aR^1_\omega b$ by (f).
    \item Rule (k): $dR^{i'}_\alpha a$ and $dR^j_\beta b$ with $\alpha>\beta$. The induction hypothesis gives $dR^1_\omega b$, and $dR^1_\omega a$ holds either directly, if $\alpha=\omega$, or by the induction hypothesis, if $m<\alpha<\omega$. Then $aR^1_\omega b$ by (k) with $\alpha=\beta=\omega$.
\end{itemize}
\end{proof}

\begin{lemma}\label{lem 12 tree approx 1}
Let $A\in\W$ be such that $\leq^1_1$ is linear below $A$ (that is, any two formulas $C,C'\in\W$ with $C<^1_1A$ and $C'<^1_1A$ are $\leq^1_1$-comparable), and let $B\in\W$ be of limit form. If $A\vdash_\bio\dmnd{1}{1}(B)_n$ for infinitely many $n$, then $A\vdash_\bio\nabla B$.
\end{lemma}
\begin{proof}
We prove by induction on the limit-form formula $B\in\W$ that for every $A\in\W$ such that $\leq^1_1$ is linear below it:
\begin{itemize}
    \item[$(*)$] if some point of $M_A$ satisfies $(B)_n$ for infinitely many $n$, then some point of $M_A$ satisfies $B$.
\end{itemize}
The lemma then follows: by Theorem \ref{thm trees}, $M_A,r_A\models\dmnd{1}{1}(B)_n$ for infinitely many $n$. Since $M_A$ is finite, this means that some point of $M_A$ satisfies $(B)_n$ for infinitely many $n$. Then by $(*)$ some point of $M_A$ satisfies $B$, and by Proposition \ref{prop embed} and Lemma \ref{lem linearity 1} we conclude that $A\vdash_\bio\nabla B$.\\

We now proceed with the proof of $(*)$. We first note that if $\leq^1_1$ is linear below $A$, and $D$ is a formula as in Lemma \ref{lem locality}(i) (so that $r_D$ is an anchor point of $M_A$) then $\leq^1_1$ is also linear below $D$. To prove this, if $C<^1_1D$, then $M_D,r_D\models\dmnd{1}{1}C$ by Theorem \ref{thm trees}, so $M_A,r_D\models\dmnd{1}{1}C$ by Lemma \ref{lem locality}(i) (i.e. $r_AR^\nabla r_DR^1_1c$ for some point $c$ with $c\models C$). Hence $r_AR^1_1c$ by (h), and $C<^1_1A$ by Theorem \ref{thm trees}. So the formulas below $D$ are below $A$, and they are pairwise $\leq^1_1$-comparable.\\

We now fix a formula $A$ such that $\leq^1_1$ is linear below it, a point $a$ of $M_A$ and an infinite set $I\subseteq\mathbb N$ with $M_A,a\models(B)_n$ for all $n\in I$. We distinguish two cases according to the form of $B$.
\begin{itemize}
    \item Case 1: $B$ has a top conjunct of the form ${\dwn}B'$. 
    
    We write $B={\dwn}B_1\wedge\dots\wedge{\dwn}B_k\wedge C_B$ with $k\geq1$, where $C_B$ is the conjunction of the $\Diamond$-conjuncts of $B$ (read as $\top$ if there are none), so that $(B)_n=({\dwn}B_1)_n\wedge\dots\wedge({\dwn}B_k)_n\wedge C_B$ by Definition \ref{def approx}.

    Each $({\dwn}B_j)_n$ is a non-empty conjunction of formulas of the form $\dmnd{1}{n}X_n$. Since $M_A$ is finite and $a\models\dmnd{1}{n}X_n$ for all $n\in I$, there is a point $d$ with $aR^1_nd$ and $d\models X_n$ for infinitely many $n\in I$. Then by Lemma \ref{lem rank bound} we have $aR^1_\omega d$ .

    In particular $a$ has an outgoing $R^1_\omega$ transition, so by Lemma \ref{lem S} there is a point $b$ and a set $T$ of anchor points. By (i), $b\models C_B$.

    We show that $b\models{\dwn}B_j$ for every $j\leq k$, which gives $b\models B$.
    \begin{itemize}
        \item $B_j$ is of limit form, so that $({\dwn}B_j)_n=\dmnd{1}{n}{\dwn}(B_j)_n$. 
        
        By finiteness there are points $d$ and $c$ with $aR^1_ndR^\dwn c$ and $c\models(B_j)_n$ for infinitely many $n\in I$. Then $aR^1_\omega d$, and (iii) gives $t\in T$ with $tR^1_1c$.

        Let $D$ be as in Lemma \ref{lem locality}(i) with $r_D=t$. Then $c\in\mathcal S_D$ and $M_D,c\models(B_j)_n$ for infinitely many $n$, and $\leq^1_1$ is linear below $D$ by the note above. The induction hypothesis for $B_j$ provides $c'\in\mathcal S_D$ with $M_D,c'\models B_j$, hence $M_A,c'\models B_j$.

        Since $bR^\dwn t$ by (ii) and $tR^\nabla c'$, we get $bR^\dwn c'$ by (j), so $b\models{\dwn}B_j$.
        \item $B_j$ is not of limit form, so that $B_j=\dmnd{1}{1}E_1\wedge\dots\wedge\dmnd{1}{1}E_p$ up to $\top$ conjuncts, with $p\geq0$, and $({\dwn}B_j)_n=\dmnd{1}{n}{\dwn}E_1\wedge\dots\wedge\dmnd{1}{n}{\dwn}E_p$, read as $\dmnd{1}{n}\top$ if $p=0$. 
        
        If $p=0$, then $bR^\dwn t$ for any $t\in T$ gives $b\models{\dwn}B_j$. 
        
        Let $p\geq1$. As before, for every $l\leq p$ there are points $d_l$ and $c_l$ with $aR^1_\omega d_lR^\dwn c_l$ and $c_l\models E_l$, and (iii) gives $t_l\in T$ with $t_lR^1_1c_l$.

        Since $r_AR^\nabla t_l$, rule (h) gives $r_AR^1_1c_l$, so $A\vdash_\bio\dmnd{1}{1}E_l$ by Theorem \ref{thm trees}, that is, $E_l<^1_1A$. As $\leq^1_1$ is linear below $A$, the formulas $E_1,\dots,E_p$ are pairwise $\leq^1_1$-comparable. The relation $\leq^1_1$ is transitive, by Proposition \ref{prop trans and compose} together with Axioms 3 and 8, so there is $l_0\leq p$ with $E_l\leq^1_1E_{l_0}$ for all $l\leq p$.

        We claim that $\dmnd{1}{1}E_{l_0}\vdash_\bio\dmnd{1}{1}E_l$ for every $l\leq p$:
        \begin{itemize}
            \item[-] if $E_l<^1_1E_{l_0}$, that is, $E_{l_0}\vdash_\bio\dmnd{1}{1}E_l$, then this follows by Mod$_1$ and Axiom 6;
            \item[-] if $\nabla E_l\dashv\vdash_\bio\nabla E_{l_0}$, then
            \[ \dmnd{1}{1}E_{l_0}\vdash_\bio\dmnd{1}{1}\nabla E_{l_0}\vdash_\bio\dmnd{1}{1}\nabla E_l\vdash_\bio\dmnd{1}{1}E_l \]
            by Axiom 3, Mod$_3$ and Axiom 9, each combined with Mod$_1$.
        \end{itemize}
        Since $t_{l_0}R^1_1c_{l_0}$ and $c_{l_0}\models E_{l_0}$, we have $t_{l_0}\models\dmnd{1}{1}E_{l_0}$, hence $t_{l_0}\models\dmnd{1}{1}E_l$ for every $l\leq p$ by soundness (Propositions \ref{prop Soundness} and \ref{prop MA is a model}), that is, $t_{l_0}\models B_j$. As $bR^\dwn t_{l_0}$ by (ii), we get $b\models{\dwn}B_j$.
    \end{itemize}

    \item Case 2: $B$ has no top conjunct of the form ${\dwn}B'$. 
    
    As in Definition \ref{def approx}, we write $B=C_1\wedge\dots\wedge C_m\wedge C^0_B$, where $C_t=\dmnd{i_t}{\alpha}C'_t$ for $t\leq m$ are the top conjuncts of maximal rank $\alpha\geq2$, and $C^0_B$ is the conjunction of the remaining top conjuncts. 
    
    Note that for a single diamond $B=\dmnd{i}{\alpha}C$ we have $m=1$ and $C^0_B=\top$. In this case we get $(B)_1=C$ rather than $C^0_B$, which does not affect the argument below since $c\models C^0_B$ is then trivial.
    
    For the duration of this proof, for $s\geq0$ we write $B^{(s)}$ for the $s^{th}$ cyclic shift of $B$, and for $s\geq1$ we write $\rho(s)$ for the element of $\{1,\dots,m\}$ congruent to $s$ modulo $m$. We call $s\geq1$ a ${\dwn}$-step if $i_{\rho(s)}=\alpha=2$.
    
    By Definition \ref{def approx}, for all $s,k\geq1$,
    \[ (B^{(s-1)})_{k+1}=\begin{cases}
        \dmnd{i_{\rho(s)}}{\alpha-1}\bigl(C'_{\rho(s)}\wedge(B^{(s)})_k\bigr) & \text{if $s$ is not a ${\dwn}$-step,}\\[2pt]
        \dmnd{1}{1}\bigl(C'_{\rho(s)}\wedge{\dwn}(B^{(s)})_k\bigr) & \text{if $s$ is a ${\dwn}$-step.}
    \end{cases}\]

    We fix $n\in I$ with $n>m\cdot|\mathcal S_A|+1$. Unfolding $a\models(B^{(0)})_n$ along this display, we obtain points $a_0=a$ and $b_s,a_s$ for $1\leq s<n$ such that $b_s\models C'_{\rho(s)}$, $a_s\models(B^{(s)})_{n-s}$, and:
    \begin{itemize}
        \item $a_{s-1}R^{i_{\rho(s)}}_{\alpha-1}b_s$ and $a_s=b_s$, if $s$ is not a ${\dwn}$-step;
        \item $a_{s-1}R^1_1b_s$ and $b_sR^\dwn a_s$, if $s$ is a ${\dwn}$-step.
    \end{itemize}
    In both cases $b_sR^\nabla a_s$ by (c) or (e), and $a_sR^1_{\alpha-1}b_{s+1}$ by (b). Hence
    \begin{equation}\label{eq 411 chain}
        b_sR^1_{\alpha-1}b_{s+1}\quad\text{for all }1\leq s<n-1,
    \end{equation}
    where for $\alpha=2$ we use (h), and for $\alpha\geq3$ there are no ${\dwn}$-steps and $a_s=b_s$ for all $s$.

    By the choice of $n$, two of the $n-1$ pairs $(\rho(s),b_s)$ with $1\leq s<n$ coincide: there are $1\leq s<s'<n$ with $s\equiv s'\pmod m$ and $b_s=b_{s'}$. We define $Q=\{b_{s+1},\dots,b_{s'}\}$. By (\ref{eq 411 chain}) and (f), passing through $b_{s'}=b_s$ where needed, we have $bR^1_{\alpha-1}b'$ for all $b,b'\in Q$.

    For every $t\leq m$ we fix $u_t\in\{s+1,\dots,s'\}$ with $\rho(u_t)=t$, which exists since $s'-s\geq m$. Then $b_{u_t}\in Q$ and $b_{u_t}\models C'_t$.

    We now find a point $c$ with $cR^{i_t}_\alpha b_{u_t}$ for all $t\leq m$, so that $c\models C_1\wedge\dots\wedge C_m$.
    \begin{itemize}
        \item $\alpha\geq3$: let $c$ be the point given by Lemma \ref{lem guards}(1) for $Q$ and $\beta=\alpha-1$. Then $cR^1_\alpha b_{u_t}$ for every $t\leq m$. If $i_t=2$, then $b_{u_t-1}R^2_{\alpha-1}b_{u_t}$ with $b_{u_t-1}\in Q$ (recall that $b_s=b_{s'}$), so $cR^2_\alpha b_{u_t}$.
        \item $\alpha=2$: let $c$ be the common anchor of the points of $Q$ given by Lemma \ref{lem guards}(2). Then $cR^1_2b_{u_t}$ for every $t\leq m$.

        Suppose that $i_t=2$, and let $u=u_t$, so that $u$ is a ${\dwn}$-step. If $u<s'$, then $b_uR^\dwn a_uR^1_1b_{u+1}$ with $b_{u+1}\in Q$. If $u=s'$, then $\rho(s)=\rho(s')=t$, so $s$ is a ${\dwn}$-step as well, and $b_u=b_sR^\dwn a_sR^1_1b_{s+1}$ with $b_{s+1}\in Q$.

        In both cases $b_uR^\dwn dR^1_1b'$ for some point $d$ and some $b'\in Q$. As $b'R^1_1b_u$, rule (f) gives $dR^1_1b_u$. By Lemma \ref{lem down}(ii), $b_u$ is an $R^2$ point, hence $cR^2_2b_u$.
    \end{itemize}
    It remains to show that $c\models C^0_B$. We have $a_{n-1}\models(B^{(n-1)})_1=C^0_B$.
    \begin{itemize}
        \item $\alpha\geq3$: then $a_{n-1}=b_{n-1}$, and $cR^1_\alpha b_{s'}$ together with (\ref{eq 411 chain}) gives $cR^1_{\alpha-1}b_{n-1}$ by (a) and (f). A top conjunct $\dmnd{i}{\gamma}C'$ of $C^0_B$ has $\gamma\leq\alpha-1$, and $b_{n-1}R^i_\gamma d$ with $d\models C'$ gives $cR^i_\gamma d$ by (a) and (f).
        \item $\alpha=2$: by (d), (e) and (g) along the unfolding we have $b_{s'}R^\nabla a_{n-1}$, so $cR^1_2b_{s'}$ gives $cR^1_1a_{n-1}$ by (i). The top conjuncts of $C^0_B$ are of the form $\dmnd{1}{1}C'$, and $a_{n-1}R^1_1d$ with $d\models C'$ gives $cR^1_1d$ by (f).
    \end{itemize}
    Hence $c\models B$, which completes the induction.
\end{itemize}
\end{proof}

\begin{theorem}[Theorem \ref{thm_lin_int}]\label{thm lin}
The relation $\le^1_1$ is a linear, well-founded pre-order on
$\W$. In particular $\mathbf W_{\mathsf{BI}}/{\equiv^1_1}$ is well-ordered by $<^1_1$.
\end{theorem}
\begin{proof}
We first argue that $<^1_1$ is well-founded on $\mathbf W_{\mathsf{BI}}$: by the arithmetical soundness of the calculus (Proposition \ref{prop ari sound} for $\vdash_{\mathsf{BI_\omega}}$ via the conservativity Proposition \ref{prop equiv}), $A<^1_1B$ implies $B^*\vdash\Pi^1_1\textsf{-RFN}(A^*)$. All the theories $A^*$ for $A\in\mathbf W_{\mathsf{BI}}$ are true, hence $\Pi^1_1$-sound, and by \cite[Theorem 3.2]{pakhomov2021reflection} the $\Pi^1_1$-sound extensions of $\mathsf{ACA}_0$ are well-founded with respect to the ordering by $\Pi^1_1\text{-}\mathsf{RFN}$. 
    
Suppose $\le^1_1$ is not linear. Using well-foundedness, we first choose $A$ to be $<^1_1$-minimal among the formulas that are $\leq^1_1$-incomparable with some formula, and then choose $B$ to be $<^1_1$-minimal among the formulas $\leq^1_1$-incomparable with $A$. For every formula $C$ we then have \begin{equation}\label{eq min comparable}
    C<^1_1 A\leftrightarrow C<^1_1 B.
\end{equation} 
Indeed, if $C<^1_1 A$, then by the minimality of $A$ the formula $C$ is $\leq^1_1$-comparable with every formula, in particular with $B$; the options $B<^1_1 C$ and $\nabla B\dashv\vdash\nabla C$ both yield $B<^1_1 A$ by Proposition \ref{prop trans and compose}, contradicting the incomparability of $A$ and $B$, so $C<^1_1 B$. Conversely, if $C<^1_1 B$, then by the minimality of $B$ among the formulas incomparable with $A$, the formula $C$ is $\leq^1_1$-comparable with $A$, and as before the options $A<^1_1 C$ and $\nabla A\dashv\vdash\nabla C$ are impossible, so $C<^1_1 A$.

We also note that any two formulas $C,C'\in\W$ with $C<^1_1A$ and $C'<^1_1A$ are $\leq^1_1$-comparable, by the minimality of $A$; by (\ref{eq min comparable}), the same holds with $B$ in place of $A$. Hence both $A$ and $B$ satisfy the hypothesis of Lemma \ref{lem 12 tree approx 1}.

We are now going to show that $A\vdash_\bio \nabla B$. By Lemma \ref{lem fundamental} we know that $A,B$ cannot be fundamental formulas. As such, we distinguish the following two cases:

1. $B$ is of limit form. By Lemma \ref{lem approx linearity} we know that $B>^1_1(B)_n$ for all $n\geq 1$. By substituting into (\ref{eq min comparable}) we have that $A>^1_1 (B)_n$ for all $n\geq 1$. Then, by simply using Lemma \ref{lem 12 tree approx 1} we get that $A\vdash_\bio \nabla B$.

2. $B$ is of successor form. Then we can write $B=\dmnd{1}{1}B_1\wedge\dots\wedge \dmnd{1}{1}B_k$. Since trivially $B>^1_1 B_i$, we have by (\ref{eq min comparable}) that $A>^1_1B_i$ for all $i\leq k$. Thus, $A\vdash_\bio \dmnd{1}{1}B_i$ for all $i\leq k$. By $k-1$ applications of R$\wedge$ on these we get $A\vdash_\bio B$, from which we immediately get $A\vdash_\bio \nabla B$.

By symmetry we also get that $B\vdash \nabla A$, from which we conclude that $A\equiv^1_1 B$, contradicting the incomparability assumption.
\end{proof}

\section{Reduction Properties}\label{section reduction properties}
In this section we prove the reduction properties stated in the introduction. Lemmas \ref{lem 1 red 1} and \ref{lem 2 red 2} are essentially known results. They are the second-order analogues of Beklemishev's reduction property for first-order arithmetic \cite[Section~4]{beklemishev2003proof} (for a $\Pi_{n+1}$-axiomatized theory $U$, the theory $U+\Pi_n\textsf{-RFN}(T)$ is $\Pi_n$-conservative over $U$ extended by the relativized reflection rule $\varphi/\Pi_{n-1}\textsf{-RFN}(T+\varphi)$), and in the present form they are \cite[Lemma~2.11]{pakhomov2023reducing} (see also \cite{pakhomov2023corrigendum}). There the theory $T$ is assumed to be $\Pi^1_n$-axiomatizable, but the same proof---an application of the conservation result \cite[Lemma~2.1]{pakhomov2023corrigendum}, exactly as in the proof of Lemma~\ref{Diam22_alt_red} below---gives the statements for $\Pi^1_{n+1}$-axiomatizable $T$ (in the case of Lemma \ref{lem 2 red 2} necessarily $n\ge 2$, as the modality $\dmnd{2}{1}$ is not part of our language). The proof in \cite{pakhomov2023reducing} is model-theoretic and is based on the ideas of Avigad \cite{avigad2002saturated}, whereas Beklemishev's proof is purely syntactic: a cut-free derivation, in a Tait-style sequent calculus, of a $\Pi_n$-sentence from the axioms of $U$ and instances of $\Pi_n\textsf{-RFN}(T)$ is transformed into a derivation from the axioms of $U$ using the reflection rule. As noted in \cite{pakhomov2023reducing}, each of the two methods can be used to prove either reduction property; in particular, Lemmas \ref{lem 1 red 1} and \ref{lem 2 red 2} could as well be established by the syntactic method of \cite{beklemishev2003proof}. The reduction properties involving the operators $\dwn$ and $\dmnd{2}{2}$, Lemmas \ref{lem 3 red 3}--\ref{lem 5 red 5}, are new, and proving them is the main content of this section.

Throughout the section we will be relying on the terminology from \cite{pakhomov2023reducing}. In some situations, it will be more convenient to use the class-sized language $\mathbf{L}_2$. This is the language of second-order arithmetic extended by constants $C_X$ for all the individual sets $X$. The theories with which we work within second-order arithmetic in this approach are extensions of $\mathsf{ACA}_0$ that have an at most countable number of axioms and are represented as sets of natural numbers coding an at most countable sequence of axioms. The arithmetically definable relation $T_1\sqsubseteq T_2$ is the relation asserting that all $T_1$-axioms are provable in $T_2$.

Recall from \cite{pakhomov2023reducing} that for classes $\mathbf{\Pi}^1_n$ of $\mathbf{L}_2$ formulas there are partial truth definitions. Following \cite{pakhomov2023reducing} we define
\[\mathbf{\Pi}^1_n\textsf{-RFN}(T):\;\;\;\forall \varphi\in \mathbf{\Pi}^1_n ((T\vdash \varphi)\to \mathsf{Tr}_{\mathbf{\Pi}^1_n}(\varphi))\]
\[\mathbf{\Pi}^1_n\textsf{-$\omega$RFN}(T):\;\;\;\forall \varphi\in \mathbf{\Pi}^1_n (\forall M(\varphi\dot \in M \land T\dot \subseteq M\land M\models T \to M\models \varphi)\to \mathsf{Tr}_{\mathbf{\Pi}^1_n}(\varphi)).\]
Here $\varphi\dot\in M$ and $T\dot\subseteq M$ mean that all the set constants occurring in $\varphi$, respectively in $T$, are among the sets of $M$. Also we define
\[\dwn T=\{\varphi\in L(T)\mid T\vdash \forall M(M\models \mathsf{ACA}_0\land \varphi\dot\in M\to M\models \varphi)\},\]
where $L(T)$ is the fragment of $\mathbf{L}_2$ using only the set constants occurring in $T$.

We note two facts about these definitions. In $\mathsf{ACA}_0$, for c.e. theories $T$ without set constants each $\mathbf{\Pi}^1_n\textsf{-RFN}(T)$ is equivalent to $\Pi^1_n\textsf{-RFN}(T)$ and each $\mathbf{\Pi}^1_n\textsf{-$\omega$RFN}(T)$ is equivalent to $\Pi^1_n\textsf{-$\omega$RFN}(T)$. And secondly, provably in $\mathsf{ACA}_0$, if $T_1\sqsubseteq T_2$, then for each $n$
\begin{enumerate}
    \item $\mathbf{\Pi}^1_n\textsf{-RFN}(T_2)\to \mathbf{\Pi}^1_n\textsf{-RFN}(T_1)$,
    \item $\mathbf{\Pi}^1_n\textsf{-$\omega$RFN}(T_2)\to \mathbf{\Pi}^1_n\textsf{-$\omega$RFN}(T_1)$,
    \item $\dwn T_1 \sqsubseteq \dwn T_2$.
\end{enumerate}
Implications 1. and 2. are obvious. 

To verify 3., the key observation is that, although by definition $\dwn T$ is a theory in the same countable fragment of $\mathbf{L}_2$ as $T$, up to deductive equivalence the theory does not change if the definition is made over a larger fragment. Namely, to prove $\dwn T_1\sqsubseteq \dwn T_2$ it is sufficient to prove that \begin{equation}\label{downarrow_agnostic}\dwn T \equiv \{\varphi\in H\mid T\vdash \forall M(M\models \mathsf{ACA}_0\land \varphi\dot\in M\to M\models \varphi)\}\text{, for any countable fragment $H\supseteq L(T)$ of $\mathbf{L}_2$,}\end{equation}
and apply it for both $T_1,T_2$ and a countable fragment of $\mathbf{L}_2$ that extends both $L(T_1)$ and $L(T_2)$. 

To prove (\ref{downarrow_agnostic}), given any $\varphi\in H$ such that $T\vdash \forall M(M\models \mathsf{ACA}_0\land \varphi\dot\in M\to M\models \varphi)$ we show that $\varphi$ is provable in $\dwn T$. We put $\varphi$ in the form $\psi(\vec{P})$, where $\psi(\vec{X})\in L(T)$ and $\vec{P}$ are constants not used in $T$. Then we observe that $T\vdash \forall \vec{X}\forall M (M\models \mathsf{ACA}_0\land \psi\dot\in M \land \vec{X}\dot\in M\to M\models \psi(\vec{X}))$. Hence $T\vdash \forall M (M\models \mathsf{ACA}_0\land \psi\dot \in M \to M\models \forall \vec{X} \psi(\vec{X}))$. Thus $\forall \vec{X} \psi(\vec{X})$ is in $\dwn T$ and therefore $\varphi$ is provable in $\dwn T$.

For countably coded $\mathbf{L}_2$ extensions $T$ of $\mathsf{ACA}_0$ we will freely write $\dmnd{1}{n} T$ as a shorthand for $\mathbf{\Pi}^1_n\textsf{-RFN}(T)$ and $\dmnd{2}{n} T$ as a shorthand for $\mathbf{\Pi}^1_n\textsf{-$\omega$RFN}(T)$.

\begin{lemma}\label{Diam22_alt}
Over $\mathsf{ACA}_0$, for any finitely axiomatizable theory $U$ extending $\mathsf{ACA}_0$, $\dmnd{2}{2} U$ is equivalent to the scheme
\begin{equation}\label{diamond^2_2_alt}\forall \vec{x},\vec{X}\big(\varphi(\vec{x},\vec{X})\to \dmnd{1}{1}(U+\dwn(\mathsf{ACA}_0+\varphi(\dot{\vec{x}},\dot{\vec{X}})))\big)\text{, where }\varphi\text{ is }\Pi^1_1.\end{equation}

\end{lemma}
\begin{proof}
Assuming $\dmnd{2}{2} U$, we prove (\ref{diamond^2_2_alt}) as follows. We consider a true $\mathbf{\Pi}^1_1$-sentence $\varphi$ and a $\mathbf{\Pi}^1_1$-sentence $\psi$ such that $U+\dwn(\mathsf{ACA}_0+\varphi)$ proves $\psi$, and claim that $\psi$ is true. Assume $\psi$ is $\forall X\;\psi'(X)$, where $\psi'$ has no second-order quantifiers. We fix some set $X$ and claim $\psi'(X)$. We use $\dmnd{2}{2} U$ first to construct an $\omega$-model $M_1\models U+\varphi$ such that all the parameters of $\psi$ and $X$ are in $M_1$ and next an $\omega$-model $M_2\models U+\varphi$ such that $M_1\dot\in M_2$. Since $M_2\models\mathsf{ACA}_0+\varphi$ and $M_1\dot\in M_2$, we conclude that  $M_1$ is a model of $\dwn (\mathsf{ACA}_0+\varphi)$ and hence of $U+ \dwn (\mathsf{ACA}_0+\varphi)$. Thus $M_1\models \psi'(X)$. Therefore $\psi'(X)$. This finishes the proof that $\dmnd{2}{2} U$ implies (\ref{diamond^2_2_alt}).

Now we assume (\ref{diamond^2_2_alt}) to prove $\dmnd{2}{2} U$. Suppose for a contradiction that $\dmnd{2}{2} U$ fails. Then there is a false $\mathbf{\Pi}^1_2$-sentence $\varphi$ that holds in all $\omega$-models of $U$ containing all parameters of $\varphi$. We can consider $\varphi$ to be of the form $\forall X \lnot \psi(X)$, where $\psi$ is $\mathbf{\Pi}^1_1$. Thus for a suitable set $A$, $\psi(A)$ is a true $\mathbf{\Pi}^1_1$-sentence that is false in all $\omega$-models of $U$ containing all the parameters of $\psi(A)$ (including $A$). Let $\theta$ be the true $\mathbf{\Pi}^1_1$ sentence that is the conjunction of $\psi(A)$ and the assertion that $\psi(A)$ fails in all $\omega$-models of $U$. Then $\dwn (\mathsf{ACA}_0+ \theta)$ on one hand implies $\psi(A)$ and on the other hand implies $U\to \lnot \psi(A)$.  That is, $U+ \dwn (\mathsf{ACA}_0+ \theta)$ is inconsistent. Given that $\theta$ was true and (\ref{diamond^2_2_alt}) we clearly get a contradiction, hence concluding our proof of $\dmnd{2}{2} U$.
\end{proof}

\begin{lemma}[Lemma \ref{lem 5 red 5}]\label{Diam22_alt_red}
Given a finitely axiomatizable theory $U$ and a $\Pi^1_2$-axiomatizable theory $T$, \begin{equation}\label{lem_red_5_eq}T+ \dmnd{2}{2} U \equiv_{\Pi^1_1} T+\left\{\frac{\varphi\in \Pi^1_1}{\dmnd{1}{1}(U+\dwn(\mathsf{ACA}_0+\varphi))}\right\}.\end{equation}
\end{lemma}
\begin{proof} 
Given the characterization of $\dmnd{2}{2} U$ from Lemma \ref{Diam22_alt} and \cite[Lemma~2.1]{pakhomov2023corrigendum} it is sufficient to prove that the theory
\[T+\left\{\frac{\varphi\in \Pi^1_1}{\dmnd{1}{1}(U+\dwn(\mathsf{ACA}_0+\varphi))}\right\}\]
is closed under the rule
\[\frac{\forall \vec{x},\vec{X},\vec{y},\vec{Y}\;(\theta(\vec{x},\vec{X},\vec{y},\vec{Y})\lor \varphi(\vec{x},\vec{X}))}{\forall \vec{x},\vec{X},\vec{y},\vec{Y}(\theta(\vec{x},\vec{X},\vec{y},\vec{Y})\lor\dmnd{1}{1}(U+\dwn(\mathsf{ACA}_0+\varphi(\dot{\vec{x}},\dot{\vec{X}}))))}\text{, where $\theta,\varphi\in \Pi^1_1$}.\]
For this the following claim is sufficient: Given some $\Pi^1_1$ formulas $\theta(\vec{x},\vec{X},\vec{y},\vec{Y})$ and $\varphi(\vec{x},\vec{X})$ we can prove in $\mathsf{ACA}_0$ that the sentence
\begin{equation} \label{lem 5 red 5 rule premise} \dmnd{1}{1}(U+\dwn(\mathsf{ACA}_0+\forall \vec{x},\vec{X},\vec{y},\vec{Y}\;(\theta(\vec{x},\vec{X},\vec{y},\vec{Y})\lor \varphi(\vec{x},\vec{X}))))\end{equation}
implies the sentence
\begin{equation}\forall \vec{x},\vec{X},\vec{y},\vec{Y}(\theta(\vec{x},\vec{X},\vec{y},\vec{Y})\lor\dmnd{1}{1}(U+\dwn(\mathsf{ACA}_0+\varphi(\dot{\vec{x}},\dot{\vec{X}})))).\end{equation}

To prove the claim from the paragraph above we reason in $\mathsf{ACA}_0$+(\ref{lem 5 red 5 rule premise}) and consider some $ \vec{x},\vec{X},\vec{y},\vec{Y}$ for which $\lnot \theta(\vec{x},\vec{X},\vec{y},\vec{Y})$ holds, and show that $\dmnd{1}{1}(U+\dwn(\mathsf{ACA}_0+\varphi(\dot{\vec{x}},\dot{\vec{X}})))$; this is sufficient to establish the claim and hence to finish the proof of the lemma. Since we can absorb true $\mathbf{\Sigma}^1_1$ sentences into $\dmnd{1}{1}$, we get from $\lnot \theta(\vec{x},\vec{X},\vec{y},\vec{Y})$ and (\ref{lem 5 red 5 rule premise}) that
\[\dmnd{1}{1}(U+\lnot \theta(\dot{\vec{x}},\dot{\vec{X}},\dot{\vec{y}},\dot{\vec{Y}}) +\dwn(\mathsf{ACA}_0+\forall \vec{x},\vec{X},\vec{y},\vec{Y}\;(\theta(\vec{x},\vec{X},\vec{y},\vec{Y})\lor \varphi(\vec{x},\vec{X})))).\]
Then since both $\dmnd{1}{1}$ and $\dwn$ are monotone, we get
\[\dmnd{1}{1}(U+\lnot \theta(\dot{\vec{x}},\dot{\vec{X}},\dot{\vec{y}},\dot{\vec{Y}}) +\dwn(\mathsf{ACA}_0+(\theta(\dot{\vec{x}},\dot{\vec{X}},\dot{\vec{y}},\dot{\vec{Y}})\lor \varphi(\dot{\vec{x}},\dot{\vec{X}})))).\]

Now to obtain $\dmnd{1}{1}(U+\dwn(\mathsf{ACA}_0+\varphi(\dot{\vec{x}},\dot{\vec{X}})))$ and hence to finish the proof it is sufficient to show that
\begin{equation}\label{downarrow_absorb}\dwn(\mathsf{ACA}_0+\psi)\sqsubseteq \dwn(\mathsf{ACA}_0+(\chi\lor\psi)) +\lnot \chi\text{, for any $\psi,\chi\in\mathbf{\Pi}^1_1$}.\end{equation}
To prove (\ref{downarrow_absorb}) we fix $\psi,\chi\in\mathbf{\Pi}^1_1$, consider $\eta$ from $\dwn(\mathsf{ACA}_0+\psi)$ and claim that $\eta$ is provable in $\dwn(\mathsf{ACA}_0+(\chi\lor\psi)) +\lnot \chi$. For this $\eta$ by the definition of $\dwn$ we know that \begin{equation}\label{downarrow_absorb_eta_prop}\mathsf{ACA}_0+\psi\vdash \forall M(M\models \mathsf{ACA}_0\land \eta\dot\in M \to M\models \eta).\end{equation} We need to show that \[\dwn(\mathsf{ACA}_0+(\chi\lor\psi))\vdash \lnot \chi \to\eta.\] Given that as indicated in (\ref{downarrow_agnostic}) the definition of downarrow is stable with respect to switches to larger countable fragments, it is sufficient to show that
\[\mathsf{ACA}_0+(\chi\lor \psi)\vdash \forall M(M\models \mathsf{ACA}_0\land \eta\dot\in M\land \chi\dot\in M \to M\models (\lnot \chi \to \eta)),\] which in turn is implied by 
\begin{equation}\label{downarrow_absorb_chi_prov}\mathsf{ACA}_0+\chi\vdash \forall M(M\models \mathsf{ACA}_0\land \eta\dot\in M\land \chi\dot\in M \to M\models (\lnot \chi \to \eta))\end{equation}
and 
\begin{equation}\label{downarrow_absorb_psi_prov}\mathsf{ACA}_0+\psi\vdash \forall M(M\models \mathsf{ACA}_0\land \eta\dot\in M\land \chi\dot\in M \to M\models (\lnot \chi \to \eta)).\end{equation}
Here (\ref{downarrow_absorb_chi_prov}) holds since $\chi$ is $\mathbf{\Pi}^1_1$ and hence $\mathsf{ACA}_0+\chi$ proves $\chi$ to hold in all $\omega$-models $M$ of $\mathsf{ACA}_0$ containing all the constants from $\chi$. And (\ref{downarrow_absorb_psi_prov}) holds as an immediate consequence of (\ref{downarrow_absorb_eta_prop}).\end{proof}

\begin{lemma}[Lemma \ref{lem:red_rule_dar_ref}] For c.e. extensions $T$ and $A$ of $\mathsf{ACA}_0$ in the language of $\mathsf{ACA}_0$:
    \begin{equation}\label{red_rule_dar_ref_stat}T+ \dwn\dmnd{1}{1} A \equiv T+\left\{\frac{\varphi}{\dmnd{1}{n}(\dwn A+\varphi)}\;\middle|\; n\in \mathbb{N}\right\}.\end{equation}
\end{lemma}
\begin{proof}
    It is enough to prove that 
    \begin{equation}\label{red_rule_dar_ref_simp}\mathsf{ACA}_0 + \dwn \dmnd{1}{1} A\equiv \mathsf{ACA}_0+\{\dmnd{1}{n}\dwn A\mid n\in\mathbb{N}\}.\end{equation} Indeed, the general equivalence (\ref{red_rule_dar_ref_stat}) by the Compactness Theorem follows from the case when $T$ is finitely axiomatizable. And in turn, the case of finitely axiomatizable $T$ follows from the case of $T=\mathsf{ACA}_0$ by the Deduction Theorem. Thus in order to finish the reduction of (\ref{red_rule_dar_ref_stat}) to (\ref{red_rule_dar_ref_simp}), we simply observe that $\mathsf{ACA}_0+\{\dmnd{1}{n}\dwn A\mid n\in\mathbb{N}\}$ is closed under the rules $\frac{\varphi}{\dmnd{1}{n}(\dwn A+\varphi)}$.

    Since the $(\supseteq)$-inclusion in (\ref{red_rule_dar_ref_simp}) holds by the arithmetical soundness of $\mathsf{RC_{BI}}$ (Proposition \ref{prop ari sound}, Axiom 14), in the rest of the proof we show the $(\subseteq)$-inclusion in (\ref{red_rule_dar_ref_simp}).

    Let us consider the theory $\mathsf{PA}(M)$ that is $\mathsf{PA}$ enhanced with an additional set constant $M$, the induction scheme for all formulas of the extended language, and the axiom stating that $M$ is a code of an $\omega$-model of $\mathsf{ACA}_0$. Let $\Pi^0_n(M)\mathsf{-RFN}(A)$ be the scheme
    $$\forall x(\mathsf{Prv}_{A}(\ulcorner\forall M(M\models \mathsf{ACA}_0\to \varphi(M,\dot x))\urcorner)\to \varphi(M,x)), \text{ where $\varphi(X,x)\in\Pi^0_n$}.$$

    We claim that for any $\varphi(X)\in\Pi^0_\infty$: \begin{equation}\label{Cons_Diamond^1_1_over_PA(M)}\dmnd{1}{1} A\vdash \forall M (M\models \mathsf{ACA}_0\to \varphi(M))\iff \mathsf{PA}(M)+\{\Pi^0_n(M)\mathsf{-RFN}(A)\mid n\in\mathbb{N}\}\vdash \varphi(M).\end{equation} Indeed, the implication $\Leftarrow$ is straightforward, since reasoning in $\dmnd{1}{1} A$ and given an $\omega$-model $M$ of $\mathsf{ACA}_0$ we clearly can prove any individual axiom of $\mathsf{PA}(M)+\{\Pi^0_n(M)\mathsf{-RFN}(A)\mid n\in\mathbb{N}\}$. For $\Rightarrow$ we use a model-theoretic argument: Given a model $$N\models\mathsf{PA}(M)+\{\Pi^0_n(M)\mathsf{-RFN}(A)\mid n\in\mathbb{N}\}+\lnot\varphi(M)$$ we expand it to a model $N_2$ of second-order arithmetic by letting the universe of sets consist of all the subsets of $N$ that are definable (with parameters) in the language of $\mathsf{PA}(M)$. Clearly, $N_2$ satisfies $\mathsf{ACA}_0$, $M\models\mathsf{ACA}_0$, and $\lnot \varphi(M)$, where $M$ is the set in $N_2$ defined by the formula $x\in M$.  Let us verify in $N_2$ an instance of $\Pi^1_1\mathsf{-RFN}(A)$ $$\forall x(\mathsf{Prv}_A(\ulcorner\forall X\; \psi(X,\dot x)\urcorner)\to \forall X\;\psi(X,x)) \text{, where $\psi(X,x)\in \Pi^0_\infty$}.$$ We immediately switch to an equivalent formula \begin{equation}\label{Pi^1_1-reflection_in_N_2}\forall X,x(\mathsf{Prv}_A(\ulcorner\forall X\; \psi(X,\dot x)\urcorner)\to \psi(X,x)).\end{equation}  Each set $X$ in $N_2$ is an $N$-definable set defined by some formula $\chi(M,y)$, where $\chi$ is a $\Pi^0_\infty$-formula (possibly with number parameters, which we suppress). Let $\psi'(M,x)$ be the result of replacement of formulas $t\in X$ in $\psi$ with $\chi(M,t)$. Hence in order to verify that $N_2$ satisfies (\ref{Pi^1_1-reflection_in_N_2}) for this particular $X$ it is enough to verify that $N$ satisfies
    \begin{equation}\label{Pi^1_1-reflection_in_N_1}\forall x (\mathsf{Prv}_A(\ulcorner\forall X\; \psi(X,\dot x)\urcorner)\to \psi'(M,x)).\end{equation}
    But since over $\mathsf{PA}$, $\mathsf{Prv}_A(\ulcorner\forall X\; \psi(X,\dot x)\urcorner)$ clearly implies $\mathsf{Prv}_A(\ulcorner\forall M\;(M\models \mathsf{ACA}_0\to \psi'(M,\dot x))\urcorner)$, we can derive (\ref{Pi^1_1-reflection_in_N_1}) from an instance of $\Pi^0_n(M)\mathsf{-RFN}(A)$, for suitably large $n$, which finishes the proof of the claim.

    Now, taking into account (\ref{Cons_Diamond^1_1_over_PA(M)}), in order to prove the $(\subseteq)$-inclusion in (\ref{red_rule_dar_ref_simp}) and hence finish the whole proof we simply need to show that
    \begin{equation}\label{Cons_PA(M)_over_ACA_with_ref}
        \mathsf{PA}(M)+\{\Pi^0_n(M)\mathsf{-RFN}(A)\mid n\in\mathbb{N}\}\vdash (M\models \varphi)\Rightarrow \mathsf{ACA}_0+\{\dmnd{1}{n} \dwn A\mid n\in\mathbb{N}\}\vdash \varphi\text{, where $\varphi\in\Pi^1_\infty$}.
    \end{equation}
    We achieve this by embedding the theory on the left into the theory on the right via a forcing interpretation that introduces a constant $M$ outside of the second-order universe by performing a L\'evy collapse of the second-order universe onto the naturals.

    We define in the language of $\mathsf{ACA}_0$ the forcing relation $P\Vdash \varphi(M,\vec{X},\vec{x})$, where $\varphi$ is a second-order formula with an additional predicate $y\in M$. Note that the meaning of $M$ is to be a countably coded $\omega$-model containing all sets from the range of second-order quantifiers and as such its intended meaning is to be outside of the range of the second-order quantifiers. The forcing conditions in the forcing formula are treated as codes for finite sequences of sets of natural numbers $P=(P_0,\ldots,P_{\mathsf{lh}(P)-1})$. We have the natural extension order on forcing conditions:
    $$(P_0,\ldots,P_{k-1})\succeq (P_0,\ldots,P_{k-1},P_k,\ldots,P_{s-1}).$$  
    We denote the above defined class ordering on forcing conditions as $\mathbb{P}$. 
    The forcing relation is defined by induction on the construction of the formula $\varphi$, where we consider $\varphi$ constructed from atomic formulas, using the connectives $\land,\lnot,\to,\forall x,\forall X$:
    \begin{enumerate}
        \item $P\Vdash \varphi$ is $\varphi$ if $\varphi$ is atomic without $M$;
        \item $P\Vdash t\in M$ is $t\in\{(x,y)\mid y\in P_x \text{ and }x<\mathsf{lh}(P)\}$;
        \item $P\Vdash \varphi \land \psi$ is $(P\Vdash \varphi)\land (P\Vdash \psi)$;
        \item $P\Vdash \forall x\; \varphi $ is $\forall x\;(P\Vdash \varphi)$;
        \item $P\Vdash \forall X\; \varphi $ is $\forall X\;(P\Vdash \varphi)$;
        \item $P\Vdash \lnot \varphi$ is $\forall P'\preceq P(P'\nVdash \varphi)$;
        \item $P\Vdash \varphi \to \psi$ is $\forall P'\preceq P((P'\Vdash \varphi)\to (P'\Vdash \psi))$.
    \end{enumerate}
    We extend the forcing relation to formulas with $\lor$ and $\exists$ by expressing them using negations, e.g.\ $P\Vdash \varphi\lor \psi$ translates to $P\Vdash \lnot (\lnot\varphi \land \lnot \psi)$. Note that $\mathsf{ACA}_0$-provably 
    \begin{enumerate}
        \item $P\Vdash \varphi\lor \psi$ is equivalent to $\forall P'\preceq P\exists P''\preceq P' ((P''\Vdash \varphi)\lor (P'' \Vdash \psi))$;
        \item $P\Vdash \exists x\; \varphi$ is equivalent to $\forall P'\preceq P\exists P''\preceq P' \exists x (P''\Vdash \varphi)$;
        \item $P\Vdash \exists X\; \varphi$ is equivalent to $\forall P'\preceq P\exists P''\preceq P' \exists X (P''\Vdash \varphi)$.
    \end{enumerate}
    
    Obviously, for $M$-free formulas $\varphi$, $$\mathsf{ACA}_0\vdash (P\Vdash \varphi)\mathrel{\leftrightarrow} \varphi.$$
    
    Notice that in fact for every model $N$ of $\mathsf{ACA}_0$, the forcing relation defines validity of formulas in the constant-domain first-order Kripke structure whose worlds are indexed by $N$-forcing conditions and the structure in each node $P$ is a copy of $N$ enhanced by a predicate $x\in M$ such that $a\in M$ iff $a$ is in the $N$-set $\{(x,y)\mid y\in P_x \text{ and }x<\mathsf{lh}(P)\}$. Hence the forcing relation in any model of $\mathsf{ACA}_0$ is deductively closed and satisfies intuitionistic logic (in the language without $\lor$ and $\exists$). It is easy to see that the predicate $x\in M$ is equivalent to its double negation in this model. And since we are only interested in formulas from the negative fragment, our forcing relation in fact satisfies classical first-order logic. 

    It is easy to see that $$\mathsf{ACA}_0\vdash (P\Vdash \forall X\exists x\forall y(y\in X\mathrel{\leftrightarrow} (x,y)\in M)).$$
    Indeed, we need to check in $\mathsf{ACA}_0$ that for any $P$, $X$, and $P'\preceq P$ we can pick $P''$ and $x$ such that for any $y$, $y\in X$ iff $P''\Vdash (x,y)\in M$. Clearly, we can pick $P''$ to be $(P'_0,\ldots,P'_{k-1},X)$ and $x=k$, where $P'=(P'_0,\ldots,P'_{k-1})$.

    Also it is easy to see that
    $$\mathsf{ACA}_0\vdash (P\Vdash \forall x\exists X\forall y(y\in X\mathrel{\leftrightarrow} (x,y)\in M)).$$
    Here given $P$, $x$ and $P'\preceq P$ we pick as $P''$ the extension of $P'$ by some sets so that $\mathsf{lh}(P'')>x$ and we pick as $X$ the $x$-th component of $P''$.

    Combining the facts that we established we see that from the perspective of the forcing relation, the sections of $M$ precisely code all the sets of the second-order universe. In particular $\mathsf{ACA}_0$ verifies that $P\Vdash (M\models \mathsf{ACA}_0)$, for every $P$.

    Next we observe that provably in $\mathsf{ACA}$ (that is, $\mathsf{ACA}_0$ with the full induction scheme; note that $\mathsf{ACA}\subseteq \mathsf{ACA}_0+\{\dmnd{1}{n}\dwn A\mid n\in\mathbb{N}\}$, since the latter theory proves $\dmnd{1}{n}\top$ for all $n$), all instances of the induction scheme in the language of $\mathsf{PA}(M)$ are forced: For this it is clearly enough to show that for a fixed formula $\varphi(x)$, provably in $\mathsf{ACA}$, for every $a$ we have $\forall P(P\Vdash (\varphi(0)\land \forall x (\varphi(x)\to\varphi(x+1))\to \varphi(a)))$. We prove this by a straightforward induction on $a$.  

    Finally let us work in $\mathsf{ACA}_0+\{\dmnd{1}{n}\dwn A\mid n\in\mathbb{N}\}$ to establish forceability of arbitrarily picked instances of $\Pi^0_\infty(M)\mathsf{-RFN}(A)$. We need to show that
    $$\forall P(P\Vdash \forall x(\mathsf{Prv}_{A}(\ulcorner\forall M(M\models \mathsf{ACA}_0\to \varphi(M,\dot x))\urcorner)\to \varphi(M,x))), \text{ where $\varphi(X,x)\in\Pi^0_n$}.$$
    This is equivalent to
    $$\forall x(\mathsf{Prv}_{A}(\ulcorner\forall M(M\models \mathsf{ACA}_0\to \varphi(M,\dot x))\urcorner)\to \forall P(P\Vdash \varphi(M,x))).$$
    Given that we have access to the full reflection scheme for $\dwn A$, it will be sufficient to show that if $A$ proves $\forall M(M\models \mathsf{ACA}_0\to \varphi(M,\underline{k}))$, then $\dwn A$ proves $\forall P(P\Vdash \varphi(M,\underline{k}))$. And for this it is enough to show that if $A$ proves $\forall M(M\models \mathsf{ACA}_0\to \varphi(M,\underline{k}))$, then $A$ proves $\forall M'((M'\models \mathsf{ACA}_0)\to (M'\models \forall P(P\Vdash \varphi(M,\underline{k}))))$. That is, it is enough to prove the following claim:
    \begin{equation}\label{forcing_in_out_claim}\begin{aligned}
     &\text{Given a $\Pi^0_\infty$ formula $\psi(X)$, the theory $A$ proves that if $\psi(M')$ holds for each $\omega$-model $M'$ of $\mathsf{ACA}_0$,}&\\&\text{then each $\omega$-model $M'$ of $\mathsf{ACA}_0$ satisfies $\forall P(P\Vdash \psi(M))$.}&   
    \end{aligned}\end{equation}

    To establish the claim (\ref{forcing_in_out_claim}) above we fix a formula $\psi(X)$ and reason in $\mathsf{ACA}_0+\forall M'(M'\models \mathsf{ACA}_0\to \psi(M'))$ to show that for any $\omega$-model $M'$ of $\mathsf{ACA}_0$ we have $M'\models (\forall P(P\Vdash \psi(M)))$. We will be using a characterization of the forceability in terms of generic filters. 
    
    For $m\ge 1$, we call a filter $G$ on $\mathbb{P}^{M'}$ an $m$-generic filter if it intersects any dense subset of $\mathbb{P}^{M'}$ that is $M'$-definable by a $\Pi^1_m$-formula with parameters. In the present proof we will be interested in $m$-generic filters for a particular finite collection of $m$'s that depend only on the formula $\psi$; note that for a fixed $m$ we can define the notion of $m$-generic filters, while working in $\mathsf{ACA}_0$, without any problems, since we can use arithmetical comprehension to construct a single set-code for the whole family of sets that are $M'$-definable by $\Pi^1_m$-formulas. Using this enumeration we can also show within $\mathsf{ACA}_0$ that any $P\in \mathbb{P}^{M'}$ lies in some $m$-generic filter $G$.
    
    Trivially all $m$-generic filters are maximal filters. So we identify generic filters $G$ with infinite sequences of $M'$-sets $(P_0,P_1,\ldots)$ such that the elements of the filter are precisely the finite initial segments of this sequence. Also it is trivial to observe that any $m$-generic $G$ will be an enumeration of all $M'$-sets. For such a $G$ we denote by $(M',G)$ the expansion of $M'$ in which the predicate $M$ is interpreted by the set $\{(x,y)\mid y\in P_x\}$.
    
    For a suitable fixed natural number $m_\psi$, via an external induction on subformulas $\theta$ of $\psi$ we show that for any $\theta'$ that is the result of replacing all free variables of $\theta$ with parameters from $M'$ and any $m_\psi$-generic filter $G$:
    \begin{equation}\label{forcing_truth_lemma}(M',G)\models \theta' \iff \exists P\in G\, (M'\models(P\Vdash \theta')).\end{equation}
    Before starting the proof we note that this is just a variant of the usual truth lemma for forcing, and that the value $m_\psi$ can be reconstructed from the complexities of the $M'$-definitions of the dense sets used in the proof. 
    
    The case when $\theta'$ is an atomic formula not using $M$ is trivial, since all forcing conditions $P$ agree with each other and with $(M',G)$ about the validity of such formulas. 
    When $\theta'$ is $t\in M$ and the value of $t$ is a pair $(k,s)$, we pick $P$ from $G$ intersected with the dense set $\{P\in \mathbb{P}^{M'}\mid \mathsf{lh}(P)> k\}$; clearly if $M'\models (P\Vdash \theta')$ then $(M',G)\models \theta'$, and if $M'\models (P\nVdash \theta')$, then $s\not\in P_k=G_k$ and hence $(M',G)\not\models \theta'$. If the value of $t$ is not a pair, then no condition forces $\theta'$ and $(M',G)\not\models\theta'$. 
    
    The step of induction for the case when $\theta'$ is $\theta_1\land \theta_2$ is trivial. 
    
    The step of induction for the case when $\theta'$ is $\lnot \theta_1$ can be reduced to proving that for any $m_\psi$-generic $G$, \begin{equation} \label{decide_forcing} \exists P\in G\, (M'\models (P\Vdash \theta'))\iff \lnot \exists P\in G\, (M'\models (P\Vdash \theta_1)).\end{equation} To prove the latter equivalence we notice that the set $K=\{P\in \mathbb{P}^{M'}\mid M'\models (P\Vdash \theta_1)\}\cup \{P\in \mathbb{P}^{M'}\mid M'\models(P\Vdash \theta')\}$ is dense, since $M'\models (()\Vdash \theta'\lor \theta_1)$, the latter being an instance of the law of excluded middle. The density of this set allows us to consider $P_0\in G\cap K$, and for this $P_0$ we have two options, $M'\models (P_0\Vdash \theta')$ or $M'\models (P_0\Vdash \theta_1)$, with the first option implying that $\exists P\in G\, (M'\models (P\Vdash \theta'))$ and the second option implying that $\exists P\in G\, (M'\models (P\Vdash \theta_1))$. This concludes the proof of (\ref{decide_forcing}) since we cannot simultaneously have $\exists P\in G\, (M'\models (P\Vdash \theta'))$ and $\exists P\in G\, (M'\models (P\Vdash \theta_1))$: as any two elements of $G$ are compatible, this would imply the existence of $P\in G$ for which $M'\models (P\Vdash \theta_1\land \theta')$.

    The step of induction for the case when $\theta'$ is $\theta_1\to\theta_2$ reduces to the cases of $\lnot$ and $\land$, since $\mathsf{ACA}_0$ proves that $P\Vdash \theta_1\to\theta_2$ is equivalent to $P\Vdash \lnot(\theta_1\land\lnot\theta_2)$ (this uses the stability of forcing under double negation).

    Now let us prove the step of induction for the case when $\theta'$ is $\forall x\;\theta_1(x)$. Here we consider the set \[K=\{P\in\mathbb{P}^{M'}\mid M'\models (P\Vdash \theta')\}\cup\{P\in\mathbb{P}^{M'}\mid \exists n\in\mathbb{N}\; M'\models (P\Vdash \lnot \theta_1(n))\}.\]
    It is dense, because if for some $P\in\mathbb{P}^{M'}$ there were no $P'\preceq P$ from $K$, then on one hand $M'\models (P\Vdash \lnot \theta')$, but on the other hand for all $n\in\mathbb{N}$, $M'\models (P\Vdash \lnot\lnot\theta_1(n))$, which implies $M'\models (P\Vdash \forall x \lnot\lnot \theta_1(x))$ and hence, by the stability of forcing under double negation, $M'\models (P\Vdash \theta')$, leading to a contradiction. Taking $P_0\in K\cap G$ we reason as follows to prove (\ref{forcing_truth_lemma}). If $M'\models (P_0\Vdash \lnot\theta_1(n))$, for some $n$, then on one hand by the induction hypothesis combined with the argument from the $\lnot$-case we get that $(M',G)\models \lnot \theta_1(n)$ and hence $(M',G)\not \models \theta'$, and on the other hand we get that there is no $P\in G$ for which $M'\models (P\Vdash \theta')$, since such a $P$ would be compatible with $P_0$. If  $M'\models (P_0\Vdash \theta')$ then clearly $\exists P\in G\, ( M'\models (P\Vdash \theta'))$ and $M'\models (P_0\Vdash \theta_1(n))$ for all individual $n$, which by the induction hypothesis implies that $(M',G)\models \theta'$.

    The case of $\theta'$ being $\forall X\; \theta_1(X)$ is analogous to the case of a leading universal first-order quantifier, but instead of instantiations in $\mathbb{N}$ we consider instantiations with all the sets $(M')_i$ of $M'$. This concludes the inductive proof of (\ref{forcing_truth_lemma}).

    Using (\ref{forcing_truth_lemma}), the fact that every $P\in \mathbb{P}^{M'}$ is contained in some $m_\psi$-generic $G$, and the density of the sets $K$ from the $\lnot$-case, we get for every $P\in\mathbb{P}^{M'}$:
    \begin{center}$M'\models(P\Vdash \theta' )\iff (M',G)\models \theta'$ for every $m_\psi$-generic $G$ that contains $P$.\end{center}    
     In light of this, to prove (\ref{forcing_in_out_claim}) we simply need to establish that $(M',G)\models \psi(M)$, for any $m_{\psi}$-generic $G$. Given that the second-order part of $M'$ is not involved here at all, this is the same as showing that we have $\psi(G)$, for any $m_{\psi}$-generic $G$, which we do have, since $G$ is an alternative enumeration of $M'$ and as such is an $\omega$-model of $\mathsf{ACA}_0$ and we are reasoning in the theory that has an axiom $\forall M'((M'\models \mathsf{ACA}_0)\to \psi(M'))$. This finishes the proof of the claim (\ref{forcing_in_out_claim}) and hence the whole lemma.
\end{proof}
The proof of Lemma \ref{lem:red_rule_dar_bas} is essentially a simpler version of the proof of Lemma \ref{lem:red_rule_dar_ref}. The only difference is that for Lemma \ref{lem:red_rule_dar_bas} we do not need to deal with reflection axioms, since the lemma can be equivalently reformulated as the equivalence $\dwn \mathsf{ACA}_0\equiv \mathsf{ACA}$. Now instead of considering the extension of $\mathsf{PA}(M)$ by an appropriate reflection principle we consider just $\mathsf{PA}(M)$ itself. In the part embedding $\mathsf{PA}(M)$ into $\mathsf{ACA}$ via a forcing interpretation, we now do not need to prove forceability of reflection principles and simply get away with forceability of the full induction scheme.

\vspace{\baselineskip}
\section{Ordinal notation system for the Bachmann-Howard ordinal}\label{section ONS}
With all the technical groundwork established, we now turn to the ordinal analysis consequences of the calculus. We claim that the $\nabla$-free closed $\mathsf{RC_{BI}}$ formulas naturally induce an ordinal notation system. Specifically, our main result of this section shows that the $\Pi^1_1$ proof-theoretic ordinal of a theory at the $\alpha$-th position in the $\leq^1_1$ ordering is exactly $\varepsilon_\alpha$.

\begin{lemma}\label{lemma finite theory}
For every formula $A\in\W$ there is a formula $B\in\W$ such that $\dmnd{2}{2} A$ is $\mathsf{RC}_{\mathsf{BI}}$-equivalent to $\dmnd{2}{2} B$ and $B^*$ is a finitely axiomatizable theory.\end{lemma}

\begin{proof}
We first observe that $B^*$ is finitely axiomatizable whenever $B$ is such that each of its conjuncts is either
$\top$ or of the form $\dmnd{i}{\alpha} A'$: $B^*$ is then the union of $\mathsf{ACA}_0$ with finitely many theories of the form $\mathsf{ACA}_0 + \Pi^1_n\text{-}\mathsf{RFN}(T)$ or $\mathsf{ACA}_0 + \Pi^1_n\text{-}\omega\mathsf{RFN}(T)$, each of which is finitely axiomatizable. Hence, given any formula $A$, it suffices to construct a formula $B$ in this form with
$\dmnd{2}{2} A \dashv\vdash_{\mathsf{BI}} \dmnd{2}{2} B$.

 Suppose we have a formula $C\wedge D$ such that $D$ is of the form $\nabla A'$ or ${\downarrow} A'$. We are going to show that:
\begin{equation}\label{eq finite theory}
    \dmnd{2}{2}(D \wedge C) \;\dashv\vdash_{\mathsf{BI_\omega}}\;\dmnd{2}{2}(\dmnd{1}{1} A' \wedge C).
\end{equation}
For the left-to-right direction, by using Axioms~2 and~1, Mod$_1$ (together with Axiom~5 when $D={\downarrow}A'$) and Axiom~9 we get
$$\dmnd{2}{2}(D \wedge C) \vdash \dmnd{1}{1}(D \wedge C)
    \vdash \dmnd{1}{1} \nabla A' \vdash \dmnd{1}{1} A'.$$
Since $\dmnd{2}{2}(D \wedge C) \vdash \dmnd{2}{2} C$, combining the two and applying Axiom~12 yields
$$\dmnd{2}{2}(D \wedge C) \vdash \dmnd{2}{2} C \wedge \dmnd{1}{1} A'
    \vdash \dmnd{2}{2}(C \wedge \dmnd{1}{1} A').$$

\noindent For the converse direction we have two cases depending on the form of $D$. If $D = \nabla A'$, since $\dmnd{1}{1} A' \wedge C \vdash \nabla A' \wedge C$ we easily get
$\dmnd{2}{2}(\dmnd{1}{1} A' \wedge C) \vdash \dmnd{2}{2}(\nabla A' \wedge C)$.
For the second case $D = {\downarrow} A'$, we use the following derivation:
$$\dmnd{2}{2}(\dmnd{1}{1} A' \wedge C)
    \vdash \dmnd{2}{2}\dmnd{1}{1} A'\vdash \dmnd{1}{1}\dmnd{1}{1} A' \vdash \dmnd{1}{1} A'.$$
This together with $\dmnd{2}{2}(\dmnd{1}{1} A' \wedge C) \vdash \dmnd{2}{2} C$ yields
$$\dmnd{2}{2}(\dmnd{1}{1} A' \wedge C)\vdash \dmnd{2}{2} C \wedge \dmnd{1}{1} A'\vdash \dmnd{2}{2}(C \wedge {\downarrow} A').$$

Each application of (\ref{eq finite theory}) decreases the number of conjuncts of the form $\nabla A'$ or ${\downarrow} A'$, so after finitely many replacements we arrive at a formula $B$ in the desired form with $\dmnd{2}{2} A \dashv\vdash_{\mathsf{BI}} \dmnd{2}{2} B$.
\end{proof}

\begin{lemma}\label{lemma order 1}
    Provably in $\mathsf{ACA}_0$, for every formula $A\in\W$,
    \[A^*\equiv_{\Pi^1_1} \mathsf{ACA}_0+\{\Pi^1_1\textsf{-RFN}(B^*)\mid B<^1_1A\}.\]
\end{lemma}
\begin{proof}
    The claim follows directly from the results of the preceding sections. The $(\subseteq)$-inclusion is an immediate consequence of the reduction properties established in Lemmas \ref{lem 1 red 1}--\ref{lem 5 red 5} (note that to meet the finite axiomatizability condition of Lemma \ref{lem 5 red 5} we can always use Lemma \ref{lemma finite theory}). The converse $(\supseteq)$-inclusion follows immediately from the definition of the $\leq^1_1$-ordering, $A<^1_1 B$ iff $B\vdash_\mathsf{BI_\omega}\dmnd{1}{1} A$, due to the arithmetical soundness of the calculus (Proposition \ref{prop ari sound}).
\end{proof}

\begin{definition}\label{iter_ref_def}
    For every formula $A\in\W$ we define theories $\Pi^1_1\textsf{-R}^A(\mathsf{ACA}_0)$ to be
    \[\Pi^1_1\textsf{-R}^A(\mathsf{ACA}_0)=\mathsf{ACA}_0+\{\Pi^1_1\textsf{-RFN}(\Pi^1_1\textsf{-R}^B(\mathsf{ACA}_0))\mid B<^1_1 A\}.\]
    Formally, we use G\"odel's Diagonal Lemma to define a $\Sigma_1$-formula $\mathsf{Ax}(x,y)$ expressing that $x$ is a formula $A\in\W$ and $y$ is an axiom of $\Pi^1_1\textsf{-R}^A(\mathsf{ACA}_0)$.
\end{definition}

\begin{theorem}\label{thm iter ref}
    Provably in $\mathsf{ACA}_0$,
    \begin{equation} \label{RC_it_ref} A^*\equiv_{\Pi^1_1} \Pi^1_1\textsf{-R}^{A}(\mathsf{ACA}_0)\text{, for every formula $A\in\W$.}\end{equation}
\end{theorem}
\begin{proof}
We prove the theorem via L\"ob's Theorem. Thus in order to conclude the proof we reason in the theory $\mathsf{ACA}_0$ plus $\mathsf{ACA}_0$-provability of (\ref{RC_it_ref})  to prove (\ref{RC_it_ref}). 

Indeed, we consider some formula $A\in\W$ and observe that  
\[\begin{aligned} A^*& \equiv &\mathsf{ACA}_0+\{\Pi^1_1\textsf{-RFN}(B^*)\mid B<^1_1A\} & \text{\;\;\;\;\;(by Lemma \ref{lemma order 1})} \\
& \equiv & \mathsf{ACA}_0+\{\Pi^1_1\textsf{-RFN}(\Pi^1_1\textsf{-R}^B(\mathsf{ACA}_0))\mid B<^1_1 A\} &\text{\;\;\;\;\;(by $\mathsf{ACA}_0$-provability of (\ref{RC_it_ref}))}\\
& \equiv & \Pi^1_1\textsf{-R}^{A}(\mathsf{ACA}_0) & \text{\;\;\;\;\;(by Definition \ref{iter_ref_def}).}\end{aligned}\]
\end{proof}

For a formula $A\in\W$, let $o(A)$ denote the ordinal $\alpha$ such that the $\leq^1_1$-equivalence class of $A$ occupies the $\alpha$-th position in the ordering; this is well defined, since $\leq^1_1$ is a linear and well-founded (Theorem \ref{thm lin}) pre-order.

\begin{definition}
For $T\supseteq \mathsf{ACA}_0$ its $\Pi^1_1$ proof-theoretic ordinal $|T|_{\Pi^1_1}$ is the supremum of the order types of recursive well-orderings\footnote{Strictly speaking, in this definition $\prec$ ranges over indices of Turing machines deciding well-ordered binary relations on subsets of naturals.} $\prec$ such that $T$ proves that $\prec$ is indeed a recursive binary relation and that $\prec$ is a well-ordering.
\end{definition}

Theorem \ref{thm iter ref} together with \cite[Theorem 3.2]{pakhomov2021reflection} gives us the following corollary.

\begin{corollary}\label{proof_theoretic_ordinals_cor}
    For every formula $A\in\W$,
    \[|A^*|_{\Pi^1_1}=\varepsilon_{o(A)}.\]
\end{corollary}

Being a computable pre-well-ordering, the order $<^1_1$ on $\nabla$-free closed $\mathsf{RC_{BI}}$ formulas can be considered an ordinal notation system. Namely, each element $A$ denotes the ordinal $o(A)$, which is the order type of the initial segment $\{B \mid B <^1_1 A\}$ modulo provable equivalence. In this paper, we do not perform any direct computations of the correspondence between the commonly used ordinal notation systems for the Bachmann-Howard ordinal and $<^1_1$. However, Corollary \ref{proof_theoretic_ordinals_cor} in fact provides such a correspondence when we restrict our attention to formulas $A$ such that $A^*$ is a theory whose proof-theoretic ordinal has already been computed in terms of standard ordinal notation systems. In particular, the following holds:
\begin{enumerate}
    \item $o(\top)=0$;
    \item $o(\dmnd{1}{1} \top)=1$;
    \item $o(\dmnd{1}{2} \top)=\omega$, note that $(\dmnd{1}{2}\top)^*$ is $\mathsf{ACA}_0'$ \cite[Theorem~15]{PachecoYokoyama2023};
    \item $o(\dwn \top)=\varepsilon_0$, note that $(\dwn \top)^*=\mathsf{ACA}$ by Lemma \ref{lem 3 red 3};
    \item $o(\dmnd{2}{2} \top)=\varphi_2(0)$, since trivially $(\dmnd{2}{2} \top)^*=\mathsf{ACA}_0^+$;
    \item $o(\dmnd{2}{3} \top)=\varphi_\omega(0)$, since $(\dmnd{2}{3} \top)^*=\Pi^1_1\textsf{-BI}_0=\Sigma^1_1\textsf{-DC}_0$ \cite[Theorem VIII.5.12]{simpson2009subsystems}, and from \cite{cantini1986relation} we have $|\Sigma^1_1\textsf{-DC}_0|_{\Pi^1_1}=\varphi_\omega(0)$;
    \item $o(\dmnd{2}{2}\dmnd{2}{3} \top)=\Gamma_0$, since $(\dmnd{2}{2}\dmnd{2}{3} \top)^*=\mathsf{ATR}_0$ \cite[Theorem 6.7]{pakhomov2023reducing} and the proof-theoretic ordinal of $\mathsf{ATR}_0$ is well-known to be $\Gamma_0$ \cite{friedman1982finite};
    \item $o(\dmnd{2}{2}\dmnd{2}{3} \top\land \dwn\top)=\Gamma_{\varepsilon_0}$, since $(\dmnd{2}{2}\dmnd{2}{3} \top\land \dwn\top)^*=\mathsf{ATR}$;
    \item $\sup_{A\in \W} o(A)=\psi(\varepsilon_{\Omega+1})$, since by  \cite{fried} $\mathsf{ACA}_0+\bigcup\limits_{n<\omega} \Pi^1_n\textsf{-}\omega\textsf{RFN} (\mathsf{ACA}_0)\equiv \mathsf{BI}$ and it is a folklore result that the proof-theoretic ordinal of $\mathsf{BI}$ is the Bachmann-Howard ordinal $\psi(\varepsilon_{\Omega+1})$.\footnote{This, for example, can be concluded from the result that $\mathsf{BI}$ has the same $\Pi^1_1$ theorems as $\mathsf{ACA}_0+\Pi^1_1\textsf{-CA}^-$ \cite[Corollary~VII.2.21]{simpson2009subsystems}, and the computation of the proof-theoretic ordinal of $\mathsf{ACA}_0+\Pi^1_1\textsf{-CA}^-$ (see \cite{freund2022impredicativity}, where the computation is done via the reduction to the computation of the proof-theoretic ordinal of $\mathsf{ID}_1$).}
\end{enumerate}

\begin{remark}
A computation of the proof-theoretic ordinals of all the theories of the form $(\dmnd{2}{n_1}\ldots\dmnd{2}{n_k}\top)^*$ can be found in \cite{probst2015habilitation}. The method employed by Probst in that work is based on cut-elimination for infinitary logic.
\end{remark}

\vspace{\baselineskip}
\section{Conclusion and Future Work}
There are several main directions of further work that naturally follow from here. 

Let us start with possible further research in the direction of strictly positive modal logics. First, it would be natural to put our result into the perspective of a broader theory exploring under which conditions strictly positive modal logics similar to ours will have a well-behaved Kripke semantics. In particular, it would be interesting to do a more systematic investigation of strictly positive modal logics with axioms corresponding to the generation of additional worlds in terms of Kripke semantics.\footnote{We note that $\mathsf{RC}_{\mathsf{BI}}$ has such axioms, namely the ones denoted as constructive axioms: 13. $\dwn\top \vdash \dmnd{1}{\alpha} \top$, 14. $\dwn\dmnd{1}{1} A\vdash \dmnd{1}{\alpha}\dwn A$, and 15. $A\wedge \dmnd{2}{2} B\vdash\nabla(\dwn A\wedge B)$.} A deeper understanding of the behavior of modal frames in this setting may simplify the approach and enable generalizations to related systems.

Another direction involves developing and studying a reflection calculus that incorporates stronger principles, such as $\beta$-model reflection. Extending the calculus and developing the appropriate reduction properties in this way could facilitate a Beklemishev-style ordinal analysis of larger fragments of second-order arithmetic, and potentially yield new ordinal notation systems. Finally, we find a further investigation of the phenomenon of connections between theories axiomatized by reflection principles and the linearity of the consistency strength to also be an intriguing research direction.

\vspace{\baselineskip}
\section*{Use of AI}
 There was some limited use of AI in the final stages of the preparation of this paper. When we already had a manuscript with full proofs, AI was used for proofreading, and it found a number of mistakes in the draft. The vast majority of these were relatively minor, and those that were not entirely trivial to fix were fixed manually by the authors. The only relatively major gap was in the proof of Lemma \ref{lem 12 tree approx 1}. Fixing it required a new framework for the analysis of the models $M_A$, built around the notion of anchor points (Definition \ref{def 12 trees 3}). This is the only part of the proofs in the paper where AI assistance was used in proof design in a substantive way. The idea of introducing anchor points, together with the accompanying framework of paths, transition ranks and connectors, is the result of interaction with AI about possible fixes of this gap. In our opinion, both the human and the AI input contributed to the end result. As a consequence, the proofs of Lemmas \ref{lem rank2 paths}, \ref{lem closure facts} and \ref{lem locality}, Proposition \ref{prop MA is a model}, Lemma \ref{lem 11 trees 3}, and the results of Subsection \ref{subsec:structural MA}, which had previously been proved in a different form, as well as the now fixed Lemma \ref{lem 12 tree approx 1} were adapted to the new framework. AI was used to help rewrite them, and the authors subsequently revised the results.

In addition, AI was used in text editing, to polish the paper with grammar fixes, stylistic adjustments such as better notations, the correction of typos and other trivial mistakes, and presentational improvements to some proofs that leave the mathematical content unchanged. Finally, AI was used in the preparation of the introduction, to help survey the existing literature related to the present paper. The authors have checked all proofs and references and take full responsibility for the content of the paper.

\vspace{\baselineskip}
\section*{Acknowledgements} 
The authors would like to thank Lev Beklemishev, David Fernández-Duque, and
 Joost J. Joosten for fruitful discussions. The work of both the authors has been funded by the FWO grant G0F8421N.

\vspace{\baselineskip}
\bibliographystyle{acm}
\bibliography{citation}
\nocite{*}
\end{document}